\documentclass[10pt]{article}
\pdfoutput=1

\usepackage{url}
\usepackage{mathtools}
\usepackage{amssymb}
\usepackage{amsthm}
\usepackage{empheq}
\usepackage{latexsym}
\usepackage{enumitem}
\usepackage{eurosym}
\usepackage{dsfont}
\usepackage{appendix}
\usepackage{color} 
\usepackage[unicode]{hyperref}
\usepackage{frcursive}
\usepackage[utf8]{inputenc}
\usepackage[T1]{fontenc}
\usepackage{geometry}
\usepackage{multirow}
\usepackage{todonotes}
\usepackage{lmodern}
\usepackage{anyfontsize}
\usepackage{stmaryrd}
\usepackage{natbib}
\usepackage{cleveref}
\usepackage{comment}

\usepackage{tcolorbox}
\setcitestyle{numbers,open={[},close={]}}

\definecolor{red}{rgb}{0.7,0.15,0.15}
\definecolor{green}{rgb}{0,0.5,0}
\definecolor{blue}{rgb}{0,0,0.7}
\hypersetup{colorlinks, linkcolor={red},citecolor={green}, urlcolor={blue}}
			
\makeatletter \@addtoreset{equation}{section}

\newtheorem{theorem}{Theorem}[section]
\newtheorem{assumption}[theorem]{Assumption}
\newtheorem{corollary}[theorem]{Corollary}

\newtheorem{lemma}[theorem]{Lemma}
\newtheorem{proposition}[theorem]{Proposition}

\newtheorem{definition}[theorem]{Definition}
\newtheorem{remark}[theorem]{Remark}

\DeclareUnicodeCharacter{014D}{\=o}
\def \E{\mathbb{E}}
\def \F{\mathbb{F}}

\def \L{\mathbb{L}}

\def \N{\mathbb{N}}
\def \O{\mathbb{O}}
\def \P{\mathbb{P}}
\def \Q{\mathbb{Q}}
\def \R{\mathbb{R}}

\def \X{\mathbb{X}}
\def \Y{\mathbb{Y}}

\def\Ac{{\cal A}}
\def\Bc{{\cal B}}
\def\Cc{{\cal C}}

\def\Fc{{\cal F}}
\def\Gc{{\cal G}}

\def\Mc{{\cal M}}

\def\Pc{{\cal P}}

\def\Sc{{\cal S}}
\def\Tc{{\cal T}}

\def\Wc{{\cal W}}
\def\Xc{{\cal X}}
\def\Yc{{\cal Y}}
\def\Zc{{\cal Z}}

\def\Yb{{\bar Y}}

\def\O{{\Omega}}
\def\o{{\omega}}
\def\cplc{{\rm Cpl_c}}
\def\cplac{{\rm Cpl_{ac}}}
\def\cplbc{{\rm Cpl_{bc}}}
\def\cplmc{{\rm Cpl_{mc}}}
\def\cpl{{\rm Cpl}}

\def \tY{\widetilde{Y}}

\usepackage{mathrsfs}
\usepackage{graphicx}

\usepackage{amsmath}
\usepackage{bbm}

\author{Daniel {\sc Kr\v{s}ek} \footnote{ETH Z\"urich, Department of Mathematics, Zurich, Switzerland, daniel.krsek@math.ethz.ch, ORCID: 0009-0001-6996-4874} \and Gudmund {\sc Pammer} \footnote{ETH Z\"urich, Department of Mathematics, Zurich, Switzerland, gudmund.pammer@math.ethz.ch, ORCID: 0000-0003-2494-8739}}

\title{General duality and dual attainment for adapted transport}

\date{\today}	

\begin{document}
\maketitle

\begin{abstract}
We investigate duality and existence of dual optimizers for several adapted optimal transport problems under minimal assumptions. 
This includes the causal and bicausal transport, the causal and bicausal barycenter problem, and a multimarginal problem incorporating causality constraints. Moreover, we characterize polar sets in the causal and bicausal setting and discuss applications of our results in robust finance. We consider a non-dominated model of several financial markets where stocks are traded dynamically, but the joint stock dynamics are unknown. We show that a no-arbitrage assumption naturally leads to sets of multicausal couplings. Consequently, computing the robust superhedging price is equivalent to solving an adapted transport problem, and finding a superhedging strategy means solving the corresponding dual.
\vspace{0.5cm}

\end{abstract}

\section{Introduction} \label{sec:intro}

\medskip Being an integral part of modern optimal transport, the duality theory traces its roots back to the 1940s with the seminal work of \citeauthor*{ka42} \cite{ka42}, subsequently extended by \citeauthor*{KaRu58} \cite{KaRu58}, \citeauthor*{Du76} \cite{Du76}, \citeauthor*{Ke84} \cite{Ke84}, \citeauthor*{KnSm84} \cite{KnSm84}, \citeauthor*{Br87} \cite{Br87,Br91}, \citeauthor*{GaMc96} \cite{GaMc96}, and many others. It has gained the interest of the mathematical community and evolved into a well-studied tool. Its numerous applications include the study of quantitative properties of optimal transport plans, stability of optimal transport problems, and it finds use in fields such as statistics, computer science, machine learning, and image processing. In this work, we study the dual problems of several versions of adapted optimal transport problems and provide duality results and show dual attainment in a very general setting. We further discuss applications of our results in robust hedging and provide additional intuition using the dynamic programming approach.

\medskip
Starting with the seminal articles of \citeauthor*{Ly95} \cite{Ly95}, \citeauthor*{AvLePa95} \cite{AvLePa95} and \citeauthor*{SoToZh11} \cite{SoToZh11} on markets with uncertain volatilities, \citeauthor*{Ho98a} \cite{Ho98a} and \citeauthor*{Co06} \cite{Co06} on robust hedging and pricing,
robust finance has emerged as an important counterpart complementing the classical theory of mathematical finance, and significantly improved our understanding of model uncertainty. 
Due to the vast number of contributions to the field in recent years, we can not do justice to all, but refer readers to 
\citeauthor*{GaLaTo} \cite{GaLaTo}, \citeauthor*{BeHePe12} \cite{BeHePe12}, \citeauthor*{TaTo13} \cite{TaTo13}, \citeauthor*{dolinsky2014robust} \cite{dolinsky2014robust}, \citeauthor*{BiBoKaNu14} \cite{BiBoKaNu14}
and the references therein.
General frameworks for model uncertainties were introduced in \citeauthor*{BoNu13} \cite{BoNu13} for discrete-time financial markets.
They establish  in great generality a robust analogue of the first fundamental theorem of asset pricing: the absence of arbitrage in a quasi-sure sense entails the existence of a suitable family of martingale measures.
Building on this, they prove existence of optimal superhedging strategies as well as a robust superhedging duality. 
In this work, we consider a particular setting that falls well within the framework of \citeauthor*{BoNu13} \cite{BoNu13} and give interpretation to an adapted multimarginal optimal transport problem as a robust hedging problem.

\medskip
In recent years, adapted transport has become increasingly popular as a means to measure distances between stochastic processes.
At least in a discrete-time seting, the adapted Wasserstein distance turns out to be the correct distance for most purposes that involve stochastic processes and their filtration.
Applications of adapted transport and topologies include:
Continuity of optimal stopping \citeauthor*{BaBaBeEd19a} \cite{BaBaBeEd19a} and multistage stochastic optimization problems \citeauthor*{PfPi12} \cite{PfPi12}, pricing, hedging, and utility indifference pricing \citeauthor*{BaBaBeEd19b} \cite{BaBaBeEd19b}, Lipschitz continuity of the Doob decomposition \citeauthor*{BaBePa21} \cite{BaBePa21}, sensitivity analysis of convex stochastic optimization problems \citeauthor*{BaWi23} \cite{BaWi23}, stability of martingale optimal transport \citeauthor*{JoPa23} \cite{JoPa23}, enlargement of filtrations \citeauthor*{AcBaZa20} \cite{AcBaZa20}, interest rate uncertainty \citeauthor*{AcBePa20} \cite{AcBePa20}, higher-rank expected signatures \citeauthor*{BoLiOb23} \cite{BoLiOb23}, and their application in the context of machine learning for filtration-dependent problems \citeauthor*{HoLeLiLySa23} \cite{HoLeLiLySa23}. Additionally, they find use in solving dynamic matching problems, see \citeauthor*{BaHa23} \cite{BaHa23} and \citeauthor*{AcKrPa23a} \cite{AcKrPa23a}.

\medskip In adapted optimal transport, one can derive a dual problem similar to that of Kantorovich and Rubinstein. The first result on duality was established in the work of \citeauthor*{BaBeLiZa17} \cite{BaBeLiZa17}, and it has been further used
or extended in several other works, \emph{e.g.} \citeauthor*{AcBaZa20} \cite{AcBaZa20}, \citeauthor*{AcBaJi20} \cite{AcBaJi20}.

\medskip
As of now, there has been a lack of results regarding general duality and dual attainment for the adapted versions of optimal transport. The main challenge we face is the relatively complex nature of the dual potentials. Firstly, these involve martingale terms with a specific structure, further dependent on the transition kernels of the given marginals, as well as variables from both (or, in the multimarginal case, all) involved spaces. Since there is no fixed coupling of the marginals, we
lack a canonical measure on the product space and the set of admissible couplings generally contains mutually singular measures. Consequently, standard arguments using Komlós' lemma are not sufficient. Moreover, unlike in entropic optimal transport, as considered in the adapted case \emph{e.g.}\ in the work of \citeauthor*{EcPa22} \cite{EcPa22}, the dual of the unregularized adapted optimal transport problem requires pointwise constraints on the potentials, which we need to verify. The complexity increases when one considers the dual of the adapted versions of the barycenter problem, where even more intricate constraints on the potentials are imposed.

\medskip
Our main contribution is the proof of duality and the existence of dual optimizers for both causal and bicausal optimal transport, as well as for the causal and bicausal barycenter problem and a multimarginal problem with causality constraints between the marginals. 
Notably, we establish these results with minimal assumptions on the cost function and the transition kernels. A slight exception here are the results for barycenters, where certain regularity is imposed to obtain duality, but can be dropped when showing attainment of the dual. As a consequence, our results cover essentially any sensible framework. Nonetheless, it is important to note that while our method allows for an extremely general setting, in general we do not expect any regularity from the potentials either. The main approach we employ involves the Choquet capacitability theorem, originally used by \citeauthor*{Ke84} \cite{Ke84} for classical optimal transport, together with showing that a maximizing sequence of potentials admits measurable convex combinations of a subsequence converging to a limit. As mentioned earlier, the dual potentials depend in a measurable way on variables potentially from several spaces. Since there is no reference measure, we cannot employ the Komlos' lemma for all potentials. We instead show that there exist measurable convex combinations that converge to a limit. This, however, has to be done in an inductive way forward in time to ensure the right structure of the limits. A slightly different approach is employed in the causal barycenter problem, where our method can only be used to show dual attainment with martingale compensators depending on a chosen barycenter candidate. To aggregate these over all candidates, we additionally assume the continuum hypothesis and apply transfinite induction. This aligns with similar results in the literature that deal with non-dominated sets of measures. Moreover, we provide a complete description of the polar sets in causal and bicausal optimal transport, an interpretation of the dual problem in the context of the robust superhedging problem and further give intuition of the dual optimizers using Bellman's optimality principle. We also refer to \cite{AcKrPa23a} where the dual problem naturally appears when one studies equilibria in dynamic matching models.

\medskip
The remainder of the paper is organized as follows. The results in \Cref{sec:main_results} are presented in a significantly simplified setting for ease of exposition and will later be largely generalized. Specifically, in \Cref{sec:AOT}, the adapted optimal transport is introduced and in \Cref{sec:multi,sec:bary}, we formulate the multicausal transport and the barycenter problems, presenting our main results. In \Cref{sec:appl}, we apply the theory to robust hedging and also address the dynamic programming approach. In doing so, we offer different points of view and interpretations for the dual problem. Finally, in \Cref{sec:general}, we formulate the problems in full generality and provide the proofs. \Cref{sec:appendix} then contains technical lemmata and supporting results.

\medskip
\textbf{Notations:} We denote by $\N$ the set of positive integers and $\R$ the set of real numbers. If $(\Omega,\Fc)$ is a measurable space, we write $\Pc(\Omega,\Fc)$ for the set of all probability measures on $(\Omega,\Fc)$. If $\Omega$ is a Polish space and $\Fc$ is the corresponding Borel $\sigma$-algebra, we write $\Pc(\Omega)=\Pc(\Omega,\Fc)$ for brevity, and equip $\Pc(\Omega)$ with the topology of weak convergence of measures. Let $\Pc \subseteq \Pc(\Omega,\Fc).$ We say that a statement holds $\Pc$--quasi-surely, abbreviated to $\Pc$--q.s., if it holds outside of a set $A \in\Fc$ satisfying $\P(A)=0$ for every $\P \in \Pc.$ For $\P \in \Pc(\Omega,\Fc)$ and $p \geq 1$, we denote by $\L^p(\Fc,\P)$ the set of all real-valued random variables on $(\Omega,\Fc)$ with finite $p$-th moment. If the $\sigma$-algebra $\Fc$ is obvious from the context, we shall write $\L^p(\P)$ instead for brevity. If $(Y,\Yc)$ is a measurable space and $f : \Omega \longrightarrow Y$ is measurable, we write $f_{\#}\P \in \Pc(Y,\Yc)$ for the push-forward of $\P$ under $f$. Further, if $N \in \N,$ we denote the elements of the canonical basis of $\R^N$ by $e_i.$ For $T \in \N$ and some given elements $x_s \in \R,$ $s \in \{1,\ldots,T\},$ we use the notation $x_{t:u}\coloneqq (x_t,x_{t+1},\ldots,x_u),$ $1\leq t \leq u \leq T.$ Let $\mu \in \Pc(\R^T)$, we say that $\mu_1 \in \Pc(\R),$ $K_t : \R^{t-1} \longrightarrow \Pc(\R),$ $t \in \{2,\ldots,T\}$ are the successive disintegrations of $\mu$ if it admits the disintegration
\[ \mu(\mathrm{d}x_1,\ldots,\mathrm{d}x_T)=\mu_1(\mathrm{d}x_1) \otimes K_2(x_1;\mathrm{d}x_2)\otimes\ldots\otimes K_T(x_{1:T-1};\mathrm{d}x_T).\]
Finally, for given sets $A^i$ and $x^i \in A^i$, $i \in \{1,\ldots,N\}$, we use the shorthand notation $x^{1:N} \coloneqq (x^1,\ldots,x^N) \in A^{1:N}\coloneqq\prod_{i=1}^N A^i.$ Similar notation shall be used for other indices.

\section{Main results} \label{sec:main_results}

In this section, we introduce the main results in a simplified setting for the ease of exposition. 
All statements will be rigorously formulated and the involved sets will be properly introduced in \Cref{sec:general}, where we also refer for more discussion and details.

\subsection{Adapted optimal transport} \label{sec:AOT}

Adapted optimal transport heavily borrows ideas from optimal transport, which we now briefly introduce.
Let us consider two probability measures on $\R^T$, say $\mu,\,\nu \in \mathcal P(\mathbb R^T)$, for some $T \in \N$.
Given a cost function $c : \R^T \times \R^T \longrightarrow \R$, the task is to find an efficient way of moving $\mu$ to $\nu$, where the transportation aspect is formalized via transport maps.
A transport map $S : \R^T \longrightarrow \R^T$ from $\mu$ to $\nu$ is a measurable map that satisfies the pushforward constraint $S_\# \mu = \nu$.
Loosely speaking, all mass that $\mu$ puts at $x$ is moved to $S(x)$, which incurs the cost $c(x,S(x))$ and thereby the total transport cost is $\int c(x,S(x)) \mu(\mathrm{d}x)$.
As the concept of transport maps has its natural limitations, one considers couplings, which can be understood as randomized transport plans.
A coupling $\pi \in \Pc(\R^T \times \R^T)$ between $\mu$ and $\nu$ is a probability measure such that $\pi( \mathrm{d}x \times \R^T) = \mu( \mathrm{d} x)$ and $\pi(\R^T \times \mathrm{d}y) = \nu(\mathrm{d}y)$ and incurs the transport cost $\int c(x,y)  \pi(\mathrm{d}x,\mathrm{d}y)$.
The set of all couplings between $\mu$ and $\nu$ is denoted by $\cpl(\mu,\nu)$.

\medskip
Adapted transport extends this approach to stochastic processes. Let us view $\R^T$ as a path space for real-valued paths in $T$ time steps, and $\mu$, $\nu$ as laws of processes.
Instead of considering all transport maps, we restrict to those $S : \R^T \longrightarrow \R^T$ with $S_\# \mu = \nu$ that are adapted, also sometimes referred to as causal. That is to say,
\begin{equation}
    \label{eq:adapted.map}
    S(x_1,\ldots,x_T) = (S_1(x_1), \dots, S_T(x_1,\ldots,x_T)), \quad (x_1,\ldots,x_T) \in \mathbb R^T,
\end{equation}
where $S_t : \R^t \longrightarrow \R$ for $t \in \{ 1,\ldots,T \}$.
We remark that in \eqref{eq:adapted.map} the value of $S$ at time $t$ does not depend on the future evolution of the path $(x_1,\ldots,x_T)$.
In this sense, $S$ is adapted to the available information at time $t$.
As above, we relax the problem by considering couplings, see for example \citeauthor*{BePaSc21c} \cite{BePaSc21c} for a recent study.
For this reason, we have to adequately translate \eqref{eq:adapted.map}.
A coupling $\pi$ between $\mu$ and $\nu$ is called causal if
\begin{equation}
    \label{eq:causal.coupling}
    (X_1,\ldots,X_T) \text{ is } \text{independent of } (Y_1,\ldots,Y_t) \text{ given }(X_1,\ldots,X_t),
\end{equation}
where $(X,Y)$ is distributed according to $\pi$, for all $t \in \{ 1,\ldots,T-1 \}$.
Moreover, we call $\pi$ bicausal if \eqref{eq:causal.coupling} also holds true when the roles of $X$ and $Y$ are reversed.
The set of all causal, resp.\ bicausal, couplings between $\mu$ and $\nu$ is denoted by $\cplc(\mu,\nu),$ resp.\ $\cplbc(\mu,\nu)$.
The set of dual variables $\Sc^{\rm c}(\mu,\nu),$ resp.\ $\Sc^{\rm bc}(\mu,\nu),$ is a subset of real-valued measurable functions on $\R^T \times \R^T$ which have a certain martingale property under couplings in $\cplc(\mu,\nu),$ resp.\ $\cplbc(\mu,\nu)$.
They will be rigorously defined in \Cref{sec:general_adapt_causal}.
As our main contribution to adapted transport we establish duality as well as dual attainment:
\begin{theorem} \label{thm:intro_causa_adapt}
    Let $c : \R^T \times \R^T \longrightarrow \R \cup \{ - \infty \}$ be measurable and bounded from above.
    Then
    \begin{align} \label{eq:thm.cbc.intro.1}
        \Cc\Wc_c(\mu,\nu) \coloneqq& \inf_{\cpl_{\rm c}(\mu,\nu)}
        \int c(x,y)  \pi(\mathrm{d}x,\mathrm{d}y) = 
        \sup_{s \in \mathcal S^{\rm c}(\mu,\nu),\,s \leq c}
        \int s(x,y) (\mu \otimes \nu)(\mathrm{d}x,\mathrm{d}y), \\ \label{eq:thm.cbc.intro.2}
        \Ac\Wc_c(\mu,\nu) \coloneqq& \inf_{\cpl_{\rm bc}(\mu,\nu)}
        \int c(x,y)  \pi(\mathrm{d}x,\mathrm{d}y) = 
        \sup_{s \in \mathcal S^{\rm bc}(\mu,\nu),\,s \leq c}
        \int s(x,y) (\mu \otimes \nu)(\mathrm{d}x,\mathrm{d}y).
    \end{align}
    Moreover, if either side in \eqref{eq:thm.cbc.intro.1}, resp.\ \eqref{eq:thm.cbc.intro.2}, is finite, then the respective right-hand side is attained.
\end{theorem}

\begin{remark} $(i)$ Here, our main contribution
is the proof of duality for measurable cost functions, as well as the attainment of the dual. So far, there have only been results concerning duality under certain continuity assumptions:
{\rm\citeauthor*{BaBeLiZa17} \cite{BaBeLiZa17}} deal with the causal setting requiring lower-semicontinuity and boundedness from below of the cost function as well as weak continuity of the successive disintegration kernels associated with the marginals.
{\rm\citeauthor*{EcPa22} \cite{EcPa22}} show duality for the causal and bicausal optimal transport problems assuming lower semicontinuity and boundedness from below of the cost.
In contrast to that, {\rm\Cref{thm:intro_causa_adapt}} establishes duality for general measurable cost functions that are bounded from above.
Moreover, to the best of our knowledge, this is the first result treating the attainment of the dual problem in adapted optimal transport.

\medskip $(ii)$ The assumption of having an upper bound of the cost function in {\rm\Cref{thm:intro_causa_adapt}} can be relaxed to being upper-bounded by integrable functions. We refer to {\rm\Cref{existDualC,existDualBC}} for details.

\medskip $(iii)$ While the dual problems in \eqref{eq:thm.cbc.intro.1} and \eqref{eq:thm.cbc.intro.2} are attained, in this generality the primal problems are not necessarily attained.
\end{remark}

\begin{remark} 
    As mentioned earlier, the dual problem is of importance when studying the properties of the primal problem and primal optimizers. It is thus natural to investigate its attainment. Moreover, motivated by standard optimal transport, see \emph{e.g.}\ {\rm \citeauthor*{Ke84} \cite{Ke84}} and {\rm \citeauthor*{BeLeSc12a} \cite{BeLeSc12a}}, measurability of the cost function is a natural condition to obtain duality.
    In light of this, the assumption on lower-semicontinuity, as considered in {\rm\cite{BaBeLiZa17, EcPa22}}, appears to be overly restrictive.
\end{remark}

\subsection{Multimarginal adapted transport} \label{sec:multi}
In this section we address the multimarginal generalization of adapted transport which requires us to introduce appropriate notions of causality specific to the setting.
This refinement is advantageous in the study of adapted Wasserstein barycenters and the robust hedging problem, as discussed in \Cref{sec:robust} below. 

\medskip
Let us consider $N$ processes with laws $(\mu^1,\ldots,\mu^N) \in \Pc(\R^T)^N.$
Naturally, the set of couplings on $(\R^T)^N$ with marginals $(\mu^1,\ldots,\mu^N)$ is denoted by $\cpl(\mu^1,\ldots,\mu^N)$.
\begin{definition}
A coupling $\pi \in \cpl(\mu^1,\ldots,\mu^N)$ is called multicausal if, for every $i \in \{1,\ldots,N\}$ and $t \in \{1,\ldots, T\}$,
\begin{equation}
    (X_1^i,\ldots,X_T^i) \text{ is } \text{independent of }\{ (X^j_1,\ldots, X^j_t) : j \neq i \} \text{ given } (X^i_1,\ldots,X^i_t),
\end{equation}
where $(X^1,\ldots,X^N)$ is a vector of $\R^T$-valued random variables distributed according to $\pi$.
The set of multicausal couplings is denoted by $\cpl_{\rm mc}(\mu^1,\ldots,\mu^N)$.
\end{definition}

\begin{remark} \label{rem:J_causal} We note that in the case $N=2,$ the notion of multicausality coincides with bicausality.
\end{remark}

Analogously to \Cref{sec:AOT}, we write $\mathcal S^{\rm mc}$ for the set of dual variables which consists of real-valued measurable functions on $(\R^T)^N$ that have a certain martingale property under couplings in $\cplmc(\mu^1,\ldots,\mu^N)$.
A proper definition of $\mathcal S^{\rm mc}$ can be found in \Cref{sec:general_multi} below.
Thanks to this set, we have the following general duality result.

\begin{theorem} \label{thm:simple_multi}
    Let $c : \R^{T \cdot N} \longrightarrow \R \cup \{ - \infty\}$ be measurable and bounded from above.
    Then
    \begin{equation}
        \inf_{\pi \in \cpl_{\rm mc}(\mu^1,\ldots,\mu^N)} \int c(x^1,\ldots,x^N) \pi(\mathrm{d}x^1,\ldots,\mathrm{d}x^N) =
        \sup_{s \in \mathcal S^{\rm mc},\, s \le c}
        \int s(x^1,\ldots,x^N)  \Big(\bigotimes_{i = 1}^N \mu^i \Big)(\mathrm{d}x^1,\ldots,\mathrm{d}x^N).
    \end{equation}
    Moreover, if either side is finite, then the right-hand side is attained.
\end{theorem}

\begin{remark} $(i)$ Similarly to the previous section, our main contribution is duality and dual attainment for general multimarginal optimal transport with a measurable cost function. We refer to {\rm \citeauthor*{AcKrPa23a} \cite{AcKrPa23a}} for potential applications of the `multicausal' transport case.

\medskip $(ii)$ The assumption that the cost is bounded from above can be relaxed to being bounded from above by integrable functions. See {\rm\Cref{existDualMC}} for details.
\end{remark}

\subsection{Causal and bicausal barycenters} \label{sec:bary}
Wassserstein barycenters offer a method of averaging probability measures that stays truthful to the geometry of the underlying base space.
This concept, initially studied by \citeauthor*{AgCa11} \cite{AgCa11}, has found significant applications in machine learning.
In mathematical economics, \citeauthor*{CaEk10} \cite{CaEk10} explored the team matching problem establishing its connection to Wasserstein barycenters.
Expanding upon this, \cite{AcKrPa23a} introduced a dynamic version of the matching problem, replacing the Wasserstein barycenter with the causal Wasserstein barycenter.
This variation of averages takes into account the geometry of the underlying spaces as well as the flow of information of the different processes.

\medskip Again, we fix $N$ laws of stochastic processes $(\mu^1,\ldots,\mu^N) \in \Pc(\R^T)^N$ and a compact set $A=\prod_{t=1}^T A_t \subseteq \R^T$.
For each $i$, we have a cost function $c^i : \R^T \times A \longrightarrow \R$.
Then the causal barycenter problem is then the minimization of
\[
    \sum_{i = 1}^N \Cc\Wc_{c^i}(\mu^i,\nu)
\]
over all measures $\nu \in \mathcal P(A)$.
The set of dual variables $\Phi(c^1,\ldots,c^N)$ associated with this problem consists of measurable functions $f^i : \R^T \longrightarrow \R$, such that there are measurable $g^i : A \longrightarrow \R$ and $M^i : \R^T \times A \longrightarrow \R$, $i \in \{1,\ldots,N\},$ where $\sum_{i = 1}^N g^i = 0$ and $M^i$ have a certain martingale property, with
\[
    f^i(x^i) + g^i(y) + M^i(x^i,y) \le c^i(x^i,y)\quad (x^i,y) \in \R^T \times A.
\]
Using the set $\Phi^0(c^1,\ldots,c^N)$, rigorously defined in \Cref{sec:general_bary} in \eqref{eq:def.Phi0}, we establish the following duality result.

\begin{theorem} \label{thm:intro.barycenter}
    Let $A \subseteq \R^T$ be compact and for $i \in \{1,\ldots,N\}$, let $c^i : \R^T \times A \longrightarrow \R$ be measurable and bounded.
    Then
    \begin{equation} \label{eq:thm.intro.barycenter}
        \inf_{\nu \in \Pc(A)} \sum_{i = 1}^N \Cc\Wc_{c^i}(\mu^i, \nu) =
        \sup_{(f^i)_{i = 1}^N \in \Phi^0(c^1,\ldots,c^N)} \sum_{i = 1}^N \int f^i(x^i) \mu^i(\mathrm{d}x^i).
    \end{equation}
     Moreover, the right-hand side is attained.
\end{theorem}

\begin{remark} The duality result for causal barycenters with lower-semicontinuous cost functions was proven in {\rm  \citeauthor*{AcKrPa23a} \cite{AcKrPa23a}}, where we also refer for more discussions. 
Here, we once again relax the assumption on the cost function and show attainment. The assumption on the cost being bounded as well as compactness of $A$ can, similarly as before, be relaxed, see {\rm\Cref{thm:attain_bary}} for details.
\end{remark}

\begin{remark}
    We want to mention that the dual problem to the barycenter problem considered in {\rm\Cref{thm:intro.barycenter}}, resp.\ {\rm\Cref{thm:attain_bary}}, closely resembles a robust version of the collective superreplication problem recently introduced by
    {\rm\citeauthor*{BiDoFoFrMe23} \cite{BiDoFoFrMe23}}.
    We leave exploring this connection for future research.
\end{remark}

Similarly, the bicausal barycenter problem is the minimization of
\[
    \sum_{i = 1}^N \Ac\Wc_{c^i}(\mu^i,\nu)
\]
over all measures $\nu \in \Pc(A)$.
The set of dual variables $\Phi^\Zc(c^1,\dots,c^N)$ associated with this problem consists of measurable functions $f^i : \R^T \longrightarrow \R$, such that there are measurable $g^i : \prod_{t = 1}^T (A_t \times \Pc(A_{t+1:T})) \longrightarrow \R$ and $M^i : \R^T \times \prod_{t = 1}^T (A_t \times \Pc(A_{t+1:T})) \longrightarrow \R$, $i \in \{1,\dots,N\}$, where $\sum_{i = 1}^N g^i = 0$ and $M^i$ have again a certain martingale property, with
\[
    f^i(x^i) + g^i(z_1) + M^i(x^i,z) \le c^i(x^i,y), \quad (x^i,z) \in \R^T \times \prod_{t = 1}^T (A_t \times \Pc(A_{t+1:T})).
\]
For $z_t = (y_t,p_t) \in A_t \times \Pc(A_{t+1:T})$, we interpret $y_t$ as the location of the process at time $t$ and $p_t$ as the predicted future evolution of $y$ at time $t$.
Using the set $\Phi^\Zc(c^1,\dots,c^N)$, rigorously defined in \Cref{sec:bcbary} in \eqref{eq:def.bcbary.potentials}, we establish the following duality result.

\begin{theorem} \label{thm:intro.bcbarycenter}
    Let $A$ be compact and for $i \in \{1,\dots,N\}$, let $c^i : \R^T \times A \longrightarrow \R$ be continuous and bounded.
    Then
    \begin{equation}
        \label{eq:thm.intro.bcbarycenter}
        \inf_{\nu \in \Pc(A)} \sum_{i = 1}^N \Ac\Wc_{c^i}(\mu^i,\nu) = \sup_{(f^i)_{i =1}^N \in \Phi^\Zc(c^1,\dots,c^N)} \sum_{i = 1}^N \int f^i(x^i) \mu^i(\mathrm{d}x^i).
    \end{equation}
    Moreover, the right-hand side is attained.
\end{theorem}

\begin{remark} In comparison to {\rm\Cref{thm:intro.barycenter}}, we require in {\rm\Cref{thm:intro.bcbarycenter}} continuity of the cost functions $c^i$, $i \in \{1,\dots,N\}$.
    The reason behind this slightly more restrictive assumption lies in the different nature of the causal and bicausal barycenter problems.  While for the causal barycenter problem there is no difference in taking the infimum either over $\Pc(A)$ or the larger space ${\rm FP}(A)$ of all filtered processes with paths in $A$, see \eqref{eq:bary.canonicalvsfp}, the very same fails for the bicausal barycenter problem.
    Nonetheless, $\Pc(A)$ can be identified with the set of stochastic processes that are equipped with their generated filtration.
    Since the latter is, according to {\rm \cite[Theorem 1.2]{BaBePa21}}, a dense subset of ${\rm FP}(\R^T)$, we recover \eqref{eq:bary.canonicalvsfp} in the bicausal case under the additional continuity assumption. We treat the general case in {\rm \Cref{sec:general}}, where we relax this assumption as well as the compactness assumption on $A$ by directly working with ${\rm FP}(\R^T)$.
\end{remark}

\section{Applications and discussions} \label{sec:appl}

The aim of this section is to discuss some applications of the results above and provide additional intuition regarding the dual problem. As before, we consider simplified settings for readers' convenience, and remark that most of the results can be further generalized.

\subsection{Robust hedging} \label{sec:robust}
In this section, we present an application of our results to robust hedging. 
This setting falls within the framework considered in the work of \citeauthor*{BoNu13} \cite{BoNu13}. The main aim of this section is thus to link our results to the previously known theory and provide, to the best of our knowledge, a new point of view on adapted optimal transport and the corresponding dual problem.

\medskip
Let $T$ be the time horizon and $N$ be the number of assets whose values are modeled by $\R$-valued processes $X^i=(X^i_t)_{t=1}^T$, $i \in \{1,\ldots,N\},$ with distributions $\mu^i \in \Pc(\R^T)$ under some risk-neutral measure. That is, $X^i$ is an $(\F^{i},\mu^i)$-martingale, where $\F^{i}$ denotes the canonical filtration of $X^i$.  We assume that $X^i$ is the canonical process on $(\R^T,\Bc(\R^T),\F^{i},\mu)$ for simplicity.
We set $X_0^i\coloneqq x_0^i$, where $x_0^i \coloneqq \E^{\mu^i} [X_1^i]$ and denote by $K^{i}_t : \R^{t-1} \longrightarrow \Pc(\R),$ $t \in \{2,\ldots,T\}$ the successive disintegrations of $\mu^i.$ Moreover, we assume that the restricted market consisting of $X^i$ is complete in the sense that for every $(\F^{i},\mu^i)$-martingale $M^i$ there exists an $\F^{i}$-adapted process $\Delta^i=(\Delta^i_t)_{t=0}^{T-1}$ such that for $t \in \{0,\ldots,T\}$
\begin{equation}
    \label{eq:intro.completeness}
    M_t^i =m_0^i +\sum_{s=1}^{t} \Delta^i_{s-1} (X_{s}^i-X_{s-1}^i)\quad \mu^i\text{{\rm--a.s.}}
\end{equation}
with $m_0^i \coloneqq \E^{\mu^i}[M^i_1]$.

\medskip
We assume that the joint distribution of the given processes is unknown. That is, we do not have information about the joint law $\pi \in \cpl(\mu^1,\ldots,\mu^N) \subset \Pc(\R^{T \times N})$ of $(X^1,\ldots,X^N)$. However, it is natural to consider only those $\pi \in \cpl(\mu^1,\ldots,\mu^N)$ which satisfy some appropriate no-arbitrage condition. More specifically, we define the following condition.

\begin{definition} We say that a coupling $\pi \in \cpl(\mu^1,\ldots,\mu^N)$ satisfies condition {\rm(NA)} if the $\R^{N}$-dimensional canonical process $\mathbf{X}\coloneqq(X^1,\ldots,X^N)$ on the product space $\R^{T \cdot N}$ is an $(\F^\mathbf{X},\pi)$-martingale, where $\F^\mathbf{X}$ denotes the filtration on $\R^{T \cdot N}$ generated by the process $\mathbf{X}.$
\end{definition}

\begin{remark} Let us point out that, slightly abusing the notations, we denote by $X^i,$ $i \in \{1,\ldots,N\},$ both the canonical processes introduced above as well as the components of the process $\mathbf{X}.$
\end{remark}

 Let us recall that $\cplmc(\mu^1,\ldots,\mu^N)$ denotes the set of all multicausal couplings introduced in \Cref{rem:J_causal}. We have the following characterization of couplings satisfying condition (NA).

\begin{proposition} \label{prop:NA}
    Let $\pi \in \cpl(\mu^1,\ldots,\mu^N)$. Then $\pi$ satisfies condition {\rm (NA)} if and only if $\pi \in \cplmc(\mu^1,\ldots,\mu^N)$.
\end{proposition}
\begin{proof}
    First, let us assume that $\pi \in \cpl(\mu^1,\ldots,\mu^N)$ is such that $\textbf{X}$ is an $(\F^\mathbf{X},\pi)$-martingale. 
    Let $i \in \{1,\ldots, N\}$ be arbitrary and let $\xi^i \in \L^1(\Fc_T^{i},\mu^i)$. 
    Using the completeness assumption in \eqref{eq:intro.completeness}, there is an $\F^{i}$-adapted process $\Delta^i$ such that 
    \[ 
        \xi^i=p_0+\sum_{t=1}^T \Delta_{t-1}^i (X_{t}^i-X_{t-1}^i), 
    \] 
    with $p_0 \coloneqq \E^{\mu^i}[\xi]$.
    Since $\mathbf{X}$ is an $(\F^\mathbf{X},\pi)$-martingale and $\xi^i$ is integrable, it is straightforward to verify that the process $\sum_{s=1}^\cdot \Delta_{s-1}^i (X_{s}^i-X_{s-1}^i)$ is an $(\F^\mathbf{X},\pi)$-martingale and, thus,
    \[ 
        \E^\pi\big[\xi^i \big\vert \Fc^\mathbf{X}_t\big]=p_0 +\sum_{s=1}^t \Delta_{s-1}^i (X_{s}^i-X_{s-1}^i)=\E^{\mu^i}\big[\xi^i \big\vert \Fc_t^{i}\big].
    \] 
    We have shown that $\xi^i$ is conditionally on $\Fc_t^{i}$ independent of the $\sigma$-algebra $\Fc^\mathbf{X}_t$ under $\pi$. Because $\xi^i \in \L^1(\Fc_T^{i},\mu^i)$ was arbitrary, this proves multicausality.

    \medskip 
    Conversely, if $\pi \in \cplmc(\mu^1,\ldots,\mu^N),$ then for any $i \in \{1,\ldots,N\}$ and $s \leq t$ we have
    \[ \E^{\pi}[ X_t^i \vert \Fc^\mathbf{X}_s]=\E^{\pi}[ X_t^i \vert \Fc^i_s]=X_s^i\quad \pi \text{--a.s},\] where the first equality follows from the fact that $\Fc^i_t$ is conditionally on $\Fc^i_s$ independent of $\Fc^\mathbf{X}_s$ under $\pi.$ Thus, $\mathbf{X}=(X^1,\ldots,X^N)$ is an $(\F^\mathbf{X},\pi)$-martingale and the proof is completed.
\end{proof}

Let us now consider a measurable payoff 
\[\xi : \R^{T\cdot N} \longrightarrow \R\] 
that we wish to hedge.
As mentioned earlier, we assume that the joint dynamics of the market are unknown. However, we are allowed to trade with each of the assets while using information about all assets. That is to say, we introduce the set of admissible trading strategies as follows:
\[ 
    \Ac \coloneqq \Big\{ \mathbf{\Delta}=(\mathbf{\Delta}_t)_{t=0}^{T-1} : \R^{T \cdot N} \longrightarrow \R \;\Big\vert\;  \mathbf{\Delta}\;\text{is}\; \F^\mathbf{X} \text{-adapted}   \Big\}.
\]

\begin{remark} 
    The trading strategy $\mathbf{\Delta}$ being $\F^\mathbf{X}$-adapted simply means that $\mathbf{\Delta}_t$ is a measurable function of the vector $(X^1_{1:t},\ldots,X^N_{1:t})$ for every $t \in \{0,\ldots,T\}$. 
    That is to say, we choose the hedging strategy while observing the paths of all the processes $X^1,\ldots,X^N$ up to time $t$. 
    Thus, to emphasize this point we sometimes write $\mathbf{\Delta}_t=\mathbf{\Delta}_t(X^1_{1:t},\ldots,X^N_{1:t})$.
\end{remark}

We are interested in the following superhedging problem
\[ p(\xi)\coloneqq \inf \bigg\{ p_0 \in \R \,\bigg\vert\, \exists \mathbf{\Delta} \in \Ac : p_0 + \sum_{t=1}^T \mathbf{\Delta}_{t-1} \cdot (\mathbf{X}_{t}-\mathbf{X}_{t-1}) \geq \xi(\mathbf{X})\; \cplmc(\mu^1,\ldots,\mu^n)\text{\rm --q.s.} \bigg\}.\]

To be able to employ attainment results for the primal and the dual multicausal problem, we assume the following boundedness condition on the payoff $\xi$, see \Cref{existDualMC} and \Cref{prop:multi_attain}.

\begin{assumption} \label{ass:xi_growth} 
    There exist functions $f^i \in \L^1(\Fc_T^{i},\mu^i)$, $i \in \{1,\ldots,N\},$
    such that $\xi(x^1,\ldots,x^N) \leq \sum_{i=1}^N f^i(x^i)$.
\end{assumption}

 We have the following duality result. We postpone the proof to \Cref{sec:general_multi} since it requires further preliminaries.

\begin{theorem}[Superhedging duality] \label{thm:intro.robustsuperhedging}
    Let $\xi$ be a measurable function satisfying {\rm \Cref{ass:xi_growth}}. 
    Then we have
    \[ 
        p(\xi)=\sup_{\P \in \cplmc(\mu^1,\ldots,\mu^N)} \E^{\P}[\xi].
    \]
\end{theorem} 

\begin{theorem}[Existence] \label{thm:intro.robustsuperhedging_existence}
    If $\xi$ is measurable and satisfies {\rm\Cref{ass:xi_growth}}, then the problem $p(\xi)$  admits a solution. If further $\xi$ is  upper-semicontinuous and {\rm\Cref{ass:xi_growth}} is satisfied by $-\xi$, then $\sup \{ \E^{\P}[\xi] : \P \in \cplmc(\mu^1,\ldots,\mu^N) \}$ is attained. That is to say, in such a case, there exist a coupling $\pi^\star \in \cplmc(\mu^1,\ldots,\mu^N)$, a superhedging strategy $\Delta^\star \in \Ac$, and an initial capital $p_0^\star \in \R$ such that $p_0^\star=p(\xi)$ and 
    \[
        p_0^\star + \sum_{t=1}^{T} \Delta_{t-1}^\star \cdot (\mathbf{X}_{t}^i-\mathbf{X}_{t-1}^i) \geq \xi(\mathbf{X})\quad\cplmc(\mu^1,\ldots,\mu^N)\text{{\rm--quasi-surely,}}
    \]
    and equality holds $\pi^\star$--almost surely.
    In particular, $p(\xi)=\E^{\pi^\star}[\xi]$.
\end{theorem}
\begin{proof} 
 To see the first part, we employ \Cref{thm:simple_multi}, see \Cref{existDualMC} for full generality, to show that the dual to $\sup\{\E^{\P}[\xi] :\P \in \cplmc(\mu^1,\ldots,\mu^N) \}$ is attained. Similarly as in proof of \Cref{thm:intro.robustsuperhedging} we can construct a solution to the problem $p(\xi).$ Existence of an optimizer for the latter problem follows from standard arguments, see \Cref{prop:multi_attain}.
\end{proof}

\begin{remark}[Market completeness] 
    The assumption of having the martingale representation property is indeed rather restrictive in discrete-time models. 
    However, we would like to point out that our results extend to the following situation. Let $X^i$, $i \in \{1,\ldots,N\}$, be continuous-time processes on some time interval $[0,\tau]$ with $\tau \in (0,\infty)$. The causality constraint is only imposed on some finite time grid, say at times $0=s_0\leq s_1 \leq \ldots\leq s_T=\tau.$ That is to say, we are given laws of continuous-time processes $\mu^i \in \Pc(\Cc([0,\tau],\R^d))$, but interpret them as laws of discrete-time processes with paths in $\prod_{t=1}^{T}\Xc_t,$ where $\Xc_t\coloneqq \Cc([s_{t-1},s_{t}],\R^d),$ \emph{c.f.}\ {\rm \Cref{sec:general}}. In such a scenario, we would thus assume that we can continuously trade in each market while only having limited information about the other markets, which is updated only at specified times $\{s_0,\ldots,s_T \}$. If all $X^i$ have the martingale representation property (as continuous-time processes), {\rm\Cref{thm:intro.robustsuperhedging,thm:intro.robustsuperhedging_existence}} extend to the above setting with the obvious modifications having been made.
\end{remark} 

\subsection{Bellman's principle} \label{sec:bellman}
A standard approach to tackling dynamic stochastic control problems is the so-called martingale optimality principle, sometimes referred to as the Bellman principle. Tracing back to the 1950s with the pioneering work of \citeauthor*{bellman1957dynamic} \cite{bellman1957dynamic}, it provides a universal way of finding optimal control through backward induction by studying the martingale properties of the value process, provided that some form of dynamic programming equation can be derived. In continuous time, this translates to a suitable form of the Hamilton--Jacobi--Bellman partial differential equation, or in the weak formulation, a suitable form of a first or second-order backward stochastic differential equation.

\medskip It has already been noted, see \emph{e.g.}\ \citeauthor*{PfPi14} \cite{PfPi14} and \citeauthor*{BaBeLiZa17} \cite{BaBeLiZa17}, that the bicausal optimal transport problem enjoys the dynamic programming principle, and thus the Bellman principle can be employed to find an optimizer by locally solving an optimization problem at every time step backwards in time. As a consequence, we can treat the problem as an optimal control problem.

\medskip In this section, we elaborate on the corresponding value process of this control problem and its link to the dual problem. In fact, we show that the dual optimizers coincide with the value process $\pi^\star$--almost surely for any optimal transport map $\pi^\star \in \cplbc(\mu,\nu)$. Thus, we provide a new interpretation for the dual potentials.

\medskip
    More specifically, let $(\mu,\nu) \in \Pc(\R^T) \times \Pc(\R^T)$ be given marginals and let $c : \R^{T} \times \R^{T} \longrightarrow [0,\infty)$ be lower-semicontinuous with $c(x,y)\leq \ell(x)+k(y)$ for some $\ell \in \L^1(\Bc(\R^T),\mu)$ and $k \in \L^1(\Bc(\R^T),\nu)$. Let us denote by $K^\mu_t$ and $K^\nu_t,$ $t \in \{2,\ldots,T\}$ the successive disintegrations of $\mu$ and $\nu,$ respectively. 
    We have the following dynamic programming principle for the bicausal optimal transport.

    \medskip
    Set $V_T(x,y)\coloneqq c(x,y)$ and define for $t \in \{ 1,\ldots, T-1\}$ inductively backwards in time
    \begin{align*}
        V_t(x_{1:t},y_{1:t})&\coloneqq\inf \bigg\{ \int V_{t+1}(x_{1:t+1},y_{1:t+1}) \pi_{t+1}(\mathrm{d}x_{t+1},\mathrm{d}y_{t+1}) \,\bigg\vert\, \pi_{t+1} \in \cpl \Big(K_{t+1}^\mu(x_{1:t};\,\cdot\,),K_{t+1}^\nu(y_{1:t};\,\cdot\,)\Big) \bigg\},    \\
        V_0&\coloneqq\inf \bigg\{ \int V_{1}(x_{1},y_{1}) \pi_{1}(\mathrm{d}x_{1},\mathrm{d}y_{1}) \,\bigg\vert\, \pi_{1} \in \cpl (\mu_1,\nu_1) \bigg\},
    \end{align*}
where by $\mu_1,$ resp.\ $\nu_1,$ we denote the projection on $\mu,$ resp.\ $\nu,$ on the first coordinate. That is, $\mu_1(\mathrm{d}x_1)\coloneqq \mu(\mathrm{d}x_1\times \R^{T-1})$ and $\nu(\mathrm{d}y_1)\coloneqq \nu(\mathrm{d}y_1\times \R^{T-1})$.
    \begin{remark} \label{rem:meas.V} 
        Because $c$ is lower-bounded and lower-semicontinuous, one can readily verify using standard methods that the map
        \[ 
            (x_{1:T-1},y_{1:T-1},\gamma,\eta)\longmapsto \inf \bigg\{ \int c(x_{1:T-1},x_T,y_{1:T-1},y_T) \pi(\mathrm{d}x_{T},\mathrm{d}y_{T}) \,\bigg\vert\, \pi \in \cpl (\gamma,\eta) \bigg\} 
        \] 
        is jointly lower-semicontinuous on $\R^{T-1}\times \R^{T-1} \times \Pc(\R) \times \Pc(\R)$. 
        As a consequence, we have that the map $ (x_{1:T-1},y_{1:T-1}) \longmapsto V_{T-1}(x_{1:T-1},y_{1:T-1})$ is Borel measurable.
        Analogously, one verifies measurability of $V_t$ for any $t \in \{1,\ldots,T-1\}$ backwards in time. 
    \end{remark}
    Note that the value process $V=(V_t)_{t=0}^T$ is by construction $\F^{X,Y}$-adapted, where $\F^{X,Y}$ denotes the filtration generated by the canonical process $(X,Y) \in \R^T \times \R^T$ and $\Fc^{X,Y}_0$ is the trivial $\sigma$-algebra on $\R^T \times \R^T$. Roughly speaking, the dynamic programming principle asserts that $\pi\in \cplbc(\mu, \nu)$ is globally optimal if an only if its transition kernels are locally optimal for every time step. This is summarized in the following proposition.

    \begin{proposition} \label{lem:DPP} 
        We have $V_0=\inf_{\pi \in \cplbc(\mu, \nu)} \int c \mathrm{d}\pi.$ Moreover, $\pi^\star \in \cplbc(\mu,\nu)$ is optimal if and only if $\pi^\star$ admits a successive disintegration
        \[ 
            \pi^\star(\mathrm{d}x,\mathrm{d}y)=\pi_1^\star(\mathrm{d}x_1,\mathrm{d}y_1)\otimes \pi_2^\star(x_1,y_1;\mathrm{d}x_2,\mathrm{d}y_2) \otimes \ldots\otimes\pi_{T}^\star(x_{1:T-1},y_{1:T-1};\mathrm{d}x_T,\mathrm{d}y_T),
        \] 
        such that $\pi^\star_t(x_{t-1},y_{t-1};\,\cdot\,)\in\cpl (K_{t}^\mu(x_{1:t-1};\,\cdot\,),K_{t}^\nu(y_{1:t-1};\,\cdot\,)),$ resp.\ $\pi^\star_1 \in \cpl(\mu_1,\nu_1),$ is optimal for the problem $V_{t-1}(x_{1:t-1},y_{1:t-1})$ for every $t \in \{1,\ldots, T \}$ and $\cplbc(\mu,\nu)$--quasi-almost every $(x_{1:t-1},y_{1:t-1}).$
    \end{proposition}
    \begin{proof}
    See, for example, \cite[Section 5]{BaBeLiZa17} and \cite[Theorem 2.2]{AcKrPa23a}.
    \end{proof}

    \begin{remark}
    A version of the dynamic programming principle can be derived for the causal (see e.g.\ {\rm\cite[Theorem 2.7]{BaBeLiZa17}}) and multicausal optimal transport problems as well. 
    For ease of exposition we consider only the bicausal case here.
    \end{remark}

As usual in optimal control theory, optimality of a transport plan is equivalent to martingale property of the value process.

\begin{theorem} \label{thm:marti_optimality}
    The value process $(V_t)_{t=0}^T$ is an $(\F^{X,Y},\pi)$-submartingale for any $\pi \in \cplbc(\mu,\nu)$.
    Moreover, $V$ is an $(\F^{X,Y},\pi)$-martingale if and only if $\pi \in \cplbc(\mu,\nu)$ is an optimal coupling.
\end{theorem}
\begin{proof} 
    The first part can readily be seen since, clearly $0\leq V_t\leq \E^{\pi}[ V_{t+1} \vert \Fc_t^{X,Y}]$ $\pi$--almost surely for any $t \in \{0,\ldots,T-1\}$ and any $\pi \in \cplbc(\mu,\nu)$. 
    The latter property is a direct consequence of \Cref{lem:DPP}.
\end{proof}

Conversely to \Cref{thm:marti_optimality}, if there exists a process having the right martingale properties, then its initial value corresponds to the value of the transport problem. This is summarized in the following lemma.

\begin{lemma} 
    Let $\tilde{V}=(\tilde{V}_t)_{t=0}^T$ be an $\F^{X,Y}$-adapted process with $\tilde{V}_T(x,y)=c(x,y)$ $\cplbc(\mu,\nu)$--quasi-surely. 
    If $\tilde{V}$ is an $(\F^{X,Y},\pi)$-submartingale for every $\pi \in \cplbc(\mu,\nu)$ and an $(\F^{X,Y},\pi^\star)$-martingale for some $\pi^\star \in \cplbc(\mu,\nu)$.
    Then $\tilde{V}_0 = V_0$ and $\pi^\star$ is an optimal coupling.
\end{lemma}

\begin{proof} 
    For any $\pi \in \cplbc(\mu,\nu),$ we have the lower bound $\E^{\pi}[c]=\E^{\pi}[\tilde{V}_T] \geq \tilde{V}_0$. 
    Moreover, the lower bound is attained for $\pi^\star,$ showing the optimality of $\pi^\star$ and $V_0=\E^{\pi^\star}[c]$.
\end{proof}

In summary, we conclude that the value of the optimal transport problem $V_0$ can be characterized by submartingales whose terminal value coincides with $c$. We will now connect this observation with the dual problem introduced in the sections above. We reformulate the dynamic programming principle as follows: Instead of studying submartingales with a terminal value of $c$, we consider martingales with a terminal value less than $c$.

\begin{remark} Vaguely speaking, we expect these two formulations to be equivalent because of the following observation: If a submartingale has the terminal value $c$, its compensated martingale counterpart will have a terminal value less than $c$. Conversely, if a martingale's terminal value is less than $c$, we can add a positive drift to target the value $c$, thus obtaining a submartingale. Indeed, if one wishes to formalize this observation, an appropriate notion of robust Doob's decomposition similarly to {\rm\cite[Theorem 6.1]{BoNu13}} has to be employed.
\end{remark}

More specifically, we consider the following problem

    \begin{equation} \label{eqn:D_0}
        D_0\coloneqq\sup \big\{ M_0 \,\big|\, M=(M_t)_{t=0}^T \in \Mc^{\mu,\nu},\; M_T(x,y)\leq c(x,y)\; \cplbc(\mu,\nu) \text{{\rm--q.s.}}\big\},
    \end{equation}
    where
    \[ 
        \Mc^{\mu,\nu} \coloneqq \big\{ M=(M_t)_{t=0}^T \text{ is an } \big(\F^{X,Y},\pi\big)\text{-martingale for any }\pi \in \cplbc(\mu,\nu)\big\}.
    \]
    
\begin{remark} Let us point out that the set $\Mc^{\mu,\nu}$ is always non-empty as it contains all martingales that are admissible for the dual problem of the adapted optimal transport, see { {\rm \eqref{eqn:dual_potent_cau_bicau}}}.
\end{remark}
\begin{theorem} \label{thm:DPP} We have $V_0=D_0$. Moreover, if $M^\star \in \Mc^{\mu,\nu}$ is a solution to the right-hand side in \eqref{eqn:D_0} and $\pi^\star \in \cplbc(\mu,\nu)$ is a solution to the adapted optimal transport problem, then $V=M^\star$ $\pi^\star$--almost surely.
\end{theorem}
\begin{proof} It is readily seen that $D_0 \leq V_0$ and the equality follows from \Cref{thm:intro_causa_adapt}, resp.\ \Cref{existDualBC}. Let now $\pi^\star \in \cplbc(\mu,\nu)$ and $M^\star \in \Mc^{\mu,\nu}$ be optimal. Then,
\begin{equation*} 0\leq\E^{\pi^\star}[c-M_T^\star]= \E^{\pi^\star}[c]-\E^{\pi^\star}[M^\star_T]= V_0-M^\star_0= V_0-D_0=0.\end{equation*}
In particular $M_T^\star=c$ $\pi^\star$--almost surely, which proves for every $t \in \{0,\ldots,T\}$ that \[V_t=\E^{\pi^\star}\Big[c \Big\vert \Fc^{X,Y}_t\Big]=\E^{\pi^\star}\Big[M_T^\star \Big\vert \Fc^{X,Y}_t\Big]=M^\star_t,\;\pi^\star\text{{\rm--a.s.}}\]
This concludes the proof.
\end{proof}

The main result of this section follows. As a consequence of the preceding discussion, we have that the solution to the problem introduced in \eqref{eqn:D_0} coincides with the value process $V.$

\begin{corollary} \label{thm:dpp.duality} Any solution to the dual problem for adapted transport, see {\rm \eqref{eq:thm.cbc.intro.1}}, coincides $\pi^\star$--almost surely with the value process $V$ for any optimal coupling $\pi^\star$.
\end{corollary}
\begin{proof} Let $M^\star$ be a solution to the dual. It suffices to verify that $M^\star$ is also a solution to the problem $D_0$ and apply \Cref{thm:DPP}. It is clear that $M^\star$ is admissible for $D_0,$ see {\eqref{eqn:dual_potent_cau_bicau}}. We thus have to verify optimality. Indeed, we have
\[ M^\star_0=V_0=D_0.\]
This concludes the proof.
\end{proof}

Let us conclude this section with the following observation. Provided that the transition kernels as well as the cost function are continuous, we can show continuity of the value process in the $(x,y)$-entries. This then naturally translates to certain regularity of the dual potentials at least on supports of optimal transport maps.

\begin{theorem} \label{thm:dpp.continuity} Assume that $\mu$ and $\nu$ have weakly continuous successive disintegrations and $c$ is continuous and bounded. Then, $(x_{1:t},y_{1:t}) \longmapsto V_t(x_{1:t},y_{1:t})$ are continuous for any $t \in \{1,\ldots,T\}$.
\end{theorem}
\begin{proof} Similarly as in \Cref{rem:meas.V} we can show that \[ (x_{1:T-1},y_{1:T-1},\gamma,\eta)\longmapsto \inf \bigg\{ \int c(x_{1:T-1},x_T,y_{1:T-1},y_T) \pi(\mathrm{d}x_{T},\mathrm{d}y_{T}) \,\bigg\vert\, \pi \in \cpl (\gamma,\eta) \bigg\} \] is a continuous map on $\R^{T-1}\times \R^{T-1} \times \Pc(\R) \times \Pc(\R)$. Thus, $(x_{1:T-1},y_{1:T-1}) \longmapsto V_{T-1}(x_{1:T-1},y_{1:T-1})$ is continuous since it is a composition of continuous maps. The result then can be verified using analogous arguments backward in time.
\end{proof}

Using this result we can prove at least some regularity of the dual optimizer, provided that the functions $c$ as well as the transition kernels of $\mu$ and $\nu$ are continuous.

\begin{corollary} Assume that $\mu$ and $\nu$ have continuous successive disintegrations and $c$ is continuous and bounded. Then, any solution $M^\star$ to the dual problem in {\rm\eqref{eq:thm.cbc.intro.2}}, see also {\rm \Cref{sec:general_adapt_causal}}, $\pi^\star$--almost surely coincides with $V$ for any $\pi^\star \in \cplbc(\mu,\nu)$ which is optimal, and thus for any $t \in \{1,\ldots,T \}$ the map $(x_{1:t},y_{1:t}) \longmapsto M^\star_t(x_{1:t},y_{1:t})$ continuous on the support of $\pi^\star.$
\end{corollary}
\begin{proof}
    The statement is a direct consequence of \Cref{thm:dpp.duality} and \Cref{thm:dpp.continuity}.
\end{proof}

\section{General formulation and proofs} \label{sec:general}
In this section, we formulate the problems in full generality and provide the main results of this work. 
First, we address the situation with two marginals, namely the causal and bicausal optimal transport problems.
In the second part, we study the multimarginal variant and the barycenter problems.
All results are formulated with as little assumptions as possible.

\medskip Let us point out that the bicausal case discussed in \Cref{sec:general_adapt_causal} is covered by the result concerning the multimarginal case in \Cref{sec:general_multi}. For clarity of presentation, we begin by addressing the case with just two marginals. This approach is taken to help the reader, because the proof for the general case involves a rather heavy notation, which may reduce the transparency of some arguments.

\subsection{Adapted optimal transport and duality} \label{sec:general_adapt_causal}
Let $T \in \N$ be a fixed time horizon and let us write $\Tc \coloneqq \{1,\ldots,T\}$.
For each $t \in \Tc$ we consider Polish spaces $\Xc_t$ and $\Yc_t$ and denote their products by $\Xc\coloneqq  \prod_{t=1}^T\Xc_t$ and, similarly, $\Yc\coloneqq \prod_{t=1}^T\Yc_t$. 
We fix filtered processes
\[ \X=\big(\O^\X,\Fc^\X,\F^\X,\P^\X,X \big)\quad{\rm and}\quad\Y=\big(\O^\Y,\Fc^\Y,\F^\Y,\P^\Y,Y \big),\] 
where $(\O^\X, \Fc^\X)$ and $(\O^\Y,\Fc^\Y)$ are standard Borel spaces, 
$(\O^\X,\Fc^\X,\F^\X,\P^\X)$, $(\O^\Y,\Fc^\Y,\F^\Y,\P^\Y)$ are filtered probability spaces, $X$ is an $\F^\X$-adapted process with $X_t \in \Xc_t$, and $Y$ is an $\F^\Y$-adapted process with $Y_t \in \Yc_t$. We work under the following simplifying standing assumption, which, however, is in most cases without loss of generality thanks to \citeauthor*{BaBePa21} \cite{BaBePa21}.
Indeed, we may replace $\X$ and $\Y$ with their `canonical filtered process{\rm'} counterparts, \emph{c.f.}\ \cite[Definition 3.7]{BaBePa21}, which satisfy the assumption below.
In turn, they allow us to pull back primal and dual optimizers to the original problem with $\X$ and $\Y$, see \Cref{rem:pull_back}.

\begin{assumption} 
    The probability space $\O^\X $ is the product of some Polish spaces $\O^\X_t$, $t \in \Tc$, \emph{i.e.}\ $\O^\X = \prod_{t=1}^T \O^\X_t$.
    The filtration $\F^\X$ is the corresponding canonical filtration generated by the coordinate projections on $\prod_{t=1}^T \O^\X_t$. 
    That is to say, $\Fc^\X_t = \bigotimes_{s=1}^t\Bc(\O^\X_s)\otimes \bigotimes_{s=t+1}^T \{\emptyset, \O^\X_s \}$ for $t \in \Tc$.
    Further, we have that $\Fc^\X=\Fc^\X_T.$ The same conditions \emph{mutatis mutandis} hold true for the filtered process $\Y$.
\end{assumption}

We use the notation $\Xc_{1:t} \coloneqq \prod_{s=1}^t \Xc_s$ and similarly $\O^\X_{1:t}\coloneqq \prod_{s=1}^t \O^\X_s$. Generic elements of $\O^\X$, resp.\ $\O^\X_t$, will be denoted by $\o^\X$, resp.\ $\o^\X_t$, and we write $\o^\X_{1:t}\coloneqq (\o^\X_1,\ldots,\o^\X_t) \in \O^\X_{1:t}$ for generic elements of $\O^\X_{1:t},$ $t \in \Tc$. Analogous notation shall be used for $\Y$. We further write $K^\X_t: \O^\X_{1:t-1}\longrightarrow \Pc(\O^\X_{t:T})$ for a regular version of $\P^\X(\mathrm{d}\omega_{t:T}^\X\vert \Fc_{t-1}^\X)$ for $t \in \{2,\ldots,T\}.$ Similarly, we use the notation $K^\Y_t,$ $t \in \{2,\ldots,T\},$ for $\P^\Y(\mathrm{d}\omega_{t:T}^\X\vert \Fc_{t-1}^\Y)$.
Finally, we set $\Fc_0^\X\coloneqq\{\Omega^\X,\emptyset \}$ and $\Fc_0^\Y\coloneqq\{\Omega^\Y,\emptyset \}.$

\begin{definition} \label{def:couplings}
    We denote by $\cpl(\X,\Y)$ the set of all probability measures on $\O^\X \times \O^\Y$ with marginals $\P^\X$ and $\P^\Y$ and call its elements couplings. We say that a coupling $\pi \in \cpl(\X,\Y)$ is:
\begin{enumerate}[label = (\roman*), leftmargin=*]
    \item \label{it:def.couplings.causal} causal, denoted by $\pi \in \cplc(\X,\Y)$, if $\Fc^\X_T \otimes \Fc^\Y_0$ is $\pi$-independent of $\Fc^\X_0 \otimes \Fc^\Y_t$ conditionally on $\Fc^\X_t \otimes \Fc^\Y_0,$ $t \in \Tc$;
    \item \label{it:def.couplings.anticausal} anticausal, denoted by $\pi \in \cplac(\X,\Y)$, if $\Fc^\X_0 \otimes \Fc^\Y_T$ is $\pi$-independent of $\Fc^\X_t \otimes \Fc^\Y_0$ conditionally on $\Fc^\X_0 \otimes \Fc^\Y_t,$ $t \in \Tc$;
    \item \label{it:def.couplings.bicausal} bicausal, denoted by $\pi \in \cplbc(\X,\Y)$, if $\pi \in \cplc(\X,\Y) \cap \cplac(\X,\Y)$.
\end{enumerate}
\end{definition}

Let $c: \O^\X \times \O^\Y \longrightarrow \R$ be a measurable cost function. We consider the following optimal transport problems
\begin{align*} \Cc\Wc_c(\X,\Y)&\coloneqq\inf_{\pi \in \cplc(\X,\Y)} \int_{\O^\X \times \O^\Y} c(\o^\X,\o^\Y) \pi(\mathrm{d}\o^\X,\mathrm{d}\o^\Y),
\\
\Ac\Wc_c(\X,\Y)&\coloneqq\inf_{\pi \in \cplbc(\X,\Y)} \int_{\O^\X \times \O^\Y} c(\o^\X,\o^\Y) \pi(\mathrm{d}\o^\X,\mathrm{d}\o^\Y)
\end{align*}
We have the following standard result, see \cite[Remark 2.3]{EcPa22}.
\begin{proposition} 
 Let $c : \O^\X \times \O^\Y \longrightarrow \R \cup \{+\infty\}$ be a lower-semicontinuous function and
 $\ell \in \L^1(\P^\X)$ and $k \in \L^1(\P^\Y)$ be such that $\ell(\o^\X)+k(\o^\Y) \leq c(\o^\X,\o^\Y),$ $(\o^\X,\o^\Y) \in \O^\X \times \O^\Y$. 
 Then the infima in $\Cc\Wc_c(\X,\Y)$ and $\Ac\Wc_c(\X,\Y)$ are attained.
\end{proposition}

As we are interested in studying the dual problems, we introduce the set of dual functions in a similar manner as in \cite[Section 5]{EcPa22}.
The causality and anticausality property of a coupling, \emph{c.f.}\ \Cref{def:couplings}$.\ref{it:def.couplings.causal}$ and $\ref{it:def.couplings.anticausal}$, can be verified by testing, for each $t \in \{2,\ldots,T\}$, against the following sets of dual variables
\begin{align*}
    \Ac_{\X,t}\coloneqq\bigg\lbrace f_t(\o^\X,\o^\Y)&=a_t(\o^\X_{1:t},\o^\Y_{1:t-1})-\int_{\O^\X_t} a_t(\o^\X_{1:t-1},\tilde{\omega}^\X_t,\o^\Y_{1:t-1}) K_t^{\X}(\o^\X_{1:t-1};\mathrm{d}\tilde{\omega}^\X_t)\,\bigg\vert\, \\
    &a_t \text{ is }\Fc_t^\X \otimes \Fc_{t-1}^\Y\text{-measurable and } a_t(\o^\X_{1:t-1},\,\cdot\,,\o^\Y_{1:t-1}) \in \L^1\big( \Bc(\O_t^\X),K_t^\X(\o^\X_{1:t-1};\,\cdot\,)\big)\bigg\rbrace,
    \\
    \Ac_{\Y,t}\coloneqq\bigg\lbrace g_t(\o^\X,\o^\Y)&=a_t(\o^\X_{1:t-1},\o^\Y_{1:t})-\int_{\O^\Y_t} a_t(\o^\X_{1:t-1},\o^\Y_{1:t-1},\tilde{\omega}^\Y_t)K_t^{\Y}(\Bc(\O_t^\Y);\mathrm{d}\tilde{\omega}^\Y_t)\,\bigg\vert\, \\
    &a_t \text{ is }\Fc_{t-1}^\X \otimes \Fc_{t}^\Y\text{-measurable and } a_t(\o^\X_{1:t-1},\o^\Y_{1:t-1},\,\cdot\,) \in \L^1\big(\Bc(\O_t^\Y),K_t^\Y(\o^\Y_{1:t-1};\,\cdot\,)\big)\bigg\rbrace.
\end{align*}
Similarly, the next sets of functions
\begin{align*}
    \Ac_{\X,1}\coloneqq \lbrace f_1(\o^\X,\o^\Y)
    =a_1(\o^\X_1) \,\vert\, a_1 \in \L^1(\Fc^\X_1,\P^\X)\rbrace
    \quad&{\rm and}\quad
    \Ac_{\Y,1}\coloneqq\lbrace g_1(\o^\X,\o^\Y)=a_1(\o^\Y_1) \,\vert\, a_1 \in \L^1(\Fc^\Y_1,\P^\Y)\rbrace,
    \\
    \Sc_{\X,0}\coloneqq  \lbrace s(\o^\X,\o^\Y)
    =f(\o^\X) \,\vert\, f \in \L^1(\Fc^\X,\P^\X)\rbrace 
    \quad &{\rm and }\quad
    \Sc_{\Y,0}\coloneqq\left\lbrace s(\o^\X,\o^\Y)=g(\o^\Y) \,\vert\, g \in \L^1(\Fc^\Y,\P^\Y) \right\rbrace,
\end{align*} 
permit us to identify certain marginals.
We set
\begin{align}
\begin{split} \label{eqn:dual_potent_cau_bicau}
\Sc_{\X} &\coloneqq\Big\lbrace s(\o^\X,\o^\Y)=f_1(\o^\X_1)+ \sum_{t=2}^T f_t(\o^\X_{1:t},\o^\Y_{1:t-1}) \,\Big\vert\, f_t \in \Ac_{\X,t} \Big\rbrace, \\
\Sc_{\Y} &\coloneqq\Big\lbrace s(\o^\X,\o^\Y)=g_1(\o^\Y_1)+ \sum_{t=2}^T g_t(\o^\X_{1:t-1},\o^\Y_{1:t}) \,\Big\vert\, g_t \in \Ac_{\Y,t} \Big\rbrace.
\end{split}
\end{align}
Finally, define $\Sc^{\rm c}\coloneqq\Sc_{\X} \oplus \Sc_{\Y,0}\coloneqq \{s(\omega^\X,\omega^\Y)=s^\X(\omega^\X,\omega^\Y)+s^\Y(\omega^\X,\omega^\Y) \,\vert\, s^\X \in \Sc_{\X},\;s^\Y \in \Sc_{\Y} \}$ and $\Sc^{\rm bc}\coloneqq\Sc_{\X} \oplus \Sc_{\Y}$.
We remark that a probability measure $\pi$ on $\O^\X\times\O^\Y$ is in $\cplc(\X,\Y)$ if and only if $\int s(\omega^\X,\omega^\Y) \pi(\mathrm{d}\omega^\X, \mathrm{d}\omega^\Y)  = \int s(\omega^\X, \omega^\Y) (\P^\X \otimes \P^\Y)(\mathrm{d}\omega^\X, \mathrm{d}\omega^\Y)$ for all integrable $s \in \Sc^{\rm c}$.
The set of bicausal couplings $\cplbc(\X,\Y)$ and $\Sc^{\rm bc}$ share an analogous relation, see also \Cref{rem:test_fun}.

\medskip We have the following standard duality results for the causal and bicausal optimal transport problems with a lower-semicontinuous cost function. In the sequel, these shall be generalized for a situation with a measurable cost. Note that the condition on being bounded from below of the cost function in \Cref{thm:caus_dual} can be relaxed by standard techniques to being bounded from below by integrable functions.

\begin{theorem}[Duality for a l.s.c.\ cost] \label{thm:caus_dual}
    Let $c : \O^\X \times \O^\Y \longrightarrow \R$ be lower-semicontinuous and such that there are functions $\ell \in \L^1(\P^\X)$ and $k\in \L^1(\P^\Y)$ such that $|c(\o^\X,\o^\Y)|\leq \ell(\o^\X)+k(\o^\Y).$ 
    Then we have
    \begin{align*}
        \Cc\Wc_c(\X,\Y)=& \sup \bigg\{  \int_{\O^\X \times \O^\Y} s(\o^\X,\o^\Y) (\P^\X \otimes \P^\Y)(\mathrm{d}\o^\X,\mathrm{d}\o^\Y) \,\bigg\vert\, s \in \Sc^{\rm c}, \; s \leq c \bigg\}, \\
        \Ac\Wc_c(\X,\Y)=&\sup \bigg\{ \int_{\O^\X \times \O^\Y} s(\o^\X,\o^\Y) (\P^\X \otimes \P^\Y)(\mathrm{d}\o^\X,\mathrm{d}\o^\Y) \,\bigg\vert\, s \in \Sc^{\rm bc}, \;s \leq c \bigg\}.
    \end{align*}
\end{theorem}

\begin{proof} 
    Similar result was proved in \citeauthor*{EcPa22} \cite[Proposition 5.2]{EcPa22}. We note that our set of dual variables is larger than in \cite{EcPa22}, so it suffices to verify weak duality, \emph{i.e.}, the inequality `$\geq$'. We consider the causal optimal transport as the bicausal can be done analogously. Let 
    \begin{equation} \label{eqn:lsc_cost_dual}\Sc^{\rm c} \ni s(\o^\X,\o^\Y)=f_1(\o^\X_1)+ \sum_{t=2}^T f_t(\o^\X_{1:t},\o^\Y_{1:t-1})+g(\o^\Y) \leq c(\o^\X,\o^\Y) \end{equation} be admissible for the dual and let $\pi \in \cplc(\X,\Y)$ be arbitrary. Integrating both sides of $\eqref{eqn:lsc_cost_dual}$ with respect to $K_2^\X(\o^i_{1};\mathrm{d} \o^i_{2:T})$ gives 
    \begin{equation}\label{eqn:lsc_cost_dual2}  f_2(\o^\X_{1:2},\o^\Y_{1}) \leq \int_{\O^\X_{2:T}} c(\o^\X,\o^\Y) K_2^\X(\o^\X_{1};\mathrm{d} \o^\X_{2:T}) -f_1(\o^\X_1)-g(\o^\Y). \end{equation}
    Using causality of $\pi,$ we may find a  disintegration of $\pi,$ see \emph{e.g.} \cite[Definition 2.1 and Remark 2.2]{EcPa22}, of the form $\pi(\mathrm{d}\o^\X,\mathrm d \o^\Y)=\pi^1(\mathrm{d}\o^\X_{1},\mathrm{d}\o^\Y_1) \otimes K^\X_2(\o^\X_1;\mathrm{d}\o^{\X}_2) \otimes K^\pi(\o^\X_{1:2},\o^\Y_1; \mathrm{d}\o^{\X}_{3:T},\mathrm{d}\o^{\Y}_{2:T})$ for suitable kernels.
    
    \medskip From the assumptions, we know that $c \in \L^1(\pi)$ and, by assumptions, $f_1, g \in \L^1(\pi).$ It follows that the right-hand side of \eqref{eqn:lsc_cost_dual2} is integrable with respect to $\pi,$ and so integral $ \int f_2 \mathrm{d} \pi$ exists and lies in $[-\infty, \infty).$ By Fubini's theorem, we then obtain that $\int f_2 \mathrm{d} \pi=0,$ and in particular $f_2 \in \L^1(\pi).$ By repeating the argument above successively for $t=3,\ldots,T,$ we obtain $f_t \in \L^1(\pi)$ and $\int f_t \mathrm d \pi=0$ for every $t \in \{2,\ldots,T\}.$ Finally, integrating \eqref{eqn:lsc_cost_dual} with respect to $\pi$ gives
    \[ \int_{\O^\X \times \O^\Y} c \mathrm{d}\pi\geq \int_{\O^\X \times \O^\Y} s \mathrm{d}\pi= \int_{\O^\X \times \O^\Y} (f_1+g) \mathrm{d}\pi=\int_{\O^\X \times \O^\Y} (f_1+g) \mathrm{d}(\P^\X \otimes \P^\Y)=\int_{\O^\X \times \O^\Y} s \mathrm{d}(\P^\X \otimes \P^\Y),\] where we have also used that $\P^\X \otimes \P^\Y \in \cplc(\X,\Y).$ Since $\pi$ and $s$ were arbitrary, this concludes the proof.
\end{proof}

\begin{remark} \label{rem:test_fun}
Let us point out that, while the integrals on the right-hand sides appear rather involved, it is clear from the proof of {\rm\Cref{thm:caus_dual}} the expression $\int s \mathrm{d}(\P^\X \otimes \P^\Y)$ substantially simplifies in both cases.
Indeed, if $s=f_1+ \sum_{t=2}^T f_t+g \in \Sc^{\rm c}$ is bounded, then we have
\[ \int_{\O^\X \times \O^\Y} s(\o^\X,\o^\Y) (\P^\X \otimes \P^\Y)(\mathrm{d}\o^\X,\mathrm{d}\o^\Y)=\int_{\O_1^\X} f_1(\o^\X_1) \P^\X(\mathrm{d}\o^\X_1)+\int_{\O^\Y} g(\o^\Y) \P^\Y(\mathrm{d}\o^\Y).\] Similarly, if $s=f_1+g_1 +\sum_{t=2}^T f_t+\sum_{t=2}^T g_t \in \Sc^{\rm bc}$, then we have
\[ \int_{\O^\X \times \O^\Y} s(\o^\X,\o^\Y) (\P^\X \otimes \P^\Y)(\mathrm{d}\o^\X,\mathrm{d}\o^\Y)=\int_{\O_1^\X} f_1(\o^\X_1) \P^\X(\mathrm{d}\o^\X_1)+\int_{\O_1^\Y} g_1(\o^\Y_1) \P^\Y(\mathrm{d}\o^\Y_1).\]
\end{remark}

We now proceed with one of the main results of this section---the duality result for general cost function as well as dual attainment for causal optimal transport. Let us emphasize that we require only measurability of the function $c$ together with an integrable upper bound to obtain duality and dual attainment. No further regularity of the cost function nor the transition kernels of $\P^\X$ and $\P^\Y$ is required. As such, our result covers basically any reasonable framework. On the other hand, under such general assumptions, there is obviously no hope to obtain any regularity of the dual potentials apart from suitable integrability. Similarly, it is clear that in general we lose attainment of the primal problem when the cost function is not lower-semicontinuous.

\begin{theorem}[General case---causal transport]\label{existDualC}
    Let $c : \O^\X \times \O^\Y \longrightarrow \R \cup \{ -\infty \}$ be measurable 
and $\ell \in \L^1(\P^\X)$, $k \in \L^1(\P^\Y)$ be such that $c(\o^\X,\o^\Y) \leq \ell(\o^\X) +k(\o^\Y) ,$ $(\o^\X,\o^\Y) \in \O^\X \times \O^\Y$.
    Then we have
    \[
        \Cc\Wc_{c}(\X,\Y) = \sup \bigg\{ \int_{\O^\X \times \O^\Y} s(\o^\X,\o^\Y) (\P^\X \otimes \P^\Y)(\mathrm{d}\o^\X,\mathrm{d}\o^\Y) \,\bigg\vert\, s \in \Sc^{\rm c},  s \leq c \bigg\}. 
    \]
    Moreover, if either side is finite, then there is $s^\star \in \Sc^{\rm c}$ with $s^\star \le c$ and
    \[
        \Cc\Wc_{c}(\X,\Y) = \int_{\O^\X \times \O^\Y} s^\star \mathrm{d}(\P^\X \otimes \P^\Y).
    \]
\end{theorem}

\begin{remark} 
\label{rem:bounded_c} 
    Note that we may without loss of generality assume that $c \le 0$.
    Indeed, let us define 
    \[ 
        \tilde{c}(\o^\X,\o^\Y)\coloneqq c(\o^\X,\o^\Y)-\ell(\o^\X)-k(\o^\Y),\quad (\o^\X,\o^\Y) \in \O^\X \times \O^\Y. 
    \] 
    Then, 
    $\tilde{c}$ is measurable with $\tilde{c} \leq 0$ and $\Cc\Wc_{\tilde{c}}(\X,\Y)=\Cc\Wc_{c}(\X,\Y)-\int_{\O^\X} \ell \mathrm{d}\P^\X - \int_{\O^\Y} k \mathrm{d}\P^\Y$.
    Similarly, we have
    \[ 
        \sup_{\tilde{s} \in \Sc^{\rm c} :\, \tilde{s} \leq \tilde{c} } \int_{\O^\X \times \O^\Y} \tilde{s} \mathrm{d}(\P^\X \otimes \P^\Y)=\sup_{s \in \Sc^{\rm c} :\, s \leq c} \int_{\O^\X \times \O^\Y} s \mathrm{d}(\P^\X \otimes \P^\Y)-\int_{\O^\X} \ell\mathrm{d}\P^\X - \int_{\O^\Y} k\mathrm{d}\P^\Y.
    \] 
    It is thus readily seen that if $\tilde{s} \in \Sc^{\rm c}$ is a solution to the left-hand side, then $s\coloneqq\tilde{s}+\ell+k$ attains the supremum on the right-hand side. 
    Indeed, it holds that $s \le c$. 
    Moreover, to verify that  $s \in \Sc^{\rm c}$, we can proceed by the same construction as in {\rm\citeauthor*{EcPa22} \cite[Remark 5.1]{EcPa22}}. 
\end{remark}

\begin{proof}[Proof of {\rm\Cref{existDualC}}]
    Thanks to \Cref{rem:bounded_c}, we can assume without loss of generality that $c\leq 0$.

    \medskip
    We write $\mathcal C^-$ for the set of continuous, bounded, and non-positive functions on $\O^\X \times \O^\Y$, and $\mathcal M^-$ for the set of measurable and non-positive functions on $\O^\X \times \O^\Y$.
    For $c \in \Mc^-,$ we denote
    \begin{align*}
        D(c) \coloneqq  \sup_{s \in \Sc^{\rm c},\, s \le c } \int_{\O^\X \times \O^\Y} s  \mathrm{d}(\P^\X\otimes\P^\Y), \quad
        V(c) \coloneqq  \Cc\Wc_c(\X,\Y).
    \end{align*}
    We show that $V$ and $D$ are continuous from below on $\mathcal C^-$ and continuous from above on $\mathcal M^-$. That is to say, for any sequence $(c^n)_{n \in \N}$  in $\Cc^-$ such that the convergence $c^n\nearrow c$ holds pointwise for some function $c: \O^\X \times \O^\Y \longrightarrow (-\infty,0],$ we have $D(c^n)\nearrow D(c)$ and $V(c^n)\nearrow V(c)$. Note that in this case $c$ is necessarily bounded and, moreover, lower-semicontinuous since it is a pointwise supremum of continuous functions. Similarly, by the latter we mean that for any sequence $(c^n)_{n \in \N}$ in $\Mc^-$ such that the convergence $c^n\searrow c$ holds pointwise for some (necessarily measurable) function $c: \O^\X \times \O^\Y \longrightarrow [-\infty,0],$ we have $D(c^n)\searrow D(c)$ and $V(c^n)\searrow V(c)$.

    \medskip
   \emph{Step 1} (Continuity of $D$ and $V$ from below):
    Let $(c^n)_{n \in \N}$ be a sequence in $\mathcal C^-$ with $c^n \nearrow c$ for some $c$.
    On the one hand, we have
    \begin{align*}
        \lim_{n \to \infty} D(c^n) = \lim_{n \to \infty} V(c^n)\quad \text{and}\quad V(c) = D(c),
    \end{align*}
    by the standard duality result for lower-semicontinuous cost function, see \Cref{thm:caus_dual}.
    On the other hand, since $\cplc(\X,\Y)$ is weakly compact, it follows by standard arguments that $V(c^n) \nearrow V(c)$, see, for example, \cite[Step 3 in the proof of Theorem 1.3]{Vi03}, which shows continuity from below of $V$ and $D$ on $\mathcal C^-$.

    \medskip
    \emph{Step 2} (Continuity of $V$ from above):
    Let $(c^n)_{n \in \N}$ in $\mathcal M^-$ and $c \in \Mc^-$ be such that $c_n \searrow c$.
    We have
    \begin{equation*}
        \lim_{n \to \infty} V(c^n) = \inf_{n \in \N} \inf_{\pi \in \cplc(\X,\Y)} \int_{\O^\X \times \O^\Y} c^n  \mathrm{d}\pi= \inf_{\pi \in \cplc(\X,\Y)} \inf_{n \in \N} \int_{\O^\X \times \O^\Y} c^n  \mathrm{d}\pi= \inf_{\pi \in \cplc(\X,\Y)} \int_{\O^\X \times \O^\Y} c \mathrm{d} \pi = V(c),
    \end{equation*}
    by monotone convergence, which yields continuity from above of $V$ on $\mathcal M^-$.

    \medskip
     \emph{Step 3} (Main part): Next, we come to the main part of the proof where we show dual attainment as well as continuity from above of $D$ on $\mathcal M^-$.
    To this end, let $(c^n)_{n \in \N}$ be a sequence in $\mathcal M^-$ with $c^n \searrow c$ for some $c \in \Mc^-$ such that $\inf_{n \in \N} D(c^n) > -\infty$. Note that in the case $\inf_{n \in \N} D(c^n) = -\infty,$ we clearly have \[-\infty = \inf_{n \in \N} D(c^n)=\lim_{n \rightarrow \infty}D(c^n)\geq D(c)\geq -\infty.\] Thus continuity from above holds trivially. Now, for $n \in\N$, we write
    \[
        S^n \coloneqq D(c^n) = \sup_{s \in \Sc^{c}  : s \leq c^n} \int_{\O^\X \times \O^\Y} s \mathrm{d}(\P^\X \otimes \P^\Y) \le 0,
    \] 
    and pick a $1/n$-optimizer for $D(c^n),$ say $s^n \in \Sc^{\rm c}$.
    That means we have for all $(\o^\X,\o^\Y) \in \O^\X \times \O^\Y$
    \begin{gather} \label{eqn:beforebound}
         s^n(\o^\X,\o^\Y) = f_1^n(\o^\X_1) + \sum_{t=2}^T f_t^n(\o^\X_{1:t},\o^\Y_{1:t-1})+g^n(\o^\Y) \leq c^n(\o^\X,\o^\Y) \leq 0,
        \\ 
        \label{eqn:bound} 
        \int_{\O^\X \times \O^\Y} \big( f_1^n+g^n\big) \mathrm{d}(\P^\X \otimes \P^\Y) \geq S^n-\frac{1}{n}
        \enspace\text{and}\;
        \lim_{n \rightarrow \infty}\int_{\O^\X \times \O^\Y} \big( f_1^n+g^n\big) \mathrm{d}(\P^\X \otimes \P^\Y) = \inf_{n \in \N} S^n \eqqcolon S.
    \end{gather}

    \emph{Step 3.1} ($t=1$): Integrating both sides of the inequality in \eqref{eqn:beforebound} with respect to the kernel $K^\X_2(\o^\X_1;\mathrm{d}\o^\X_{2:T})$ yields
    \begin{align} \label{bd1}
        f_1^n(\o^\X_1)+g^n(\o^\Y) \leq 0,
    \end{align}
    using that $c^n \le 0$ and that we have for every $t \in \{2,\ldots,T\}$ that $f_t^n \in \Ac_{\X,t}$, thus 
    \[ 
        \int_{\O^\X_{2:T}} f_t^n(\o^\X_{1:t},\o^\Y_{1:t-1}) K^\X_2(\o^\X_1;\mathrm{d}\o^\X_{2:T})=0.
    \]
    Moreover, we can assume without loss of generality that $ \int f^n_1\mathrm{d}\P^\X=\int g^n \mathrm{d}\P^\Y\eqqcolon I^n,\; n \in \N$. 
    Indeed, if this is not the case, we can subtract a constant from $f^n_1$ and add the same constant to $g^n$, which clearly does not change admissibility in the dual or value when integrated against $\P^\X \otimes \P^\Y$.
    \medskip

    It follows from \eqref{eqn:bound} that  $S^n \geq 2 \cdot I^n \geq S-1$  for every $n \in \N$.
    By integrating both sides of \eqref{bd1} with respect to $\P^\Y$ and $\P^\X$ respectively we find that
    \begin{align*}
    f_1^n(\o^\X_1)\leq - I^n \leq - \frac{S-1}{2} \quad \text{and} \quad g^n(\o^\Y)\leq - I^n \leq -\frac{S-1}{2}.
    \end{align*} 
    Hence, using the preceding inequalities we obtain
    \[ \int_{\O^\X_1} |f_1^n| \mathrm{d}\P^\X=\int_{\O^\X_1} [2 (f_1^n)^+ -  f_1^n ]\mathrm{d}\P^\X\leq -2\frac{S-1}{2} - \frac{S-1}{2} = -\frac32(S-1). \] 
    By applying analogous arguments to $g^n$ we conclude
    \[ 
        \sup_{n \in \N} \lVert f_1^n \rVert_{\L^1(\Fc^\X_1,\P^\X)}< \infty\quad\text{and}\quad \sup_{n \in \N} \lVert g^n \rVert_{\L^1(\Fc^\Y,\P^\Y)}< \infty. 
    \]
    Using Koml\'{o}s' lemma, \cite[Theorem 1]{Ko67}, we can without loss of generality, up to taking the Ces\`aro means of a subsequence, assume that there are functions $\tilde f_1 \in \L^1(\Fc^\X_1,\P^\X)$ and $\tilde g \in \L^1(\Fc^\Y,\P^\Y)$ such that $f_1^n \longrightarrow \tilde f_1$ on a $\P^\X$--full Borel set $A^1 \in \Fc_1^\X$ and $g^n \longrightarrow \tilde g$ on a $\P^\Y$--full Borel set $B \in \Fc^\Y$. Indeed, note that $\Sc^{\rm c}$ is closed under taking averages.
    Taking Ces\`aro means of subsequences of all the remaining sequences of functions thus preserves their properties as well as the bound \eqref{eqn:beforebound}.
    Consequently, for simplicity we can without loss of generality assume that the original sequences $(f_1^n)_{n \in \N}$ and $(g^n)_{n \in \N}$ were convergent. 
    It follows from Fatou's lemma, which can be used thanks to the bound \eqref{bd1}, that
    \begin{align*}
     \int_{\O^\X_1} \tilde f_1 \mathrm{d}\P^\X+ \int_{\O^\Y} \tilde g \mathrm{d}\P^\Y= \int_{\O^\X \times \O^\Y} \limsup_{k \rightarrow \infty} (f^k_1+g^k) \mathrm{d}(\P^\X \otimes \P^\Y)\geq \limsup_{k \rightarrow \infty} \int_{\O^\X \times \O^\Y} (f^k_1+g^k) \mathrm{d}(\P^\X \otimes \P^\X)=S.
    \end{align*}
    As it will be convenient in the later steps of the proof, note that for every $t \in \{1,\ldots,T-1\}$ the map 
    \begin{equation} \label{eqn:bound_y} 
        \o^\Y_{1:t} \longmapsto \inf_{k \in \N} \sup_{\tilde \o^\Y \in B,\, \tilde \o^\Y_{1:t} = \o^\Y_{1:t}} g^k(\tilde \o^\Y)
    \end{equation} 
    is upper-semianalytic, see \cite[Proposition 7.47]{BeSh78}.
    Therefore, there exists a Borel subset $B^T$ of $B$ that is $\P^\Y$--full, $B^t \coloneqq {\rm proj}^{1:t}(B^T)$ is Borel measurable for every $t \in \{1,\ldots,T-1\}$, and \eqref{eqn:bound_y} is Borel measurable restricted to $B^t$. 
    Indeed, this set can be constructed as follows: Set $\hat B^{T+1} \coloneqq B$.
    By applying \cite[Lemma 7.27]{BeSh78} backwards for every $t \in \{T,\ldots, 1\},$ which gives that the restriction of \eqref{eqn:bound_y} to a $\P^\Y$--full Borel subset $\hat B^t$ of ${\rm proj}^{1:t}(\hat B^{t + 1}) \subseteq \O^\Y_{1:t}$ is Borel measurable. Taking intersections of the corresponding Borel sets $\hat B^t \times \O^\Y_{t+1:T}$ over $t$ gives $B^T.$
    We replace $\tilde f_1$ and $\tilde g$ by
    \[
        f_1(\o^\X_1) \coloneqq 
        \begin{cases}
            \tilde f_1(\o^\X_1) & \o^\X_1 \in A^1, \\
            -\infty & \text{else,}
        \end{cases}
        \text{ and }\;
        g(\o^\Y) \coloneqq
        \begin{cases}
            \tilde g(\o^\Y) & \o^\Y \in B^T, \\
            -\infty & \text{else}.
        \end{cases}
    \]

    It remains to show that $f_1$ and $g$ are admissible in the sense that there exist functions $f_t \in \Ac_{\X,t}$, $t \in \{2,\ldots,T \}$, such that $\Sc^{\rm c} \ni f_1(\o^\X_1)+ \sum_{t=2}^T f_t(\o^\X_{1:t},\o^\Y_{1:t-1})+g(\o^\Y) \leq c(X(\o^\X),Y(\o^{\Y}))$. We proceed by induction in $t$, showing that there  there exist convex combinations of $(f_t^n)$ converging to a limit $a_t.$ 

    \medskip
    \emph{Step 3.2} ($t=2$):
    Clearly, when $\o^\X_1 \notin A^1$ or $\o^\Y \notin B^T$ we can simply set $f_t(\o^\X_{1:t},\o^\Y_{1:t-1}) = 0$ for every $t \in \{2,\dots,T\}$ without interfering with admissibility.
    Therefore, the aim of this step is to show that there exist convex combinations of $(f_2^n)_{n \in \N}$ converging to some suitable limit whenever $(\o^\X_1,\o^\Y) \in A^1 \times B^T$.
    Integrating both sides in \eqref{eqn:beforebound} with respect to $K^\X_3(\o^\X_2;\mathrm{d}\o^\X_{3:T})$ gives 
    \begin{gather} \label{bound1} 
    f_1^n(\o^\X_1)+ f_2^n(\o^\X_{1:2},\o^\Y_1)+g^n(\o^\Y) \leq 0.
    \end{gather}
    In particular, we have
    \[    
        f_2^n(\o^\X_{1:2},\o^\Y_1) \leq  
        -\inf_{k \in \N} g^k(\o^\Y) 
        -\inf_{k \in \N} f_1^k(\o^\X_1)<\infty. 
    \]
    It is clear that $\inf_{k \in \N} f_1^k(\o^\X_1)$ is Borel measurable and finitely-valued on $A^1$, while (for $t = 1$) \eqref{eqn:bound_y} is also Borel measurable on $B^1$.
    We set
    \[
        \O^1 \coloneqq A^1 \times B^1.
    \]
    Hence, there is a Borel measurable function $C_2 : \O^1 \longrightarrow \R$ such that for all $\o^\X_2 \in \O^\X_2$
    \begin{equation} \label{fl} 
        f_2^n(\o^\X_{1:2},\o^\Y_1) \leq  C_2(\o^\X_1,\o^\Y_1).
    \end{equation}
    We invoke \Cref{komlos} applied to $x=(\omega^\X_1,\omega^\Y_1), y= \omega^\X_2,$ $Y^n=-f_2^n,$ $n \in \N,$ and $\P(x)(\mathrm{d}y)=K_2(\omega_1^\X;\mathrm{d}\omega_2^\X)$ with $C(x,y)=-C_2(\o^\X_1,\o^\Y_1)$ to conclude that there exist convex combinations, for $n \in \N$,
    \[ 
        a_2^{n,2}(\o^\X_{1},\,\cdot\,,\o^\Y_1) \in \text{conv}\left(f_2^n(\o^\X_{1},\,\cdot\,,\o^\Y_1),f_2^{n+1}(\o^\X_{1},\,\cdot\,,\o^\Y_1),\ldots \right), 
    \]
    with coefficients depending in a measurable way on $(\o^\X_1,\o^\Y_1) \in \O^1$ such that for each $(\o^\X_1,\o^\Y_1) \in \O^1 $ the limit $\lim_{n} a^{n,2}_2(\o^\X_1,\,\cdot\,,\o^\Y_1)$ exists $K^\X_2(\o^\X_1;\,\cdot\,)$--almost surely.
    We denote by $A^2_{\o^\X_1,\o^\Y_1} \subseteq \O^\X_{2}$ the set where this limit exists in $\R$ and define
    \begin{equation}
        \label{causal:a_2}
        a_2(\o^\X_{1:2},\o^\Y_1) \coloneqq 
        \begin{cases}
            \lim_n a^{n,2}_2(\o^\X_{1:2},\o^\Y_1) & (\o^\X_1,\o^\Y_1) \in \O^1, \o^\X_2 \in A^2_{\o^\X_1,\o^\Y_1}, \\
            -\infty & (\o^\X_1,\o^\Y_1) \in \O^1, \o^\X_2 \notin A^2_{\o^\X_1,\o^\Y_1}, \\
            0 & \text{else}.
        \end{cases}
    \end{equation}
    As \eqref{fl} provides an upper bound also to $a_2^n$, we can apply Fatou's lemma to obtain on $\O^1$
    \begin{multline}\label{mart1} 
        0=\limsup_{n \rightarrow \infty} \int_{\O^\X_2} a_2^n(\o^\X_{1:2},\o^\Y_1)K_2^\X(\o^\X_1;\mathrm{d}\o^\X_2)
        \\
        \leq 
        \int_{\O^\X_2} \limsup_{n \rightarrow \infty}  a_2^n(\o^\X_{1:2},\o^\Y_1) K_2^\X(\o^\X_1;\mathrm{d}\o^\X_2)
        =
        \int_{\O^\X_2} \tilde a_2(\o^\X_{1:2},\o^\Y_1)K_2^\X(\o^\X_1;\mathrm{d}\o^\X_2).
    \end{multline}
    Finally, we denote the same convex combinations of $f_1^{n}(\o^\X_1)$, $g^n(\o^\Y)$, resp.\ $c^n(\o^\X,\o^\Y),$ by $f_1^{n,2}(\o^\X_1,\o^\Y_1)$, $g^{n,2}(\o^\X_1,\o^\Y)$, resp.\ $c^{n,2}(\o^\X,\o^\Y)$.
    Since the coefficients depend measurably on $(\o^\X_1,\o^\Y_1),$ this doesn't interfere with the convergence of the sequences.
    Define the set
    \[
        \O^2 \coloneqq \{ (\o^\X_{1:2},\o^\Y_{1:2}) \in \O^{\X}_{1:2} \times \O^{\Y}_{1:2} \,\vert\, (\o^\X_1,\o^\Y_1) \in \O^1, \o^\X_2 \in A^2_{\o_1^\X,\o_1^\Y}, \o^\Y_{1:2} \in B^2 \}.
    \]
    Note that 
    \[
        \O^2=\{(\o^\X_{1:2},\o^\Y_{1:2}) \in \O^{\X}_{1:2} \times \O^{\Y}_{1:2} \,\vert\,(\o^\X_1,\o^\Y_1) \in \O^1, \lim_{n \rightarrow \infty} a^{n,2}_2(\o^\X_1,\cdot,\o^\Y_1) \text{ exists in }\R, \o^\Y_{1:2} \in B^2 \},
    \] 
    from which it's clear that $\O^2$ is Borel. In summary, we have found Borel measurable functions $f^{n,2}_1,a^{n,2}_2,a_2$, $(f^{n,2}_t)_{t = 3}^T \in \Ac_{\X,3:T},g^{n,2}, c^{n,2}$ and a Borel set $\O^2$ such that
    \begin{equation}
        f^{n,2}_1 + a_2^{n,2} + \sum_{t = 3}^T f_t^{n,2} + g^{n,2} \le c^{n,2} \le 0,
    \end{equation}
    and for $n \longrightarrow \infty$ and $(\o^\X,\o^\Y) \in \O^\X \times B^T$ with $(\o^\X_{1:2},\o^\Y_{1:2}) \in \O^2$ we have
    \begin{gather*}
        f^{n,2}_1(\o^\X_1,\o^\Y_1) \longrightarrow f_1(\o^\X_1), \quad
        a^{n,2}_2(\o^\X_{1:2},\o^\Y_1) \longrightarrow a_2(\o^\X_{1:2},\o^\Y_1),
        \\
        g^{n,2}(\o^\X_1,\o^\Y) \longrightarrow g(\o^\Y), \quad 
        c^{n,2}(\o^\X,\o^\Y) \longrightarrow c(\o^\X,\o^\Y),\\
        0 \le \int_{\O^\X_2} a_2(\o^\X_1,\tilde \o^\X_2, \o^\Y_1) \, K_2^\X(\o^\X_1; \mathrm d\tilde \o^\X_2) < \infty.
    \end{gather*}

    \medskip
    \emph{Step 3.3} ($t \mapsto t + 1$): In this step, we inductively construct convergent convex combinations building on the previous step.
    This means, we assume that we have found Borel functions $f^{n,t}_1, (a^{n,t}_s)_{s = 2}^t, (a_s)_{s = 2}^t$, $(f^{n,t}_s)_{s = t + 1}^T, g^{n,t}, c^{n,t}$ and a Borel set
    \[
        \O^t = \{ (\o^\X_{1:t},\o^\Y_{1:t}) \in \O^\X_{1:t} \times \O^\Y_{1:t} \,\vert\, (\o^\X_{1:t-1},\o^\Y_{1:t-1}) \in \O^{t-1}, \o^\X_t \in A^t_{\o^\X_{1:t-1},\o^\Y_{1:t-1}} \o^\Y_{1:t} \in B^t \},
    \]
    where for fixed $(\o^\X_{1:t-1},\o^\Y_{1:t-1}) \in \O^{t-1}$ the slice $A^t_{\o^\X_{1:t-1},\o^\Y_{1:t-1}} = \{ \o^\X_t \in \O^\X_t \,\vert\, \lim_{n} a^{n,t}_t(\o^\X_{1:t},\o^\Y_{1:t-1}) \text{ exists}\}$ is a $K^\X_{t}(\o^\X_{1:t};\,\cdot\;)$--full set, such that
    \begin{equation}
        \label{bd3a}
        f^{n,t}_1 + \sum_{s = 2}^t a^{n,t}_s + \sum_{s = t + 1}^T f^{n,t}_s + g^{n,t} \le c^{n,t} \le 0,
    \end{equation}
    and for $n \longrightarrow \infty$, $(\o^\X,\o^\Y) \in \O^\X \times B^T$ with $(\o^\X_{1:t}, \o^\Y_{1:t}) \in \O^t$, and $s \in \{2,\dots,t\}$ we have
    \begin{gather*}
        f^{n,t}_1(\o^\X_{1:t-1},\o^\Y_{1:t-1}) \longrightarrow f_1(\o^\X_1), \quad
        a^{n,t}_s(\o^\X_{1:\max(s,t-1)},\o^\Y_{1:t-1}) \longrightarrow a_s(\o^\X_{1:s},\o^\Y_{1:s-1})
        \\
        g^{n,t}(\o^\X_1,\o^\Y) \longrightarrow g(\o^\Y), \quad c^{n,t}(\o^\X,\o^\Y) \longrightarrow c(\o^\X,\o^\Y),\\
        0 \le \int_{\O^\X_s} a_s(\o^\X_{1:s-1},\tilde \o^\X_s, \o^\Y_{1:s-1}) \, K_s^\X(\o^\X_{1:s-1};\mathrm d\tilde \o^\X_s) < \infty.
    \end{gather*}
    \medskip
    Integrating \eqref{bd3a} with respect to $K^\X_{t+2}(\o^\X_{1:t+1};\mathrm{d}\o^\X_{t+2:T})$ yields
    \[
        f_1^{n,t} + \sum_{s=2}^{t} a_s^{n,t} + f_{t + 1}^{n,t} + g^{n,t} \leq 0.
    \]
    Hence,
    \begin{align*}
        f_{t+1}^{n,t}&(\o^\X_{1:t+1},\o^\Y_{1:t}) \\
        &\leq - \Big(f_1^{n,t}(\o^\X_{1:t-1},\o^\Y_{1:t-1}) + \sum_{s=2}^{t} a_s^{n,t}(\o^\X_{1:\max(s,t-1)},\o^\Y_{1:t-1}) +g^{n,t}(\o^\X_{1:t-1},\o^\Y) \Big) \\
        & \leq - \inf_{k \in \N}\Big( f_1^{k,t}(\o^\X_{1:t-1},\o^\Y_{1:t-1})+ \sum_{s=2}^{t} a_s^{k,t}(\o^\X_{1:\max(s,t-1)},\o^\Y_{1:t-1}) \Big) - \inf_{k \in \N} g^{k}(\o^\Y),
    \end{align*}
    where we used for the last inequality that $g^{n,t}(\o^\X_{1:t-1},\o^\Y)$ was constructed as a convex combination of $(g^k(\o^\Y))_{k \in \N}$.
  As in the previous step, we find a Borel measurable function $C_t \colon \O^{t} \to \R$ such that for all $(\o^\X_{1:t},\o^\Y_{1:t}) \in \O^{t}$ and $\o^\X_{t + 1} \in \Omega^\X_{t + 1}$ we have
    \begin{equation} \label{eqn:bound_t}f_{t + 1}^{n,t}(\o^\X_{1:t + 1},\o^\Y_{1:t}) \leq C_t(\o^\X_{1:t},\o^\Y_{1:t}),\quad n \in \N.
    \end{equation}
    Therefore, we once again exploit \Cref{komlos} applied to $x=(\omega^\X_{1:t},\omega^\Y_{1:t}), y= \omega^\X_{t + 1},$ $Y^n=-f^{n,t}_{t + 1},$ $n \in \N,$ and $\P(x)(\mathrm{d}y)=K^\X_{t + 1}(\omega_{1:t}^\X;\mathrm{d}\omega_{t + 1}^\X)$ with $C(x,y)=-C_t(\o^\X_{1:t},\o^\Y_{1:t})$ and proceed with the same construction as in the previous step. 
    Hence, there is a sequence of convex combinations
    \[
        a_{t + 1}^{n,t + 1}(\o^\X_{1:t},\,\cdot\,,\o^\Y_{1:t}) \in {\rm conv}(f_{t + 1}^{n,t}(\o^\X_{1:t},\,\cdot\,,\o^\Y_{1:t}),f_{t + 1}^{n+1,t}(\o^\X_{1:t},\,\cdot\,,\o^\Y_{1:t}),\ldots)
    \] 
    with coefficients depending measurably on $(\o^\X_{1:t},\o^\Y_{1:t})$ such that for each $(\o^\X_{1:t},\o^\Y_{1:t}) \in \O^t$ we have that the limit $\lim_n a^{n,t + 1}_{t + 1}(\o^\X_{1:t},\cdot,\o^\Y_{1:t})$ exists $K^\X_{t + 1}(\o^\X_{1:t};\,\cdot\,)$-almost surely. 
    Denote the set where the limit exists by $A^{t + 1}_{\o^\X_{1:t},\o^\Y_{1:t}} \subseteq \O^\X_{t + 1}$ and define
    \[
        a_{t + 1}(\o^\X_{1:t + 1},\o^\Y_{1:t}) \coloneqq
        \begin{cases}
            \lim_n a^{n,t + 1}_{t + 1}(\o^\X_{1:t + 1},\o^\Y_{1:t}) & (\o^\X_{1:t},\o^\Y_{1:t}) \in \O^t, \o^\X_{t + 1} \in A^{t + 1}_{\o^\X_{1:t},\o^\Y_{1:t}}, \\
            -\infty & (\o^\X_{1:t},\o^\Y_{1:t}) \in \O^t, \o^\X_{t + 1} \notin A^{t + 1}_{\o^\X_{1:t},\o^\Y_{1:t}}, \\
            0 & \text{else}.
        \end{cases}
    \]
    Due to Fatou's lemma we have on $\O^{t}$
    \begin{multline}\label{mart2} 
        0=\limsup_{n \rightarrow \infty} \int_{\O^\X_{t+1}} a_{t + 1}^{n, t + 1}(\o^\X_{1:t},\tilde \o^\X_{t + 1},\o^\Y_{1:t})K_t^\X(\o^\X_{1:t};\mathrm{d}\tilde \o^\X_{t + 1})\\
        \leq \int_{\O^\X_{t+1}}  \limsup_{n \rightarrow \infty} a_{t + 1}^{n, t + 1}(\o^\X_{1:t},\tilde \o^\X_{t + 1}, \o^\Y_{1:t})K_t^\X(\o^\X_{1:t};\mathrm{d}\tilde \o^\X_{t + 1})= \int_{\O^\X_{t+1}} a_{t + 1}(\o^\X_{1:t},\tilde \o^\X_{t + 1},\o^\Y_{1:t})K_t^\X(\o^\X_{1:t};\mathrm{d}\tilde \o^\X_{t + 1}).
    \end{multline} We again denote the convex combinations with coefficients as above of $f_1^{n,t},g^{n,t},$ $(a_{s}^{n,t})_{s = 2}^t,$ $(f_s^{n, t})_{s = t + 1}^T$ and $c^{n,t}$ by $f_1^{n,t + 1},g^{n,t + 1}, (a_{s}^{n,t + 1})_{s = 2}^{t + 1},(f_s^{n, t + 1})_{s = t + 2}^T$ and $c^{n,t + 1},$ respectively.
    Since the coefficients depend measurably on $(\o^\X_{1:t},\o^\Y_{1:t}),$ this doesn't interfere with the convergence of the sequences.
    So, we conclude this step by defining
    \[
        \O^{t + 1} \coloneqq \{ (\o^\X_{1:t+1},\o^\Y_{1:t+1}) \in \O^\X_{1:t+1} \times \O^\Y_{1:t+1} \,\vert\, (\o^\X_{1:t},\o^\Y_{1:t}) \in \O^t, \o^\X_{t + 1} \in A^{t + 1}_{\o^\X_{1:t},\o^\Y_{1:t}}, \o^\Y_{1:t + 1} \in B^{t + 1} \},
    \]
    which is as in the previous step Borel, and observe that for fixed $(\o^\X_{1:t},\o^\Y_{1:t}) \in \O^t$ the slice $A^{t + 1}_{\o^\X_{1:t},\o^\Y_{1:t}} = \{ \o_{t + 1}^\X \in \O^\X_{t+1} \,\vert\, \lim_n a^{n, t + 1}_{t + 1} \text{ exists}\}$ is $K^\X_{t + 1}(\o^\X_{1:t};\,\cdot\,)$--full,
    \[
        f^{n,t + 1}_1 + \sum_{s = 2}^{t + 1} a_s^{n,t+1} + \sum_{s =  t +2}^T f_s^{n,t+1} + g^{n,t+1} \le c^{n,t + 1} \le 0,
    \]
    and for $n \longrightarrow \infty$, $(\o^\X,\o^\Y) \in \O^\X \times B^T$ with $(\o^\X_{1:t+1}, \o^\Y_{1:t+1}) \in \O^{t + 1}$, and $s \in \{2,\dots,t + 1\}$ we have
    \begin{gather*}
        f^{n,t+1}_1(\o^\X_{1:t},\o^\Y_{1:t}) \longrightarrow f_1(\o^\X_1), \quad
        a^{n,t + 1}_s(\o^\X_{1:\max(s,t)},\o^\Y_{1:t}) \longrightarrow a_s(\o^\X_{1:s},\o^\Y_{1:s-1})
        \\
        g^{n,t + 1}(\o^\X_1,\o^\Y) \longrightarrow g(\o^\Y), \quad c^{n, t + 1}(\o^\X,\o^\Y) \longrightarrow c(\o^\X,\o^\Y),\\
        0 \le \int_{\O^\X_s} a_s(\o^\X_{1:s-1},\tilde \o^\X_s, \o^\Y_{1:s-1}) \, K_s^\X(\o^\X_{1:s-1};\mathrm d\tilde \o^\X_s)<\infty.
    \end{gather*}
    
    \medskip \emph{Step 3.4}: At the end, we have constructed convex combinations 
    satisfying
    \begin{equation}
        \label{eq:causal.step.3.4.ineq}
        f_1^{n,T} + \sum_{t = 2}^T a^{n,T}_t + g^{n,T} \le c^{n,T} \le 0,
    \end{equation}
    and found a Borel set $\O^T$ given by
    \[
        \O^T = \{ (\o^\X,\o^\Y) \in \O^\X \times \O^\Y \,\vert\, \forall t \in \{1,\dots,T\}, \o^\X_t \in A^t_{\o^\X_{1:t-1},\o^\Y_{1:t-1}}, \o^\Y \in B^T \},
    \]
    where on $\O^T$ the slices $A^t_{\o^\X_{1:t-1},\o^\Y_{1:t-1}}$ are $K^\X_t(\o^\X_{1:t-1};\,\cdot\;)$--full and for $n \longrightarrow \infty$ and $t \in \{1,\ldots,T\}$
    \begin{gather} \nonumber
        f^{n,T}_1(\o^\X_{1:T-1},\o^\Y_{1:T-1}) \longrightarrow f_1(\o^\X_1), \quad
        a^{n,T}_t(\o^\X_{1:\max(t,T-1)},\o^\Y_{1:T-1}) \longrightarrow a_t(\o^\X_{1:t},\o^\Y_{1:t-1})
        \\ \nonumber
        g^{n,T}(\o^\X_{1:T-1},\o^\Y) \longrightarrow g(\o^\Y), \quad c^{n,T}(\o^\X,\o^\Y) \longrightarrow c(\o^\X,\o^\Y),
        \\
        \label{mart3}
        0 \le \int_{\O^\X_t} a_t(\o^\X_{1:t-1},\tilde \o^\X_t, \o^\Y_{1:t-1}) \, K_t^\X(\o^\X_{1:t-1};\mathrm d\tilde \o^\X_t) < \infty.
    \end{gather}
    Passing to the limit in \eqref{eq:causal.step.3.4.ineq} yields on $\O^T$
    \[ 
        f_1(\o^\X_1)+\sum_{t=2}^T a_t(\o^\X_{1:t},\o^\Y_{1:t-1})+g(\o^\Y)\leq c(\o^\X,\o^\Y), 
    \]
    while on the complement of $\O^T$ the left-hand side is $-\infty$.
    Thanks to \eqref{mart3} we get
    \[ 
        f_1(\o^\X_1)+\sum_{t=2}^T \bigg[a_t(\o^\X_{1:t},\o^\X_{1:t-1})-\int_{\O^\X_t} a_t(\o^\X_{1:t},\o^\Y_{1:t-1})K^\X_t(\o^\X_{1:t-1};\mathrm{d}\o^\X_t)\bigg]+g(\o^\Y)\leq c(X(\o^\X),Y(\o^\Y)).
    \]
    
    \medskip We have hence constructed admissible functions that attain $D(c)$ and $\lim_{n \to \infty} D(c^n) = S = D(c)$.

\medskip
    \emph{Step 4} (Choquet): Finally, we can invoke the functional version of the Choquet capaticability theorem provided in \cite[Proposition 2.1]{BaChKu19} and obtain
    \[
        D(c) = \inf \{ D(\hat c) \,\vert\, \hat c \ge c,\; \hat c \text{ is l.s.c.\ and lower-bounded} \}
        = \inf \{ V(\hat c) \,\vert\, \hat c \ge c,\; \hat c \text{ is l.s.c.\ and lower-bounded} \}
        = V(c),
    \]
    where the second equality is due to the standard duality for causal optimal transport with lower-semicontinuous cost function, see \Cref{thm:caus_dual}.
    This concludes the proof.
\end{proof}

Next, we give a parallel result for duality and dual attainment for adapted optimal transport.
Similarly to the causal case, we require only minimal assumptions. 
The proof follows similar steps to those of \Cref{existDualC}, but each step in the induction is divided into two further sub-steps.

\begin{theorem}[General case---adapted transport] \label{existDualBC} 
    Let $c : \O^\X \times \O^\Y \longrightarrow \R \cup \{ -\infty \}$ be measurable
    and $\ell \in \L^1(\P^\X)$, $k \in \L^1(\P^\Y)$ be such that $c(\o^\X, \o^\Y) \leq \ell(\o^\X)+k(\o^\Y),$ $(\o^\X,\o^\Y) \in \O^\X \times \O^\Y$.
    Then we have
    \[
        \Ac\Wc_{c}(\X,\Y) = \sup \bigg\{ \int_{\O^\X \times \O^\Y} s(\o^\X,\o^\Y) (\P^\X \otimes \P^\Y)(\mathrm{d}\o^\X,\mathrm{d}\o^\Y) \,\bigg\vert\, s \in \Sc^{\rm bc},\;  s \leq c  \bigg\}. 
    \]
    Moreover, if either side is finite, then there is $s^\star \in \Sc^{\rm bc}$ with $s^\star \le c$ and
    \[
        \Ac\Wc_{c}(\X,\Y) = \int_{\O^\X \times \O^\Y} s^\star \mathrm{d}(\P^\X \otimes \P^\Y).
    \]
\end{theorem}

\begin{proof}
The proof follows analogous arguments as in the causal case.
We therefore provide a sketch of the main steps. 
Assume without loss of generality that the cost function is non-positive. 
The general case can be recovered in the same way as before.

\medskip
    \emph{Step 1 and 2} (Continuity of $D$ and $V$ from below and of $V$ from above): We denote
    \begin{align*}
        D(c) \coloneqq  \sup_{s \in \Sc^{\rm bc},\, s \le c} \int_{\O^\X \times \O^\Y} s  \mathrm{d}(\P^\X\otimes\P^\Y), \quad
        V(c) \coloneqq  \Ac\Wc_c(\X,\Y).
    \end{align*}
    In the first step we show that $V$ and $D$ are continuous from below on $\mathcal C^-$ and $V$ is continuous from above on $\mathcal M^-$. This can be done in the very same fashion as before.

\medskip
    \emph{Step 3} (Main part): Next, we show that $D$ is continuous from above on $\Mc^-$ and $D(c)$ is attained. To that end, let $(c^n)_{n \in \N}$ be a sequence in $\Mc^-$ such that $c^n \searrow c$ for some $c \in \Mc^-$ such that $\inf_{n \in \N} D(c^n)>-\infty$. Let us denote 
    \[ 
        S^n \coloneqq D(c^n)= \sup_{s \in \Sc^{\rm bc},\, s \le c} \int_{\O^\X \times \O^\Y} s  \mathrm{d}(\P^\X\otimes\P^\Y),\;{\rm and}\; S \coloneqq \inf_{n \in \N} S^n. 
    \]

Let for any $n \in \N,$ let $s^n \in \Sc^{\rm bc}$, be an $1/n$-maximizer for $D(c^n)$. This means that we have for all $(\o^\X,\o^\Y) \in \O^\X \times \O^\Y$
 \begin{align*}
    s^n(\o^\X,\o^\Y)=f_1^n(\o^\X_1) + \sum_{t=2}^T f_t^n(\o^\X_{1:t},\o^\Y_{1:t-1})+g^n_1(\o^\Y_1) + \sum_{t=2}^T g_t^n(\o^\X_{1:t-1},\o^\Y_{1:t}) \leq c^n(\o^\X,\o^\Y), \\
    \int_{\O^\X_1 \times \O^\Y_1} \big( f_1^n+g^n_1\big) \mathrm{d}(\P^\X \otimes \P^\Y) \geq S^n-\frac{1}{n},\; \text{and, moreover,}\;
    \lim_{n \rightarrow \infty} \int_{\O^\X_1 \times \O^\Y_1} \big( f_1^n+g^n_1\big) \mathrm{d}(\P^\X \otimes \P^\Y)= S.
\end{align*}
\medskip
 \emph{Step 3.1} ($t=1$): Integrating the following inequality 
\begin{align}\label{boundBC} f_1^n(\o^\X_1) + \sum_{t=2}^T f_t^n(\o^\X_{1:t},\o^\Y_{1:t-1})+g^n_1(\o^\Y_1) + \sum_{t=2}^T g_t^n(\o^\X_{1:t-1},\o^\Y_{1:t})\leq c^n(\o^\X, \o^\Y) \leq 0,
\end{align} 
with respect to $K^\X_2(\o^\X_1;\mathrm{d}\o^\X_{2:T})\otimes K^\Y_2(\o^\Y_1;\mathrm{d}\o^\Y_{2:T})$ yields 
$f_1^n(\o^\X_1)+g^n_1(\o^\Y_1) \leq 0$.
As in the proof of Theorem \ref{existDualC}, we can bound the $\L^1(\Fc^\X_1,\P^\X),$ resp.\ $\L^1(\Fc^\Y_1,\P^\Y)$-norms of the respective sequences and apply Koml\'{o}s' lemma to find, up to taking the Ces\`aro means of a subsequence, functions $f_1 \in \L^1(\Fc^\X_1,\P^\X)$ and $g_1 \in \L^1(\Fc^\Y_1,\P^\Y)$ with $f_1^n \longrightarrow f_1$ on a $\P^\X\text{--full}$ Borel set $A^1$ and $g^n_1 \longrightarrow g_1$ on a $\P^\Y\text{--full}$ Borel set $B^1$, while $f_1$ and $g_1$ are both set to $-\infty$ outside of these sets. We proceed by induction in $t$.

\medskip
 \emph{Step 3.1} ($t=2$): $(a)$ 
Integrating both sides of \eqref{boundBC} with respect to $K^\X_3(\o^\X_{1:2};\mathrm{d}\o^\X_{3:T})\otimes K^\Y_2(\o^\Y_1;\mathrm{d}\o^\Y_{2:T})$ gives \[f_1^n(\o^\X_1)+ f_2^n(\o^\X_{1:2},\o^\Y_1)+g^n_1(\o^\Y_1) \leq 0.\]
Hence, as in the proof of Theorem \ref{existDualC}, we get convex combinations 
\[ 
    a_2^{\X,n}(\o^\X_{1},\cdot,\o^\Y_1) \in \text{conv}\left(f_2^n(\o^\X_{1},\cdot,\o^\Y_1),f_2^{n+1}(\o^\X_{1},\cdot,\o^\Y_1),\ldots \right) 
\] 
with coefficients depending measurably on $(\o^\X_1,\o^\Y_1)$ such that for each $(\o^\X_1,\o^\Y_1) \in A^1 \times B^1 \eqqcolon \O^1$ the limit $\lim_n a_2^{\X,n}$ exists on a $K^\X_2(\o^\X_1;\,\cdot\;)$--full set $A^2_{\o^\X_1,\o^\Y_1}$.
Next, define $a_2$ as in \eqref{causal:a_2}.
Once again, we can argue that taking the very same convex combinations of the remaining sequences of functions, including the sequence $(c^n)_{n \in \N},$ doesn't change their convergence properties nor other relevant characteristics.

\medskip $(b)$ Integrating both sides of \eqref{boundBC} with respect to $K^\X_3(\o^\X_{1:2};\mathrm{d}\o^\X_{3:T})\otimes K^\Y_3(\o^\Y_{1:2};\mathrm{d}\o^\Y_{3:T})$ gives 
\begin{align*} f_1^n(\o^\X_1)+ f_2^n(\o^\X_{1:2},\o^\Y_1)+g^n_1(\o^\Y_1)+g_2^n(\o^\X_{1},\o^\Y_{1:2}) \leq 0. 
\end{align*} 
Replacing $f^n_1, f^n_2, g_1^n, g^n_2$ with their convex combinations with the very same weights as in $(a)$ preserves this inequality. In particular, we can find a suitable upper bound for the sequence $f_2^n$ and, thus, we again find convex combinations $a_2^{\Y,n}(\o^\X_{1},\o^\Y_{1},\,\cdot\,)$ with coefficients depending measurably on $(\o^\X_1,\o^\Y_1)$ such that on $\O^1$ the limit $\lim_n a_2^{\Y,n}$ exists on a $K^\Y_2(\o^\Y_1;\,\cdot\,)$--full set $B^2_{\o^\X_1,\o^\Y_1}$.
We can thus take the same convex combinations of the remaining sequences of functions, define $a_2^\Y$ analogously as in \eqref{causal:a_2} using $B^2_{\o^\X_1,\o^\Y_1}$ and  $\lim_n a_2^{\Y,n}$, and proceed by setting
\[
    \O^2 \coloneqq \{ (\o^\X_{1:2},\o^\Y_{1:2}) \,\vert\, (\o^\X_1,\o^\Y_1) \in \O^1, \o^\X_2 \in A^2_{\o^\X_1,\o^\Y_1}, \o^\Y_2 \in B^2_{\o^\X_1,\o^\Y_1} \}.
\]

\medskip  \emph{Steps 3.3 -- 3.4 and 4} ($t \mapsto t + 1$ and Choquet): Finally, the inductive step $t \mapsto t + 1$ as well as the remaining parts can be done in a similar fashion as before, by repeating the steps $(a)$ and $(b)$ appropriately.
\end{proof}

\subsection{Multicausal optimal transport} \label{sec:general_multi}
We proceed to formulate the multimarginal case as well as the barycenter problem. In this section, we thus consider $N \in \N$ filtered processes
\[\X^i=(\O^i,\Fc^i,\F^i,\P^i,X^i),\quad i \in \{1,\ldots,N\}, \] where, as before,
$(\O^i,\Fc^i,\F^i,\P^i)$ is a filtered probability space and $X^i=(X^i_t)_{t=1}^T$ is an $\F^i$-adapted process such that $X_t^i \in \Xc^i_t$, where $\Xc^i_t$ is a given Polish space. Analogously as before, we use the notation $\Xc^i=\prod_{t=1}^T \Xc_t^i$ and we work under the following standing assumption. We once again insist that this assumption is in most cases without loss of generality.

\begin{assumption} 
    For every $i \in \{1,\ldots,N\},$ we have that the probability space $\O^i$ is the product of some Polish spaces $\O^i_t$, $t \in \Tc$, \emph{i.e.}\ $\O^i = \prod_{t=1}^T \O^i_t$.
    The filtration $\F^i$ is the corresponding canonical filtration generated by the coordinate projections on $\prod_{t=1}^T \O^i_t$. 
    That is to say, $\Fc^i_t = \bigotimes_{s=1}^t\Bc(\O^i_s)\otimes \bigotimes_{s=t+1}^T \{\emptyset, \O^i_s \}$ for $t \in \Tc$.
    Further, we have that $\Fc^i=\Fc^i_T.$
\end{assumption}

Analogously as before, we use the notation $\Xc^i_{1:t} \coloneqq \prod_{s=1}^t \Xc_s^i$ and, similarly, $\O^i_{1:t}\coloneqq \prod_{s=1}^t \O^i_s$. Generic elements of $\O^i$, resp.\ $\O^i_t$, will be denoted by $\o^i$, resp.\ $\o^i_t$, and we write $\o^i_{1:t}\coloneqq (\o^i_1,\ldots,\o^i_t) \in \O^i_{1:t}$ for generic elements of $\O^i_{1:t},$ $t \in \Tc$. Finally, we denote by $K^i_t: \O^i_{1:t-1}\longrightarrow \Pc(\O^i_{t:T})$ a regular version of $\P^i(\mathrm{d}\omega^i_{t:T}\vert \Fc_{t-1}^i)$ for $t \in \{2,\ldots,T\}$ and  set $\Fc^i_0\coloneqq \{ \O^i,\emptyset \}.$ For a vector $t \in \{0,\ldots,T\}^N$ we write 
\[ 
    \bar{\Fc}_t\coloneqq \bigotimes_{i=1}^N \Fc^i_{t} \subseteq \Bc\Big(\prod_{i=1}^N \O^i\Big).
\]

\begin{definition} We denote by $\cpl(\X^1,\ldots,\X^N)$ the set of all probability measures on $\prod_{i=1}^N \O^i$ with marginals $\P^1,\ldots,\P^N$ and call its elements couplings. We say that a coupling $\pi \in \cpl(\X^1,\ldots,\X^N)$ is multicausal, denoted by $\cplmc(\X^1,\ldots,\X^N)$, if for any $i \in \{1,\ldots,N\}$ and any $t \in \Tc$ we have that $\bar{\Fc}_{T}$ is conditionally $\pi$-independent of $\bar{\Fc}_{t}$ given $\bar{\Fc}_{{t}}$. 
\end{definition}

\medskip Let $c: \prod_{i=1}^N \O^i \longrightarrow \R$ be a measurable cost function. We consider the following optimal transport problem
\[ 
    \inf_{\pi \in \cplmc(\X^1,\ldots,\X^N)} \int_{\O^1 \times \ldots \times \O^N} c(\o^1,\ldots,\o^N) \pi(\mathrm{d}\o^1,\ldots,\mathrm{d}\o^N).
\]

\begin{proposition} \label{prop:multi_attain} 
    Let $c: \prod_{i=1}^N \O^i \longrightarrow \R \cup \{+\infty\}$ be a lower-semicontinuous function and $\ell^i \in \L^1(\P^i)$ be such that $\sum_{i=1}^N \ell^i(\o^i)\leq c(\o^1,\ldots,\o^N),$ $(\o^1,\ldots,\o^N) \in \prod_{i=1}^N \O^i.$ 
    Then the problem 
    \[ 
        \inf_{\pi \in \cplmc(\X^1,\ldots,\X^N)} \int_{\O^1 \times \ldots \times \O^N} c(\o^1,\ldots,\o^N) \pi(\mathrm{d}\o^1,\ldots,\mathrm{d}\o^N)
    \] 
    is attained.
\end{proposition}

\begin{proof} The proof is standard, see \emph{e.g.}\ \cite[Remark 2.4]{AcKrPa23a}.
\end{proof} 
Similarly as before, we are interested in the dual problem. To that end, let us define for $i \in \{1,\ldots,N\}$
\[\Ac_{i,1} \coloneqq \{ f_1^i(\o^1,\ldots,\o^N)=a_1(\o^i_1) : a_1 \in \L^1(\Fc^i_1,\P^i) \}.\] and, for $t \in \{2,\ldots,T\}$
\begin{multline*} 
    \Ac_{i,t}^{\rm mc}\coloneqq \bigg\{ f^i_t(\o^1,\ldots,\o^N)=a^i_t(\o^i_{1:t},(\o^j_{1:t-1})_{j \neq i})-\int_{\O^i_t} a^i_t(\o^i_{1:t-1},\tilde{\o}^i_t,(\o^j_{1:t-1})_{j \neq i}) K^i_t(\o^i_{1:t-1};\mathrm d \tilde{\o}^i_t) \,\bigg\vert\,
    \\
    a_t^i\text{ is Borel measurable and }
    a^i_t(\o^i_{1:t-1},\,\cdot\,,(\o^j_{1:t-1})_{j \neq i}) \in  \L^1(\Bc(\O_t^i),K_t^i\big(\o^i_{1:t-1};\,\cdot\,) \big) \bigg\}.
\end{multline*}

Further, also analogously as before,
\begin{align} \label{eqn:dual_potent_multi}
\Sc_i^{\rm mc} \coloneqq \bigg\{ s(\o^1,\ldots,\o^N)=f^i_1(\o^i_1) + \sum_{t=2}^T f^i_t(\o^i_{1:t},(\o^j_{1:t-1})_{j \neq i}) \,\bigg\vert\, f^i_t \in \Ac_{i,t},\; t \in \{1,\ldots,T\} \bigg\}.  
\end{align}
Finally, we set $\Sc^{\rm mc} \coloneqq \Sc_1^{\rm mc} \oplus \Sc_2^{\rm mc} \oplus \ldots \oplus \Sc_N^{\rm mc}$. We have the following duality.

\begin{theorem} 
    Let $c : \prod_{i = 1}^N \O^i \longrightarrow \R \cup \{ +\infty \}$ be a lower-semicontinuous and lower-bounded function. Then we have
    \begin{equation*} 
        \inf_{\pi \in \cplmc(\X^1,\ldots,\X^N)} \int_{\O^1 \times \ldots \times \O^N} c \mathrm{d}\pi  =\sup \bigg\{ \int_{\O^1 \times \ldots \times\O^N} s \mathrm{d}(\P^1\otimes \ldots\otimes \P^N) \,\bigg\vert\, s \in \Sc^{\rm mc},\; s \leq c \bigg\}. \end{equation*}
    \end{theorem}
\begin{proof} 
    We refer to \cite{AcKrPa23a} for more details.
\end{proof}

\begin{theorem}[General case---multimarginal adapted transport] \label{existDualMC}
    Let $c : \prod_{i=1}^N \O^i \longrightarrow \R \cup \{ -\infty \}$ be measurable and $\ell^i \in \L^1(\P^i),$ $i \in \{1,\ldots,N\},$ be such that $c^i(\o^1,\ldots,\o^N)\leq\sum_{i} \ell^i(\o^i),$ $(\o^1,\ldots,\o^N) \in \prod_{i=1}^N \O^i$. 
    Then we have
    \begin{equation*} 
        \inf_{\pi \in \cplmc(\X^1,\ldots,\X^N)} \int_{\O^1 \times \ldots \times \O^N} c \mathrm{d}\pi 
        =\sup \bigg\{ \int_{\O^1 \times \ldots \times\O^N} s \mathrm{d}(\P^1\otimes \ldots\otimes \P^N) \,\bigg\vert\, s \in \Sc^{\rm mc},\; s \leq c \bigg\}.
    \end{equation*}
    Moreover, if either side is finite, then there exists $s^\star \in \Sc^{\rm mc}$ with $s^\star \leq c(X^1,\ldots,X^N)$ and
    \[
        \inf_{\pi \in \cplmc(\X^1,\ldots,\X^N)} \int_{\O^1 \times \ldots \times \O^N} c \mathrm{d}\pi =  \int_{\O^1 \times \ldots \times \O^N} s^\star  \mathrm{d}(\P^1\otimes \ldots\otimes \P^N).
    \]
\end{theorem}

\begin{proof}Since the proof follows analogous steps as the proof of \Cref{existDualC,existDualBC}, we just provide an outline. We can without loss of generality assume that $c$ is non-positive. For a measurable function $c: \prod_{i=1}^N \Xc^i \longrightarrow \R \cup \{-\infty\}$ we denote
    \begin{align*}
        D(c) &\coloneqq  \sup \bigg\{ \int_{\O^1 \times \ldots \times\O^N} s \mathrm{d}(\P^1\otimes \ldots\otimes \P^N) \,\bigg\vert\, s \in \Sc^{\rm mc},\; s \leq c \bigg\}, \\
        V(c) &\coloneqq  \inf_{\pi \in \cplmc (\X^1,\ldots,\X^N)} \int_{\O^1 \times \ldots \times \O^N} c \mathrm{d}\pi.
    \end{align*}
   \emph{Step 1 and 2} (Continuity of $D$ and $V$ from below and of $V$ from above):  First, we show that $V$ and $D$ are continuous from below on $\mathcal C^-$ and $V$ continuous from above on $\mathcal M^-$, where, in this case, we write $\mathcal C^-$ for the set of continuous, bounded, and non-positive functions on $\prod_i \Xc^i$, and $\mathcal M^-$ for the set of measurable and non-positive functions on $\prod_i \Xc^i$. This can be done by analogous arguments as before.

\medskip
     \emph{Step 3} (Main part): Next, we show that $D$ is continuous from above on $\Mc^-$ and $D(c)$ is attained. To that end, let $(c^n)_{n \in \N}$ be a sequence in $\Mc^-$ such that $c^n \searrow c$ for some $c \in \Mc^-$ such that $\inf_{n \in \N} D(c^n)>-\infty$. Let us denote \[ S^n \coloneqq D(c^n),\;{\rm and}\; S \coloneqq \inf_{n \in \N} S^n. \]

For any $n \in \N,$ let $s^n \in \Sc^{\rm mc}$ be a $1/n$-maximizer for $D(c^n)$. This means that for every $(\o^1,\dots,\o^N) \in \prod_{i = 1}^N \O^i$
 \begin{gather} \label{bound1_multicausal}
    s^n(\o^1,\ldots,\o^N)=\sum_{i=1}^N \bigg(f^{i,n}_1(\o^i_1) + \sum_{t=2}^T f^{i,n}_t(\o^i_{1:t},(\o^j_{1:t-1})_{j \neq i}) \bigg) \leq c^n(\o^1,\ldots,\o^N), \\
    \int_{\O^1\times \ldots \times \O^N} \Big(\sum_{i=1}^N f^{i,n}_1 \Big) \mathrm{d}(\P^1\otimes \ldots,\otimes \P^N) \geq S^n-\frac{1}{n},\; \text{and}\;
    \lim_{n \rightarrow \infty} \int_{\O^1\times \ldots \times \O^N} \Big(\sum_{i=1}^N f^{i,n}_1\Big) \mathrm{d}(\P^1\otimes \ldots\otimes \P^N)= S. \nonumber
\end{gather}
Moreover, we can without loss of generality assume \[ \int_{\O^1_1} f_1^{1,n}(\o^1_1) \P^1(\mathrm{d}\o^1_1)=\int_{\O^2_1} f_1^{2,n}(\o^2_1) \P^2(\mathrm{d}\o^2_1)=\cdots=\int_{\O^N_1} f_1^{N,n}(\o^N_1) \P^N(\mathrm{d}\o^N _1),\; n \in \N.\]
In the first step, we integrate both sides of the inequality \eqref{bound1_multicausal} with respect to $K^1_2(\o^1_1;\mathrm{d}\o^1_{2:T})\otimes \ldots \otimes K^N_2(\o^N_1;\mathrm{d}\o^N_{2:T})$. This together with the bound $c^n \leq 0$ yields
\[ \sum_{i=1}^N f^{i,n}_1(\o^i_1) \leq 0. \] 
Analogously as before, we can show boundedness of the sequence $(f^{i,n})_{n \in \N}$ in $\L^1(\Fc^i_1,\P^i).$ We thus employ Koml\'{o}s' lemma to conclude that, up to taking Ces\`aro means of a subsequence, there exist functions $f^i \in \L^1(\Fc^i_1,\P^i)$ for every $i \in \{1,\ldots,N\}$ such that $f^{i,n}_1 \longrightarrow f^i$ on a $\P^i$--full set $A^{i,1}$ and we set $f^i=-\infty$ outside of $A^{i,1}.$ Moreover, we find
\[ \sum_{i=1}^N \int_{\O^i_1} f_1^{i}(\o^i_1) \P^i(\mathrm{d}\o^i_1)\geq S. \]
We thus only need to verify that $f_1^{i},$ $i \in \{1,\ldots,N\},$ are admissible. In order to do so, we repeat the following steps. We proceed inductively forward in time. For every $t \in \{2,\ldots,T-1\}$ we repeat these steps for every $i \in \{1,\ldots,N\}$:
\begin{enumerate}[label = (\roman*), leftmargin=*]
    \item[(a)] We integrate the inequality \eqref{bound1_multicausal} with respect to 
    \begin{equation}
        \label{eq:multi_kernel}
         K^i_{t+1}(\o^i_{1:t};\mathrm{d}\o^i_{t+1:T}) \otimes \bigotimes_{j \neq i} K^j_{t}(\o^j_{1:t-1};\mathrm{d}\o^j_{t:T})
    \end{equation}
    and show that the sequence $(f_t^{i,n})_{n \in \N}$ is bounded from above by a suitable constant.
    \item[(b)] \label{item:b_multi} We employ \Cref{komlos} to conclude that there exists an appropriate sequence of measurable convex combinations $a^{i,n}_t(\o^i_{1:t-1},\,\cdot\,,(\o^j_{1:t-1})_{j \neq i}) \in {\rm conv}(f^{i,n}_t(\o^i_{1:t-1},\,\cdot\,,(\o^j_{1:t-1})_{j \neq i}),f^{i+1,n}_t(\o^i_{1:t-1},\,\cdot\,,(\o^j_{1:t-1})_{j \neq i}),\ldots)$ with coefficients depending measurably on $(\o^i_{1:t-1},(\o^j_{1:t-1})_{j \neq i})$ such that the limit $\lim_{n} a^{i,n}$ exists on a $K^i_t(\o^i_{1:t-1};\,\cdot\,)$--full set $A^{i,t}_{\o^i_{1:t-1},(\o^j_{1:t-1})_{j \neq i}}.$ Finally, we define $a^i_t$ analogously as in \eqref{causal:a_2}.
    
    \item[(c)] We show that $a^i_t(\o^i_{1:t-1},\,\cdot\,,(\o^j_{1:t-1})_{j \neq i}) \in  \L^1\big(\Fc^i_t,K^i_t(\o^i_{1:t-1};\,\cdot\,)\big)$ and, by Fatou's lemma, \begin{equation} \label{bound_A} 0 \leq \int_{\O^i_t} a^i_t(\o^i_{1:t-1},\tilde{\o}^i_t,(\o^j_{1:t-1})_{j \neq i}) K^i_t(\o^i_{1:t-1};\mathrm{d}\tilde{\o}^i_t). \end{equation}
    \item[(d)] We take the very same convex combinations with weights from \ref{item:b_multi} of the remaining sequences of functions, including the sequence of cost functions, which might change their dependence, but doesn't change their convergence properties as well as inequalities that they satisfy.
\end{enumerate}
In the end, we find that 
\[ 
    \sum_{i=1}^N \bigg(f^{i}_1(\o^i_1) + \sum_{t=2}^T a^{i}_t(\o^i_{1:t},(\o^j_{1:t-1})_{j \neq i}) \bigg) \leq c(\o^1,\ldots,\o^N),
\] 
and since \eqref{bound_A} holds for every $t \in \{2,\ldots,T\}$ and $i \in\{1,\ldots,T\},$ we also have
\[ 
    \sum_{i=1}^N \bigg(f^i_1(\o^i_1) + \sum_{t=2}^T a^{i}_t(\o^i_{1:t},(\o^j_{1:t-1})_{j \neq i})- \int_{\O^i_t} a^i_t(\o^i_{1:t-1},\tilde{\o}^i_t,(\o^j_{1:t-1})_{j \neq i}) K^i_t(\o^i_{1:t-1};\mathrm{d}\tilde{\o}^i_t) \bigg) \leq c(\o^1,\ldots,\o^N).
\]
We have thus constructed admissible functions that attain $D(c)$ and $\lim_{n \to \infty} D(c^n) = S = D(c)$.

\medskip
     \emph{Step 4} (Choquet): Finally, we can employ the Choquet capaticability theorem and obtain
    \[
        D(c) = \inf \{ D(\hat c) \,\vert\, \hat c \ge c,\; \hat c \text{ is l.s.c.\ and lower-bounded} \}
        = \inf \{ V(\hat c) \,\vert\, \hat c \ge c,\; \hat c \text{ is l.s.c.\ and lower-bounded} \}
        = V(c),
    \]
    This concludes the proof.
\end{proof}

 We conclude this section with the proof of \Cref{thm:intro.robustsuperhedging}.
\begin{proof}[Proof of \Cref{thm:intro.robustsuperhedging}]
    The statement is a direct consequence of \Cref{thm:simple_multi}, resp.\ \Cref{existDualMC}. It suffices to verify that the problem $p(\xi)$ corresponds exactly to the dual problem of the multicausal optimal transport $\sup_{\P \in \cplmc(\mu^1,\ldots,\mu^N)} \E^{\P}[\xi].$ Indeed, it is easy to verify that if  $p_0 \in \R$ is so that there exists  $\mathbf{\Delta} \in \Ac$ such that \[p_0 + \sum_{t=1}^T \mathbf{\Delta}_{t-1} \cdot (\mathbf{X}_{t}-\mathbf{X}_{t-1}) \geq \xi(\mathbf{X}), \] then
    $s(\mathbf{X})\coloneqq p_0 + \sum_{t=1}^T \mathbf{\Delta}_{t-1} \cdot (\mathbf{X}_{t}-\mathbf{X}_{t-1})$
    is admissible for the dual, see {\eqref{eqn:dual_potent_multi}}. Conversely, let \begin{align} \label{eqn:martingale} 
        p_0 + \sum_{i=1}^N \sum_{t=1}^T a_t^i(x^i_{1:t},(x^j_{1:t-1})_{j \neq i})- \int a_s^i(x^i_{1:t-1},\tilde{x}^i_t,(x^j_{1:t-1})_{j \neq i}) K^{i}_t(x^i_{1:t-1};\mathrm{d}\tilde{x}^i_t) \geq \xi(x^1,\ldots,x^N)
    \end{align} be admissible for the dual. It can readily be seen that the process 
    \[ 
        M_t^{i}(\mathbf{x}_{1:t}) \coloneqq \sum_{s=1}^t a_s^i(x^i_{1:s},(x^j_{1:s-1})_{j \neq i})- \int a_t^i(x^i_{1:s-1},\tilde{x}^i_s,(x^j_{1:s-1})_{j \neq i}) K^{i}_s(x^i_{1:s-1};\mathrm{d}\tilde{x}^i_s),\; t \in \{1,\ldots,T\}, 
    \] 
    is an $(\F^{i},\mu^i)$-martingale with $M^i_0\coloneqq\E^{\mu^i}[M_1^i]=0$.
    Hence, using \eqref{eq:intro.completeness} we find $\Delta^{i}(\mathbf{x})$ such that for $ t \in \{0,\ldots,T\}$
    \begin{equation}
        \label{eq:intro.robustexistence.completeness}
        M_t^{i}(\mathbf{x}_{1:t}) =\sum_{s=1}^{t} \Delta^{i}_{s-1}(\mathbf{x}) (x_{s}^i-x_{s-1}^i).
    \end{equation}
    As $M^{i}$ is $\F^\mathbf{X}$-adapted, we conclude that $\Delta^{i}_s(\mathbf{x})$ is, in fact, a measurable function of $(\mathbf{x}_{1:s})=(x^1_{1:s},\ldots,x^N_{1:s})$ for every $s \in \{0,\ldots,T\},$ which can be verified inductively, forward in time. Indeed, we use that for a fixed $(x^j)_{j \neq i}$, the process $\Delta^{i}$ is adapted to $\F^{i}$ and is uniquely determined by \eqref{eq:intro.robustexistence.completeness}. 
    Therefore, we get for every $t \in \{0,\ldots,T\}$
    \[ 
        M_t^{i}(\mathbf{x}_{1:t}) =\sum_{s=1}^{t} \Delta^{i}_{s-1}(\mathbf{x}_{1:s-1}) (X_{s}^i-X_{s-1}^i). 
    \] 
    It suffices to define $\mathbf{\Delta}\coloneqq (\Delta^{1},\ldots,\Delta^{N})$ to obtain 
    \[ 
        p_0^\star + \sum_{t=1}^T \mathbf{\Delta}_{t-1} \cdot (\mathbf{X}_{t}-\mathbf{X}_{t-1}) \geq \xi(\mathbf{X}).
    \]
    We thus obtain admissible strategy for $p(\xi).$ This concludes the proof.
\end{proof} 

\subsection{Characterization of polar sets}

An interesting mathematical application of the general duality theory in optimal transport is the characterization of so-called polar sets, that is, sets which have zero probability under any coupling with given (fixed) marginals, see \cite[Proposition 2.1]{BeGoMaSc08}.
For certain constrained optimal transport problems such as martingale optimal transport, see  \cite[Proposition 3.1]{BeNuTo16}, having the correct understanding of polar sets is crucial for posing the correct formulation of the dual problem in order to have attainment.
The remainder of this section is concerned with the characterization of polar sets in the multicausal and causal setting, which is done in \Cref{thm:polar_multi,thm:polar_causal}. For the particular result in the bicausal setting we refer to \Cref{cor:polar_bicausal} below.

\medskip To ease notation, we introduce notation for gluing of sets.

\begin{definition} 
    Let $\Ac_1$ and $\Ac_2$ be Polish spaces. 
    Let further $A^1 \subseteq \Ac_1$ and let $A^2_{a_1} \subseteq \Ac_2,$ $a_1 \in A^1,$ be a family of sets. We define the gluing of $A^1$ and $A^2_\bullet$ by
    \[ 
        (A^1 \boxtimes A^2_\bullet) \coloneqq \big\{ (a_1,a_2) \in \Ac_1 \times \Ac_2\,\big\vert\, a_1 \in A^1 \text{ and } a_2 \in A^2_{a_1} \big\}. 
    \]
\end{definition}

\begin{remark} \begin{enumerate}[label=(\roman*)]
    \item Note that the set $A^1 \boxtimes A^2_\bullet$ might not be Borel even if $A^1$ and $A^2_{a_1},$ $a_1 \in A^1,$ are Borel.
    \item If $\Ac_3$ is a further Polish space and we are given a family of sets $A^3_{a_1,a_2} \subseteq \Ac_3,$ $(a_1,a_2) \in A^1 \boxtimes A^2_\bullet,$  we write $A^1 \boxtimes A^2_\bullet \boxtimes A^3_\bullet$ or $\boxtimes_{t = 1}^3 A^t_\bullet$ for the set 
    \[
        (A^1 \boxtimes A^2_\bullet) \boxtimes A^3_\bullet =\big\{ (a_1,a_2, a_3) \in \Ac_1 \times \Ac_2 \times \Ac_3 \,\big\vert\, a_1 \in A^1,\; a_2 \in A^2_{a_1} \text{ and } a_3 \in A^3_{a_1,a_2} \big\} 
    \] to simplify the notation.
    \item If $A^3_{a_1}$ does not depend on $a_2$ we still write $A^1 \boxtimes A^2_\bullet \boxtimes A^3_\bullet$ or $\boxtimes_{t = 1}^3 A^t_\bullet$ for the set
    \[
        \big \{(a_1,a_2,a_3) \in \Ac_1 \times \Ac_2 \times \Ac_3 \,\big\vert\, a_1 \in \Ac^1, a_2 \in A^2_{a_1} \text{ and } a_3 \in A^3_{a_1} \big\}.
    \]
    \item We use similar notation for cases involving more than three spaces or different set dependencies. This notation should not cause confusion, as there is little ambiguity.
\end{enumerate} 
\end{remark}

\begin{theorem}[Multicausal transport] \label{thm:polar_multi}
    Consider the setting of {\rm \Cref{sec:general_multi}} and let $E \subseteq \O^{1:N}$ be Borel. Then the following are equivalent:
    \begin{enumerate}[label = (\roman*)]
        \item \label{it:polar_multi.1} $\pi(E) = 0$ for all $\pi \in \cpl_{\rm mc}(\X^1,\ldots,\X^N);$
        \item \label{it:polar_multi.2} For every $i \in \{1,\ldots, N\}$ there are $\P^i$--full sets $A^{i,1}$ and for every $t \in \{2,\ldots,T\}$ and $\o^{1:N}_{1:t-1} \in \O^{1:N}_{1:t-1},$ there are $K^i_t(\o^i_{1:t-1};\,\cdot\,)$--full sets $A^{t,i}_{\o^i_{1:t-1},(\o^j_{1:t-1})_{j \neq i}} \subseteq \O^i_t,$  such that the set \begin{equation}
            \label{eqn:polar_multi}
            \Big[ \boxtimes_{t = 1}^T (A^{t,1}_\bullet \times \dots \times A^{t,N}_\bullet)  \Big]^C.
        \end{equation} is Borel and contains $E.$
    
    \end{enumerate}
\end{theorem}

\begin{proof}
    The implication `$\ref{it:polar_multi.2}\implies\ref{it:polar_multi.1}$' follows simply by successively disintegrating any $\pi \in \cplmc(\X^1,\ldots,\X^N)$, and observing that for every $t \in \{2,\ldots,T\}$ and $\o^{1:N}_{1:t-1} \in \O^{1:N}_{1:t-1}$
    \begin{align*}
        K_t^\pi \Big(\o^{1:N}_{1:t-1}; A^{t,1}_{\o^1_{1:t-1}, ( \o^j_{1 : t - 1} )_{j \neq 1} } \times \dots \times A^{t,1}_{\o^N_{1:t-1}, ( \o^j_{1 : t - 1} )_{j \neq N} } \Big)
        & \geq 1-\sum_{i = 1}^N K^\pi_t\Big( \o^{1:N}_{1:t-1}; \Big[ A^{t,i}_{\o^i_{t-1}, (\o^j_{1:t-1})_{j \neq i}} \Big]^C \Big) \\
        &= 1-\sum_{i = 1}^N K_t^i\Big(\o^{i}_{1:t-1};\Big[  A^{t,i}_{\o^i_{t-1}, (\o^j_{1:t-1})_{j \neq i}} \Big]^C \Big)
        \\
        &= 1,
    \end{align*}
    where $K^\pi_t$ denotes the disintegration of $\pi$ given $\o^{1:N}_{1:t-1}$ and we used multicausality for the first equality. 
    In the case when $t=1,$ we have by analogous arguments that $\pi(A^{1,1}\times \cdots\times A^{N,1})\geq 1-\sum_{i=1}^N \P^i([A^{N,1}]^C)=1$ Consequently, we have $\pi(E^C) = 1$ by Fubini's theorem, which concludes the first part.

    \medskip To show the implication `$\ref{it:polar_multi.1}\implies\ref{it:polar_multi.2}$', we consider the corresponding multicausal transport problem with cost $c = -\mathbf 1_{E}$, which clearly has value $0$.
    Employing \Cref{existDualMC} provides us with admissible dual potentials $f_t^i \in \Ac^i_t$ for $t \in \{1,\ldots,T\}$ and $i \in \{1,\ldots, N\}$ that satisfy
    \begin{gather}
        \sum_{i=1}^N \Big(f^{i}_1(\o^i_1) + \sum_{t=2}^T f^{i}_t(\o^i_{1:t},(\o^j_{1:t-1})_{j \neq i}) \Big) 
        \le -\mathbf 1_{N}(\o^{1:N}) \le 0, \label{eqn:polar_multi_dual}\\
        \sum_{i=1}^N\int f_1^i \mathrm{d}\P^i=\int \Big[ \sum_{i=1}^N\sum_{t=2}^T f_t^i \Big] \mathrm{d}\Big( \bigotimes_{i=1}^N\P^i \Big) 
        =0. \label{eqn:polar_multi_dual2}
    \end{gather}
    Let us inductively set, for $t \in \{2,\ldots,T\}$,
    \begin{align*}
     A^{1,i}&\coloneqq \{ \o^i_1 \in \O^i_1 \,\vert\, f_1^i(\o^i_1)=0 \}, \\
     A^{t,i}_{\o^{1:N}_{1:t-1}}&\coloneqq \begin{cases}\big\{ \tilde \o^i_{t} \in \O^i_{t} \,\big\vert\, f_t^i(\o^i_{1:t-1},\tilde \o^i_t,(\o^j_{1:t-1})_{j \neq i})=0 \big\} &\o^{1:N}_{1:t-1} \in \boxtimes_{s = 1}^{t-1} \big(A^{s,1}_\bullet \times \dots \times A^{s,N}_\bullet \big) \\
     \O^i_t &\text{otherwise.}
     \end{cases}
    \end{align*}

\medskip We note that whenever $c<0,$ the value of at least one of the dual potentials $f_i^t,$ $i \in \{1,\ldots,N\},$ $ t \in \{1,\ldots,T\}$ must be negative as well. It follows that
\begin{align*} E&= \{ \o^{1:N} \in \O^{1:N}  \,\vert\, c(\o^{1:N})<0 \} \\
&\subseteq \Big[ \bigcap_{i=1}^N \bigcap_{t=1}^T \{ f_i^t=0 \} \Big]^C \\
&= \Big[ \bigcap_{i=1}^N  \{f_1^i=0 \}\cap \bigcap_{i=1}^N \bigcap_{t=2}^T \big\{f_t^i=0 \text{ and } f_s^j=0 \text{ for every } s \in \{1,\ldots,t-1\} \text{ and } j \in  \{1,\ldots,N\} \big\} \Big]^C\\
&=\Big[ \boxtimes_{t = 1}^T (A^{t,1}_\bullet \times \dots \times A^{t,N}_\bullet) \Big]^C.
\end{align*}
It remains to verify that $A^{t,i}_\bullet$ has the desired properties. For $t \in \{2,\ldots,T\},$ we integrate left-hand and right-hand side of \eqref{eqn:polar_multi_dual} with respect to the kernel $\bigotimes_{i=1}^N K^i_{2}(\o^i_{1};\mathrm{d}\o^i_{2:T})$. We find that $\sum_{i = 1}^N f_1^i \le 0$ and $\sum_{i = 1}^N \int f_1^i \mathrm{d}\P^i = 0$.
    By translating the potentials, as done in the proof of \Cref{existDualMC}, we may also assume without loss of generality that $\int f_1^i \mathrm{d} \P^i = 0$, for every $i \in \{1,\ldots,N\}.$ Thus, integrating \eqref{eqn:polar_multi_dual} with respect to $K^i_{2}(\o^i_{1};\mathrm{d}\o^i_{2:T}) \otimes \bigotimes_{j \neq i} \P^j(\mathrm{d}\o^j_{t:T}),$ we find $f_1^i \le -\sum_{j \neq i} \int f_1^j \mathrm{d}\P^j = 0$, from where we conclude that the set
    \[
        A^{i,1} = \{ \o^i_1 \in \O^i_1 \,\vert\, f_1^i(\o^i_1) = 0 \},
    \]
    is $\P^i$--full for every $i \in \{1,\ldots,N\}$. Next, let $t \in \{2,\ldots,T\}.$ It is clear that whenever $\o^{1:N}_{1:t-1} \notin \boxtimes_{s = 1}^{t-1} \big(A^{s,1}_\bullet \times \dots \times A^{s,N}_\bullet \big),$ the set $A^{t,i}_{\o^{1:N}_{1:t-1}}$ is $K^i_t(\o^i_{1:t-1};\,\cdot\,)$--full. Assume thus the contrary. In other words, assume that for this $\o^{1:N}_{1:t-1}$ it holds
    \[ \forall s \in \{1,\ldots,t-1\},\,\forall i \in \{1,\ldots,N\},\;\forall k \in \{1,\ldots,N\}\;\forall u \in \{1,\ldots, s \},\; f_u^k(\o^k_{1:u},(\o^j_{1:u-1})_{j \neq k}) = 0.\]
    Again, integrating \eqref{eqn:polar_multi_dual} with respect to the kernel $ K^i_{t+1}(\o^i_{1:t};\mathrm{d}\o^i_{t+1:T}) \otimes \bigotimes_{j \neq i} K^j_{t}(\o^j_{1:t-1};\mathrm{d}\o^j_{t:T})$ gives

    \[
        f_t^i(\o^i_{1:t},(\o^j_{1:t-1})_{j \neq i}) \le 0.
    \]
    
    We have that $f_t^i \in \Ac_t^i,$ and so \[\int_{\O^i_t}f_t^i(\o^i_{1:t},(\o^j_{1:t-1})_{j \neq i}) K^i_t(\o^i_{1:t-1};\mathrm{d}\o^i_t)=0. \] We deduce that in this case $\{ \o^i_t \in \O^i_t \,\vert\, f_t^i(\o^i_{1:t},(\o^j_{1:t-1})_{j \neq i}) = 0 \}$ is $K^i_t(\o^i_{1:t-1};\,\cdot\,)$--full.

\medskip Hence, the sets $A^{i,1}_\bullet$ have the desired properties and the proof is concluded.
\end{proof}

    We have seen already in {\rm\Cref{rem:J_causal}} that the multicausal setting includes bicausal optimal transport.
    As a consequence of {\rm\Cref{thm:polar_multi}} we obtain the following characterization of polar sets in these settings, which we formulate here explicitly.

\begin{corollary}[Bicausal transport] \label{cor:polar_bicausal}
     Let $E \subseteq \O^\X \times \O^\Y$ be Borel. The following are equivalent:
        \begin{enumerate}[label = (\roman*)]
            \item $\pi(E) = 0$ for all $\pi \in \cplbc(\X,\Y)$;
            \item  There are a $\P^\X$--full set $A^1$ and a $\P^\Y$--full set $B^1$ and for every $t \in \{2,\ldots,T\}$ and $(\o^{\X}_{1:t-1},\o^{\Y}_{1:t-1}) \in \O^{\X}_{1:t-1} \times \O^{\Y}_{1:t-1},$ there are a $K^\X_t(\o^\X_{1:t-1};\,\cdot\,)$--full set $A^{t}_{\o^\X_{1:t-1},\o^\Y_{1:t-1}} \subseteq \O^\X_t$ and a $K^\Y_t(\o^\Y_{1:t-1};\,\cdot\,)$--full set $B^{t}_{\o^\X_{1:t-1},\o^\Y_{1:t-1}} \subseteq \O^\Y_t$ such that \begin{equation*}
            \Big[ \boxtimes_{t = 1}^T (A^{t}_\bullet \times B^{t}_\bullet)  \Big]^C.
        \end{equation*} is Borel and contains $E.$
        \end{enumerate}
\end{corollary}

\begin{theorem}[Causal transport] \label{thm:polar_causal} 
        Consider the setting of {\rm\Cref{sec:general_adapt_causal}} and let $E \subseteq \O^\X \times \O^\Y$ be Borel. The following are equivalent:
        \begin{enumerate}[label = (\roman*)]
            \item $\pi(E) = 0$ for all $\pi \in \cplc(\X,\Y)$;
            \item There are a $\P^\X$--full set $A^1$ and a $\P^\Y$--full set $B^1$ and for every $t \in \{2,\ldots,T\}$ and $(\o^{\X}_{1:t-1},\o^{\Y}_{1:t-1}) \in \O^{\X}_{1:t-1} \times \O^{\Y}_{1:t-1},$ there are a $K^\X_t(\o^\X_{1:t-1};\,\cdot\,)$--full set $A^{t}_{\o^\X_{1:t-1},\o^\Y_{1:t-1}} \subseteq \O^\X_t$ and a $K^\Y_t(\o^\Y_{1:t-1};\,\cdot\,)$--full set $B^{t}_{\o^\Y_{1:t-1}} \subseteq \O^\Y_t$ such that \begin{equation*}
            \Big[ \boxtimes_{t = 1}^T (A^{t}_\bullet \times B^{t}_\bullet)  \Big]^C.
        \end{equation*} is Borel and contains $E.$
        \end{enumerate}
\end{theorem}

\begin{proof}
    The proof of \Cref{thm:polar_multi} carries over with the obvious modifications.
    Note that in this case, $B^{t}_{\o^\Y_{1:t-1}}$ does not depend on $\o^\X_{1:t-1},$ and so $B\coloneqq \boxtimes_{t=1}^T B^{t}_{\o^\Y_{1:t-1}} \subseteq \O^\Y$ is a $\P^\Y$--full Borel set.
\end{proof}

\subsection{Causal barycenters} \label{sec:general_bary}

In this section, we consider the dual problem of barycenters. We refer interested readers to the work \citeauthor*{AcKrPa23a} \cite{AcKrPa23a} for further discussion, where an application
in a dynamic matching problem is also presented. 
To that end, we consider
\[ \inf_{\Y \in {\rm FP}(\Yc)} \sum_{i=1}^N \Cc\Wc_{c^i}(\X^i,\Y),\] 
where $c^i : \O^i \times \Yc \longrightarrow \R$ is a measurable function and $\Yc = \prod_{t=1}^T \Yc_t$ is a given complete separable metric space, thus Polish space, with the metric $d_{\Yc}$. Recall that for the sake of brevity, we sometimes write $c^{1:N} = (c^1,\ldots,c^N)$. Here, ${\rm FP}(\Yc)$ denotes the factor space of all filtered processes
\[ \Y=(\Omega,\Fc,\F,\P,Y), \] where $Y$ is an $\F$-adapted process with $Y_t \in \Yc_t,$ $t \in \Tc,$ with respect to the equivalence \[ \Y \sim \Y^\prime \iff \Ac\Wc_{d_{\Yc}\wedge 1}(\Y,\Y^\prime)=0.\] See \citeauthor*{BaBePa21} \cite{BaBePa21} for more details. Let us now set $\Y^\nu \coloneqq (\Yc,\Fc_T^Y,\F^Y,\nu,Y),$ where $\F^Y$ is the canonical $\sigma$-algebra on the path space $\Yc=\prod_{t=1}^T\Yc_t,$ $Y$ is the canonical process on $\Yc$ and $\nu \in \Pc(\Yc).$ Thanks to \cite[Remark 4.2]{AcKrPa23a}, we have
\begin{equation}
    \label{eq:bary.canonicalvsfp}
    \inf_{\Y \in {\rm FP}(\Yc)} \sum_{i=1}^N \Cc\Wc_{c^i}(\X^i,\Y)=\inf_{\nu\in \Pc(\Yc)} \sum_{i=1}^N \Cc\Wc_{c^i}(\X^i,\Y^\nu).
\end{equation}
Let us further denote
\[ \cplc(\X^i,\ast)\coloneqq \bigcup_{\nu \in \Pc(\Yc)}\cplc(\X^i,\Y^\nu).  \]

\medskip
Let us define for $i \in \{1,\ldots, N\}$ and $t \in \{2,\ldots,T\}$ the following sets of functions, which allow us to test causality and identify the marginals:
 \begin{align*}   
     \Ac_{i,t}^{\Yc}&\coloneqq \bigg\{ f^i_t(\o^i,y)=a^i_t(\o^i_{1:t},y_{1:t-1})-\int_{\O^i_t} a^i_t(\o^i_{1:t-1},\tilde{\o}^i_t,y_{1:t-1}) K^i_t(\o^i_{1:t-1};\mathrm{d}\tilde{\o}^i_t) \,\bigg\vert\, \\
    &\hspace{2cm}a^i_t \text{ is Borel measurable and } a^i_t(\o^i_{1:t-1},\,\cdot\,,y_{1:t-1}) \in  \L^1\big(\Bc(\O^i_t),K^i_t\big(\o^i_{1:t-1};\,\cdot\,)\big) \bigg\}, \\
    \Ac^{\Yc}_{i,1}&\coloneqq \Big\{ f^i_1(\o^i,y)=a_1^i(\o^i_1) \,\Big\vert\, a^i_1 \in \L^1(\Fc^i_1,\P^i) \Big\}.
\end{align*}
Furthermore, to simplify notation, we set $\Gc_{\Yc} \coloneqq \{ s(y)=g(y) \,\vert\, g : \Yc \longrightarrow \R \text{ is measurable } \}$ and we define the set of dual potentials by
\begin{multline}
    \label{eq:def.Phi0}
    \Phi^0(c^{1:N}) \coloneqq\bigg\{ f^{1:N}_1 \in \prod_{i=1}^N \L^1(\Fc^i_1,\P^i) \,\bigg\vert\, \forall \nu \in \Pc(\Yc),\,\exists f_{2:T}^i \in \Ac_{2:T}^{\Yc},\,\exists g^i \in \Gc_{\Yc},\; i \in \{1,\ldots,N\}: \\
    f_1^i(\o^i_1) + \sum_{t=2}^T f_t^i(\o^i_{1:t},y_{1:t-1}) + g^i(y) \leq c^i(\omega^i,y)\text{ and } \sum_{i=1}^N g^i(y)=0\;\cplc(\X^i,\Y^\nu)\text{--q.s.} \bigg\}.
\end{multline}

We note that in the causal barycenter problem, we allow the dual potentials to partly depend on the marginal $\nu.$ This relaxation is necessary in our proofs to construct limits and to have attainment, \emph{c.f.}\ proof of \Cref{lem:step3.attain_bary}. 
This will later be relaxed in \Cref{thm:attain_bary.CH}, but at the cost of assuming the continuum hypothesis and universal measurability of the dual variables.

\begin{remark} \label{rem:qs_bary} 
    We point out that if, for $i \in \{1,\dots,N\}$ and some $\nu \in \Pc(\Yc)$,
    \begin{align} \label{eqn:rem_qs}
        f_1^i(\o^i_1) + \sum_{t=2}^T f_t^i(\o^i_{1:t},y_{1:t-1}) + g^i(y) \leq c^i(\omega^i,y)\text{ and } \sum_{i=1}^N g^i(y)=0\quad\cplc(\X^i,\Y^\nu)\text{{\rm--q.s.}},
    \end{align} 
    then \eqref{eqn:rem_qs} holds everywhere on a $\cplc(\X^i,\Y^\nu)$--full Borel set $A^i$ of the form
    \[
        A^i = \boxtimes_{t = 1}^T (A^{i,t}_\bullet \times B^t_\bullet),
    \]
    where $B = \boxtimes_{t = 1}^T B_\bullet^t \subseteq \Yc$.
    Indeed, {\rm\Cref{thm:polar_causal}} gives us that, for every $i \in \{1,\ldots,N\}$, \eqref{eqn:rem_qs} holds everywhere on a $\cplc(\X^i,\Y^\nu)$--full Borel set $\tilde A^i$ of the form
    \[
        \tilde A^i = \boxtimes_{t = 1}^T (A^{i,t}_\bullet \times B^{i,t}_\bullet),
    \]
    where $B^i = \boxtimes_{t = 1}^T B_\bullet^{i,t} \subseteq \Yc$. 
    Then it suffices to set $B^t_\bullet\coloneqq \bigcap_{i=1}^N B^{i,t}_\bullet$ to obtain the conclusion above.

    \medskip Moreover, observe that it is clear from the proof of {\rm \Cref{thm:polar_causal}} that, for every $t \in \{1,\ldots,T\}$, the projection of $A^i$ onto $\O^i_{1:t} \times \Yc$ is Borel.
\end{remark}

\begin{remark} \label{rem:baryc_L1}
Let $\nu \in \Pc(\Yc)$ and assume $c^i(\o^i,y)\leq \ell^i(\o^i)+k(y)$ for some $k \in \L^1(\nu)$ and $\ell^i \in \L^1(\P^i)$, $i \in \{1,\dots,N\}$.
Let further $\pi \in \cplc(\X^i,\Y^\nu)$ and $f_1^{1:N} \in \Phi^0(c^{1:n}).$ That is, there are $f^{1:N}_{2:T}$ and $g^{1:N}$ satisfying \eqref{eqn:rem_qs}. Summing over all $i \in \{1,\ldots, N\}$ gives 
\begin{equation} \label{rem_bary_integr} \sum_{i=1}^N \Big(f_1^i+\sum_{t=2}^T f_t^i\Big) \leq \sum_{i=1}^N c^i\leq \sum_{i=1}^N \ell^i + Nk.\end{equation}
Let now $i \in \{1,\ldots,N\}$ be fixed. 
Integrating \eqref{rem_bary_integr} with respect to $K^i_3(\o^i_{1:2};\mathrm{d}\o^i_{3:t}) \otimes \bigotimes_{j \neq i} \P^j(\mathrm{d}\o^j)$ gives
\[ 
    f_2^i(\o^i_{1:2},y_1) \leq -f_1^i(\o^i_1) + Nk(y) + \int \ell^i(\o^i)K^i_3(\o^i_{1:2};\mathrm{d}\o^i_{3:t}) + \sum_{j \neq i} \int \big(\ell^j(\o^j) - f_1^j(\o^j_1)\big) \P^j(\mathrm{d}\o^j).
\]

It follows that, since the right-hand side is integrable with respect to $\pi^i,$ we can conclude similarly as in the proof of {\rm\Cref{thm:caus_dual}}, that $f_2^i \in \L^1(\pi^i)$ for every $i \in \{1,\ldots,N\}$ and, in particular, $\int f_2^i \mathrm{d} \pi^i=0.$ Also similarly as in the proof of {\rm\Cref{thm:caus_dual}}, we can inductively forward in time verify that $f_t^i \in \L^1(\pi^i)$ and $\int f^i \mathrm{d}\pi^i=0$ for every $i \in \{1,\ldots,N\}$ and $t \in \{1,\ldots,T\}.$

\medskip We further have that $g^{i}(y)  \leq c^i-\sum_{t=1}^T f_t^i\leq \ell^i+k-\sum_{t=1}^T f_t^i$ and, consequently, 
\[g^{i}(y)=-\sum_{j \neq i} g^j(y) \geq -\sum_{j \neq i} \Big[\ell^i(\o^i)+k(y)- f_1^i(\o^i_1)-\sum_{t=2}^T f_t^i(\o^i_{1:t},y_{1:t-1}) \Big]. \]
These two inequalities give that $g^i \in \L^1(\pi^i),$ and so $g^i \in \L^1(\nu),$ $i \in\{1,\ldots,N\}.$ 
\end{remark}

\begin{proposition}\label{prop:weak_duality}
    For every $i \in \{1,\dots,N\}$, let $c^i : \O^i \times \Yc \longrightarrow \R$ be measurable and bounded from below.
    Then we have
    \[
        \inf_{\nu \in \Pc(\Yc)} \sum_{i = 1}^N \Cc\Wc_{c^i}(\X^i,\Y^\nu) 
        \ge \sup_{f^{1:N}_1 \in \Phi^0(c^{1:N})} \sum_{i = 1}^N \int_{\O^i_1} f_1^i \mathrm d\P^i.
    \]
\end{proposition}

\begin{proof}
    Pick any $f_1^{1:N} \in \Phi^0(c^{1:N})$ and $\nu \in \Pc(\Yc)$ with $\sum_{i = 1}^N \Cc\Wc_{c^i}(\X^i,\Y^\nu) < \infty$. If such $\nu$ doesn't exist, the statement holds trivially.
    By definition, there are $(f_{2:T}^{1:N},g^{1:N}) \in \Ac_{1:N,2:T}^{\Yc} \times \Gc_\Yc^N$ such that
    \begin{equation}
        \label{eq:weak_duality.1}
        \sum_{t = 1}^T f_t^i + g^i \le c^i \text{ and } \sum_{i = 1}^N g^i = 0 \quad \cplc(\X^i,\Y^\nu)\text{--q.s.}
    \end{equation}
    For $i \in \{1,\dots,N\}$, let $\pi^i \in \cplc(\X^i,\Y^\nu)$ with $c^i \in \L^1(\pi^i)$. Summing both sides of \eqref{eq:weak_duality.1} over $i \in \{1,\ldots,N\}$ yields 
    \begin{equation} \label{eq:bary_integr} \sum_{i=1}^N \Big(f_1^i+\sum_{t=2}^T f_t^i\Big) \leq \sum_{i=1}^N c^i.\end{equation}
     Similarly as in {\rm\cite[Definition 22]{Pa22}}, there exists a $\gamma \in \cpl(\X^1,\ldots,\X^N,\Y^\nu)$ such that ${\rm proj}^{i,N+1}_{\#}\gamma =\pi^i.$ It is easy to see that the process $M_t\coloneqq \sum_{s=1}^t \sum_{i=1}^N f_s^i$ is a local martingale under $\gamma$ with $M_0\in \L^1(\gamma)$ and $(M_T)^+ \in \L^1(\gamma)$ by  \eqref{eq:bary_integr}. It follows that $M$ is a true martingale under $\gamma$ and, by integrating \eqref{eq:bary_integr} with respect to $\gamma,$ we obtain
     \[ \sum_{i = 1}^N \int f_1^i \mathrm d\pi^i\leq \sum_{i = 1}^N \int c^i \mathrm d\pi^i.  \]
    From this inequality, the assertion readily follows.
\end{proof}

We have the following duality result.

\begin{proposition} \label{prop:lsc_bary} 
    Assume that $c^i,$ $i \in \{1,\ldots,N\},$ are lower-semicontinuous and lower-bounded.
    Then we have
    \[
        \inf_{\Y \in {\rm FP}(\Yc)} \sum_{i=1}^N \Cc\Wc_{c^i}(\X^i,\Y)=\sup_{f_1^{1:N} \in \Phi^0(c^{1:N})} \sum_{i=1}^N \int_{\O^i_1} f^i_1(\o^i_1) \P^i(\mathrm{d}\o^i_1). 
    \]
\end{proposition}
\begin{proof}
   The duality was proved in \cite[Theorem 4.7]{AcKrPa23a} with a smaller class of dual potentials. Thus, the statement follows from the weak duality in \Cref{prop:weak_duality}.
\end{proof}

\begin{theorem}[Relaxed duality for causal barycenters] \label{thm:attain_bary}
    Let $\Yc$ be a $\sigma$-compact space.
    For every $i \in \{1,\ldots,N\},$ let $c^i : \Xc^i \times \Yc \longrightarrow \R$ be measurable and lower-bounded, $k : \Yc \longrightarrow \R$ be bounded on compacts, and $\ell^i \in \L^1(\P^i)$ be such that $c^i(\o^i,y) \leq  \ell^i(\o^i)+k(y),$ $(\o^i,y) \in \O^i \times \Yc$. 
    Then, we have
    \[ 
        \inf_{\Y \in {\rm FP}(\Yc)} \sum_{i=1}^N \Cc\Wc_{c^i}(\X^i,\Y)=\sup_{f_1^{1:N}\in \Phi^0(c^{1:N})} \sum_{i=1}^N \int_{\O^i_1} f^i_1(\o^i_1) \P^i(\mathrm{d}\o^i_1).
    \]
    Moreover, there is $f_1^{1:N}\in \Phi^0(c^{1:N})$ such that 
    \[
        \inf_{\Y \in {\rm FP}(\Yc)} \sum_{i=1}^N \Cc\Wc_{c^i}(\X^i,\Y)=\sum_{i=1}^N \int_{\O^i_1} f^i_1(\o^i_1) \P^i(\mathrm{d}\o^i_1).
    \]
\end{theorem}

\begin{proof} Let us without loss of generality assume that $c^i \geq 0.$
    We write
    \begin{align*}
        D(c^1,\ldots,c^N) \coloneqq \sup_{f_1^{1:N}\in \Phi^0(c^{1:N})} \sum_{i=1}^N \int_{\O^i_1} f^i_1(\o^i_1) \P^i(\mathrm{d}\o^i_1), \quad
        V(c^1,\ldots,c^N) \coloneqq  \inf_{\Y \in {\rm FP}(\Yc)} \sum_{i=1}^N \Cc\Wc_{c^i}(\X^i,\Y).
    \end{align*}

    \medskip \underline{Part A} (Compact $\Yc$): In this part, we assume that $\Yc$ is compact. For every $i \in \{1,\ldots,N\}$, let $\mathcal C^-_i$ be the set of non-negative continuous functions on $\O^i \times \Yc$ that are dominated from above by $\ell^i+ C$, and let $\mathcal M^-_i$ be the set of non-negative measurable functions on $\O^i \times \Yc$ that are dominated from above by $\ell^i+C$. Here, we denote $C\coloneqq \sup_{y \in \Yc}k(y)$, which is finite thanks to compactness of $\Yc$ and $k$ being bounded on compacts.

    \medskip  \emph{Step 1 and 2} (Continuity of $D$ and $V$ from below and of $V$ from above): First, we show that $V$ and $D$ are continuous from below on $\mathcal C^-_{1:N}$ and $V$ is continuous from above on $\mathcal M^-_{1:N}$. This can be done in the very same fashion as before using compactness of $\Yc$.

    \medskip  \emph{Step 3} (Continuity from above of $D$ and attainment): Next, we show that $D$ is jointly continuous from above on $\Mc^-_{1:N}$ and that $D(c)$ is attained. To that end, for every $i \in \{1,\ldots,N\}$, let $(c^{i,n})_{n \in \N}$ be a sequence in $\Mc^-_i$ such that $c^{i,n} \searrow c^i$ for some $c^i \in \Mc^-_i.$
    Our assumptions clearly give $\inf_{n \in \N} D(c^{1,n},\ldots,c^{N,n}) \in \R.$
    Let us denote 
    \[ 
        S^n \coloneqq D(c^{1,n},\ldots,c^{N,n})= \sup_{f_1^{1:N}\in \Phi^0(c^{1:N,n})} \sum_{i=1}^N \int_{\O^i_1} f^i_1(\o^i_1) \P^i(\mathrm{d}\o^i_1),\quad{\rm and}\quad S \coloneqq \inf_{n \in \N} S^n. 
    \]

    For $n \in \N$, let $f^{1:N,n}_1 \in \Phi^0(c^{1:N,n})$ be a $1/n$-minimizer for $D(c^{1:N,n})$.
    We can invoke \Cref{lem:step3.attain_bary} to find $f^{1:N}_1 \in \Phi^0(c^{1:N})$ with
    \[
        S=\limsup_{n \rightarrow \infty} \sum_{i = 1}^N \int_{\O^i_1} f^{i,n}_1 \mathrm d \P^i \le \sum_{i = 1}^N \int_{\O^i_1} f^i_1 \mathrm d \P^i \le D(c^{1:N}) \le S,
    \]
    where the last inequality follows from weak duality. Thus, we have shown attainment and that $D$ is continuous from above on $\Mc_{1:N}^-$.

    \medskip\emph{Step 4} (Choquet): Finally, we invoke the multidimensional Choquet capaticability theorem, see \Cref{lem:choquet}, to obtain
    \begin{align*}
        D(c^1,c^2,\ldots,c^N) &= \inf \{ D(\hat c^1,\hat{c}^2,\ldots,\hat{c}^N) \,\vert\, \hat c^i \ge c^i,\; \hat c^i \text{ is l.s.c., lower-bounded,}\;\hat  c^i \in \Mc^-_i \} \\
        &=\inf \{ V(\hat c^1,\hat{c}^2,\ldots,\hat{c}^N) \,\vert\, \hat c^i \ge c^i,\; \hat c^i \text{ is l.s.c., lower-bounded},\; \hat  c^i \in \Mc^-_i \}\\
        &= V(c^1,c^2,\ldots,c^N),
    \end{align*}
    where the second equality is due to \Cref{prop:lsc_bary}. This concludes the proof.

    \medskip \underline{Part B} ($\sigma$-compact $\Yc$): It remains to verify the result for a general $\sigma$-compact path space $\Yc$.
    For $n \in \N$ and $t \in \{1,\dots,T\}$, let $\Yc_t^n \subseteq \Yc_t$ be compact with $\Yc_t^n \nearrow \Yc_t$.
    Define $\xi_t^n : \Yc_t \longrightarrow \Yc_t^n$ by
    \[
        \xi_t^n(y_t) \coloneqq 
        \begin{cases}
            y_t & y_t \in \Yc_t^n, \\
            y_t^0 & \text{otherwise},
        \end{cases}
    \]
    where $y_t^0 \in \Yc_t^1$ is fixed, and set $\xi^n \coloneqq \xi^n_{1:T} : \Yc \longrightarrow \Yc^n$.
    We further note that the space $\Yc^n \coloneqq \Yc^n_{1:T}$ is compact.
    Finally, we define $c^{i,n} : \O^i \times \Yc \longrightarrow \R$ by $c^{i,n}(\o^i,y) \coloneqq c^i(\o^i,\xi^n(y))$ and write $c^i_{|\Yc^n}$ for $c^i$ restricted to $\O^i \times \Yc^n$.

    \medskip First, note that
    \[
        \lim_{n \to \infty} D(c^{1:N}_{|\Yc^n}) =
        \lim_{n \to \infty} \inf_{\nu \in \Pc(\Yc^n)} \sum_{i = 1}^N \Cc\Wc_{c^i}(\X^i,\Y^\nu) \ge
        \inf_{\nu \in \Pc(\Yc)} \sum_{i = 1}^N \Cc\Wc_{c^i}(\X^i,\Y^\nu),
    \]
    since we have for all $n \in \N$
    \[
        \inf_{\nu \in \Pc(\Yc^n)} \sum_{i = 1}^N \Cc\Wc_{c^i}(\X^i,\Y^\nu)  = 
        \inf_{\Y \in {\rm FP}(\Yc^n)} \sum_{i = 1}^N \Cc\Wc_{c^i_{|\Yc^n}}(\X^i,\Y^\nu) \ge 
        \inf_{\Y \in {\rm FP}(\Yc)} \sum_{i = 1}^N \Cc\Wc_{c^i}(\X^i,\Y^\nu).
    \]

    \medskip As weak duality holds, see \Cref{prop:weak_duality}, it remains to show that
    \[
        \lim_{n \to \infty} \sup_{f_1^{1:N} \in \Phi^0(c^{1:N}_{|\Yc^n})} \sum_{i = 1}^N \int_{\O^i_1} f^i_1 \mathrm d\P^i =
        \sup_{f_1^{1:N} \in \Phi^0(c^{1:N})} \sum_{i = 1}^N \int_{\O^i} f^i_1 \mathrm d\P^i.
    \]
    To this end, for every $n \in \N$, let $f^{1:N,n} \in \Phi^0(c^{1:N}_{|\Yc^n})$ be an optimizer to the barycenter problem with costs $c^{1:N}_{|\Yc^n}$, which exists by Part A of this proof as $\Yc^n$ is compact.
    Fix $\nu \in \Pc(\Yc)$ and write $\nu_n \coloneqq (\xi^n)_\# \nu$.
    Then, there are $\tilde g^{1:N,n} \in \Gc_{\Yc^n}^N$ and $\tilde f_{2:T}^{1:N,n} \in \Ac_{1:N,2:T}^{\Yc}(\Yc^n),$ where $\Ac_{1:N,2:T}^{\Yc}(\Yc^n)$ is the set of martingale compensators corresponding to the space $\Yc^n$, such that
    \[
        f_1^{i,n} + \sum_{t = 2}^T \tilde f_t^{i,n} + \tilde g^{i,n} \le c^i_{|\Yc^n}
        \text{ and }
        \sum_{i = 1}^N \tilde g^{i,n} = 0 \quad \cplc(\X^i,\Y^{\nu_n})\text{--q.s.}
    \]
    For $(\o^i,y) \in \O^i \times \Yc$ and $t \in \{2,\dots,T\}$ we set
    \[
        f_t^{i,n}(\o^i_{1:t},y_{1:t-1}) \coloneqq \tilde f_t^{i,n}(\o^i_{1:t}, \xi^n_{1:t-1}(y_{1:t-1}))
        \quad \text{and} \quad
        g^{i,n}(y) \coloneqq \tilde g^{i,n}(\xi^n(y)).
    \]
    We observe that $f_{2:T}^{1:N,n} \in \Ac_{1:N,2:T}^{\Yc}$ and $g^{1:N,n} \in \Gc_{\Yc}$ and that
    \[
        f_1^{i,n} + \sum_{t = 2}^T f_t^{i,n} + g^{i,n} \le c^{i,n}\text{ and } \sum_{i = 1}^N g^{i,n} = 0
        \quad\cplc(\X^i,\Y^\nu)\text{--q.s.,}        
    \]
    where we used that if $\pi \in \cplc(\X^i,\Y^\nu)$ then $((\o^i,y) \mapsto (\o^i,\xi^n(y)))_\# \pi \in \cplc(\X^i,\Y^{\nu_n})$.
    The latter follows readily by $\xi^n_t$ only depending on $y_t$ for every $t \in \{1,\dots,T\}$.
    Hence, as $\nu \in \Pc(\Yc)$ was arbitrary, we have shown that $f_1^{1:N,n} \in \Phi^0(c^{1:N,n})$.
    This allows us to invoke \Cref{lem:step3.attain_bary} to find $f_1^{1:N} \in \Phi^0(c^{1:N})$ with
    \[
        S \le \lim_{n \to \infty} \sum_{i = 1}^N \int_{\O^i_1} f_1^{i,n} \mathrm d\P^i \le
        \sum_{i = 1}^N \int_{\O^i_1} f_1^i \mathrm d\P^i \le D(c^{1:N}) \le V(c^{1:N}) \le S.
    \]
    Again, we conclude that $S = D(c^{1:N}) = V(c^{1:N})$ and that $f_1^{1:N}$ is a dual optimizer.
\end{proof}

\begin{lemma} \label{lem:step3.attain_bary}
   In the setting of {\rm \Cref{thm:attain_bary}}, for every $(i,n) \in \{1, \dots, N\} \times \N$, let $c^{i,n} : \O^i \times \Yc \longrightarrow [0,\infty)$ be measurable such that $c^i(\o^i,y) \le \ell^i(\o^i) + k(y)$, $(\o^i,y) \in \O^i \times \Yc$, and $c^{1:N,n} \longrightarrow c^{1:N}$ holds pointwise.
    Further, for $n \in \N$, let $f_1^{1:N,n} \in \Phi^0(c^{1:N,n}).$
    Then, there exists $f_1^{1:N} \in \Phi^0(c^{1:N})$ such that
    \begin{equation}
        \label{eq:step3.attain_bary.2}
        \sum_{i = 1}^N \int_{\O^i_1}f_1^{i}(\o^i_1) \P^i(\mathrm d\o^i) \ge \limsup_{n \to \infty} \sum_{i = 1}^N \int_{\O^i_1}f_1^{i,n}(\o^i_1)\P^i(\mathrm{d}\o^i_1).
    \end{equation}
\end{lemma}

\begin{proof}
    Without loss of generality, we can assume that $S\coloneqq \limsup_{n \to \infty} \sum_{i = 1}^N \int_{\O^i_1}f_1^{i,n}(\o^i_1) \P^i(\mathrm{d}\o^i_1)> -\infty,$  $\int \ell^i \mathrm d\P^i = 0$ for $i \in \{1,\dots,N\}$, and that the limit superior in \eqref{eq:step3.attain_bary.2} is simply a limit, and also that for all $n \in \N$
    \begin{equation} \label{bound2}
        \sum_{i=1}^N \int_{\O^i_1} f_1^{i,n} \mathrm{d}\P^i \geq S-1.
    \end{equation}

    \medskip \emph{Step 1} (Construction of $f_1^{1:N}$):
    Let $\nu \in \Pc(\Yc)$ be such that $k \in \L^1(\nu)$. 
    Then, by definition of $\Phi^0(c^{1:N})$ and using \Cref{rem:qs_bary}, there are functions $g^{i,n} \in \Gc_\Yc$ and $f^{i,n}_{2:T} \in \Ac_{2:T,i}^{\Yc}$ and a $\cplc(\X^i,\Y^\nu)$--full Borel set $A^{i,0} = \boxtimes_{t = 1}^T (A^{i,t}_\bullet \times B^t_\bullet)$ such that everywhere on $A^{i,0}$ holds
    \begin{equation} \label{bound}
        f_1^{i,n}(\o^i)+\sum_{t=2}^T f_t^{i,n}(\o^i_{1:t},y_{1:t-1}) +g^{i,n}(y) \leq c^{i,n}(\o^i,y) \text{ and } \sum_{i = 1}^N g^{i,n}(y) = 0.
    \end{equation}
    By shifting $f^{i,n}$ and $g^{i,n}$ by constants, we assume without loss of generality that for all $(i,j,n) \in \{1,\ldots,N\}^2 \times \N$
    \[ 
        \int_{\O^i_1} f_1^{i,n} \mathrm{d}\P^i=\int_{\O^j_1} f_1^{j,n} \mathrm{d}\P^j.
    \] 
    We note that the projections of $A^{i,0}$ onto $\O^i_{1:t} \times \Yc$ are Borel by \Cref{rem:qs_bary}. Moreover, arguing similarly as in \eqref{eqn:bound_y}, we can replace $\tilde B$ by a $\nu$-full subset $B$ such that the projection of $B$ on $\Yc_{1:t},$ denoted by $B^t,$ is Borel and the map
    \begin{equation}
        \label{eqn:bary_inf_meas}
        B^t \ni y^{1:t} \longmapsto \inf_{\tilde y \in \tilde B: \, y_{1:t} = \tilde y_{1:t}} k(\tilde y)
    \end{equation} is Borel measurable.
    Integrating both sides of inequality \eqref{bound} with respect to $\P^i(\mathrm{d}\o^i)$ and using $c^i\leq k$ gives 
    \[
        \int_{\O_1^i} f^{i,n}_1(\o^i_1)\P^i(\mathrm{d}\o^i_1)\leq k(y)-g^{i,n}(y),
     \]
    where $(i,n,y) \in \{1,\ldots,N \} \times \N \times B$, whence, using \eqref{bound2} we find
    \[ 
        g^{i,n}(y) \leq - \int_{\O^i_1} f_1^{i,n}\mathrm{d}\P^i+k(y)\leq - \frac{S-1}{N}+k(y).
    \] 
    Since $\sum_{i=1}^N g^{i,n}=0$, the preceding inequality yields 
    \[ 
        g^{i,n}(y) = - \sum_{j \neq i} g^{j,n}(y) \geq  (N-1)\frac{(S-1)}{N} - (N-1) k (y). 
    \]
    Consequently, we obtain
    \begin{equation} \label{wbound} 
    \left\lvert g^{i,n}(y) \right\rvert \leq \lvert S-1 \rvert+ N k(y), \quad (i,n,y) \in \{1,\ldots,N \} \times \N \times  B.
    \end{equation}
    For $(i,n) \in \{1,\ldots,N \} \times \N$, integrating both sides of the inequality \eqref{bound} with respect to $K^{i}_2(\o_1^i;\mathrm{d}\o^i_{2:T})$ and using $c^i \leq \ell^i + k$ and the bound \eqref{wbound} yield 
    \begin{equation} \label{boundofvarphi}
        f_1^{i,n}(\o_1^i)\leq \ell^i(\o_1^i) + \lvert S-1 \rvert + (N + 1) k(y), \quad (\o^i_1,y) \in A^{i,1} \times B.
    \end{equation} 
    Combining \eqref{boundofvarphi} with \eqref{bound2} we conclude
    \[ 
        \int_{\O^i_1}  \lvert f_1^{i,n}\rvert \mathrm{d}\P^i=2\int_{\O^i_1} (f_1^{i,n})^+ \mathrm{d}\P^i-\int_{\O^i_1}  f_1^{i,n} \mathrm{d}\P^i\leq 2\Big(\int_{\O^i_1} |\ell^i| \mathrm d\P^i + \lvert S-1 \rvert+ (N + 1)k(y)\Big)+\frac{\lvert S-1 \rvert}{N}.
    \]
Thus, we can apply Koml\'{o}s' lemma \cite[Theorem 1]{Ko67} to obtain that there exists a subsequence such that its Ces\`aro means converge to a limit $f_1^i \in \L^1(\Fc_1^i,\P^i)$ $\P^i$--almost surely. Redefining $A^{i,1}$ if necessary, we may assume that the convergence holds everywhere on $A^{i,1}$, and we set $f_1^i=-\infty$ outside of $A^{i,1}.$
For simplicity, we shall denote the sequence of Ces\`aro means of this subsequence by $(f_1^{i,n})_{n \in \N}$. Taking the Ces\`aro means of the same subsequences of $(g^{i,n})_{n \in \N},$ $(h^{i,n})_{n \in \N}$, and $(c^{i,n})_{n \in \N},$ where by abuse of notation we write $g^{i,n}$, resp.\ $h^{i,n}$ and $c^{i,n}$, for the corresponding Ces\`aro means, all the previously shown inequalities remain true. 
Moreover, due to the uniform upper bound \eqref{boundofvarphi} we can invoke Fatou's lemma and get
\[ 
    \sum_{i=1}^N \int_{\O^i_1} f^i_1\mathrm{d}\P^i = \int_{\O_1^1\times \cdots \times \O^N_1} \limsup_{n \rightarrow \infty} \Big( \sum_{i=1}^N f_1^{i,n}\Big)\mathrm{d}(\P^1 \otimes \cdots\otimes \P^N) \geq \limsup_{n \rightarrow \infty} \sum_{i=1}^N \int_{\O^i_1} f_1^{i,n}\mathrm{d}\P^i = S.
\] 
It remains to verify that $f^i_1$, $i \in \{1,\ldots,N\}$, are admissible, \emph{i.e.}, there are $g^i \in \Gc_{\Yc}$ and $f_{2:T}^i \in \Ac_{2:T,i}^{\Yc}$ such that 
\[
   \sum_{i = 1}^N  g^i(y) = 0 \text{ and } f^i_1(\o^i_1) + \sum_{t=2}^T f_2^i(\o^i_{1:t},y_{1:t-1}) +g^i(y) \leq c^i(\o^i,y),\quad \cplc(\X^i,\Y^\nu)\text{--q.s}.
\] 
We point out that convergence of the sequence $(f^{i,n}_1)_{n \in \N}$ is clearly independent of the choice of $\nu$. That is to say, if $\tilde \nu \in \Pc(\Yc)$ and $\tilde g^{i,n} \in \Gc_\Yc$ and $ \tilde f^{i,n}_{2:T} \in \Ac_{2:T,i}^{\Yc}$, $i \in \{1,\ldots,N\}$ are such that
    \begin{equation*} 
        f_1^{i,n}(\o^i)+\sum_{t=2}^T \tilde f_t^{i,n}(\o^i_{1:t},y_{1:t-1}) + \tilde g^{i,n}(y) \leq c^{i,n}(\o^i,y) \text{ and } \sum_{i = 1}^N \tilde g^{i,n}(y) = 0,\quad \cplc(\X^i,\Y^{\tilde \nu})\text{--q.s.,}
    \end{equation*}  
        then passing to the subsequence found above and taking Ces\`aro means, we obtain that $f^{i,n}_1 \longrightarrow f^i_1$ on $A^{i,1}.$ 

    \medskip \emph{Step 2} (Construction of $g^{1:N}$):
    The bound established in \eqref{wbound} and using $k \in \L^1(\nu)$ yields that \[\sup_{n \in \N} \int_{\Yc} |g_1^{i,n}| \mathrm d \nu < \infty,\quad i \in \{1,\dots,N\}. \]
    We can thus employ Koml\'{o}s' lemma \cite[Theorem 1]{Ko67} to obtain that the Ces\`{a}ro means of a subsequence of $(g^{i,n})_{n \in \N}$ converge $\nu$--almost everywhere.
    Thus, replacing the original sequence as well as the other sequences with Ces\`{a}ro means of this subsequence, we may assume without loss of generality that $(g^{i,n})_{n \in \N}$ converges on a $\nu$--full Borel subset of $B$ to a limit $g^i$, for every $i \in \{1,\ldots,N\}.$
    Potentially by replacing $B$ with that subset, we can assume that this convergence holds everywhere on $B$, and set $g^{i}$ to $0$ on the complement.

    \medskip \emph{Step 3} (Construction of martingale compensators):
    Observe that the only constraint that is coupling, for $i \in \{1,\dots,N\}$, the inequalities in \eqref{bound}, is the congruency condition $\sum_{i = 1}^N g^{i,n}(y) = 0$.
    In the previous step we have constructed suitable sequences such that $g^{1:N,n} \longrightarrow g^{1:N}$ pointwise on $B$ and the limits satisfy the congruency condition.
    Hence, for the rest of the proof these inequalities completely decouple, which allows us to use the same construction as in Step 3.2 and Step 3.3 of \Cref{thm:caus_dual} and thereby find suitable martingale compensators $f_{2:T}^{1:N}$ such that on the $\cplc(\X^i,\Y^\nu)$--full Borel set $\boxtimes_{t = 1}^T (A^{i,t}_\bullet \times B^t_\bullet)$ we have
    \[
        f_1^i(\o^i_1) + \sum_{t = 2}^T f_t^i(\o^i_{1:t},y_{1:t-1}) + g^i(y) \le c^i(\o^i,y).
    \]

    \medskip Since $\nu \in \Pc(\Yc)$ with $k \in \L^1(\nu)$ was arbitrary, we conclude by \Cref{lem:attain_bary.k_int} that $f_1^{1:N} \in \Phi^0(c^{1:N})$.
\end{proof}

\begin{lemma} \label{lem:attain_bary.k_int}
    In the setting of {\rm\Cref{thm:attain_bary}}, let $f_1^{1:N} \in \Ac_1^{\Yc}$ be such that for every $\nu \in \Pc(\Yc)$ with $k \in \L^1(\nu)$ there are $(f_{2:T}^{1:N},g^{1:N}) \in \Ac_{1:N,2:T}^{\Yc} \times \Gc_{\Yc}^N$ with
    \begin{equation}
        \label{eq:attain_bary.k_int.1}
        \sum_{t = 1}^T f_t^i + g^i \le c^i \text{ and } \sum_{i = 1}^N g^i = 0 \quad \cplc(\X^i,\Y^\nu)\text{--{\rm q.s.}}
    \end{equation}
    for every $i \in \{1,\dots,N\}$.
    Then, $f_1^{1:N} \in \Phi^0(c^{1:N})$.    
\end{lemma}

\begin{proof}
    To show the claim, let $\nu \in \Pc(\Yc)$ be arbitrary.
    As $\Yc$ is $\sigma$-compact, we have that $\Yc = \bigcup_{n \in \N} \Yc^n$ where $\Yc^n$ is compact.
    Since $k_{|_{\Yc^n}}$ is bounded for every $n \in \N$, we can find weights $(w_n)_{n \in \N} \in [0,1]^\N$ with $\sum_{n \in \N} w_n = 1$ such that
    \[
        \tilde \nu \coloneqq \sum_{n \in \N} w_n \nu_{|_{\Yc^n}} \gg \nu \text{ and }k \in \L^1(\tilde \nu).
    \]
    By assumption, there are $(f_{2:T}^{1:N},g^{1:N}) \in \Ac_{1:N,2:T}^{\Yc} \times \Gc_{\Yc}^N$ so that \eqref{eq:attain_bary.k_int.1} holds $\cplc(\X^i,\Y^{\tilde \nu})$--q.s.
    But, it is immediate from the representation of polar sets in \Cref{thm:polar_causal} that every $\cplc(\X^i,\Y^{\tilde \nu})$--full set is also $\cplc(\X^i,\Y^\nu)$--full, because $\nu \ll \tilde \nu$.
    Hence, \eqref{eq:attain_bary.k_int.1} holds also $\cplc(\X^i,\Y^\nu)$--q.s., which concludes the proof.
\end{proof}

In the proof of \Cref{thm:attain_bary} the construction of suitable martingale compensators for the dual optimizers forces us to fix a distribution $\nu \in \Pc(\Yc)$.
This part can be strengthened when working under the continuum hypothesis which then allows us to do a simplified construction based on a transfinite induction over the set of all causal couplings with arbitrary second marginal, \emph{i.e.}, the set $\cplc(\X^i,\ast) \coloneqq \bigcup_{\nu \in \Pc(\Yc)} \cplc(\X^i,\Y^\nu)$. Consequently, the optimal martingale compensators can be `aggregated' over all couplings. We refer an interested reader \emph{e.g.}\ to \cite[Chapter 4]{Ciesielski_1997} for more details.

\begin{remark} We remark that when working with a non-dominated set of probabilities, the fact that `aggregation' of some terms is only possible only under some additional set-theoretic axioms appears throughout the literature. We refer for instance to  {\rm \citeauthor*{BaChKu19} \cite{BaChKu19}} and {\rm \citeauthor*{Nu11} \cite{Nu11}}.
\end{remark} 

\medskip Let us define for $i \in \{1,\dots,N\}$ and $t \in \{2,\dots,T\}$ the following sets of functions
\begin{align*}
    \Ac_{i,t}^{\Yc, u} &\coloneqq \bigg\{ 
    f_t^i : \O^i_{1:t} \times \Yc_{1:t-1} \longrightarrow \R \, \bigg\vert  \, f_t^i \text{ is $\cplc(\X^i,\ast)$--universally measurable and $\cplc(\X^i,\ast)$--q.s.\ holds} \\ 
    &\phantom{\coloneqq \Big\{ f_t^i\,\Big\vert\,}\qquad
    f^i_t(\o^i_{t-1},\,\cdot\,,y_{1:t-1}) \in \L^1(K_t^i(\o^i_{1:t-1};\,\cdot\,)) \text{ and } \int_{\O^i_{t}} f^i_t(\o^i_{t-1},\tilde \o^i_t ,y_{1:t-1}) K^i_t(\o^i_{1:t-1};\mathrm{d}\tilde \o^i_t) = 0 \bigg\}, \\
    \Ac_{i,1}^{\Yc,u} &\coloneqq \Big\{ f_1^i(\o^i,y) = a^i_1(\o^i_1) \,\Big\vert\, a_1^i \in \L^1(\Fc_1^i,\P^i) \Big\}.
\end{align*}
Furthermore, we set $\Gc_\Yc^{\rm u} \coloneqq \{ g : \Yc \longrightarrow  \R \,\vert\, g \text{ is universally measurable} \}$ and define the set of admissible dual potentials by
\begin{multline}
    \Phi^{\rm u}(c^{1:N}) \coloneqq \Big\{ f_1^{1:N} \in \prod_{i = 1}^N \L^1(\Fc_1^i,\P^i) \,\Big\vert\, \exists f_{2:T}^i \in \Ac_{2:T}^{\Yc, u},\,\exists g^i \in \Gc_\Yc^{\rm u}, i \in \{1,\dots,N\} : \\
    f_1^i + \sum_{t = 2}^T f_t^i + g^i \le c^i,\, \sum_{i = 1}^N g^i = 0 \Big\}.
\end{multline}

\begin{remark}[On universal measurability] \label{rem:universal_measurability}
    $(i)$ A function $g : \Yc \longrightarrow \R$ is called universally measurable if, for every $\nu \in \Pc(\Yc)$, $g$ is measurable with respect to the $\nu$-completion of the Borel $\sigma$-algebra on $\Yc$.
    The latter can be equivalently expressed as: for every $\nu \in \Pc(\Yc)$, there exists a Borel function $g^{\nu} : \Yc \longrightarrow \R$ with $g = g^{\nu}$ $\nu$--a.s.

    \medskip $(ii)$ Similarly, a function $a : \O^i \times \Yc \to \R$ is called $\cplc(\X^i,\ast)$--universally measurable if, for every $\pi \in \cplc(\X^i,\ast)$, $a$ is measurable with respect to the $\pi$-completion of the Borel $\sigma$-algebra on $\O^i \times \Yc$.
    Again, the latter can be equivalently expressed as: for every $\pi \in \cplc(\X^i,\ast)$, there exists a Borel function $a^{\pi} : \O^i \times \Yc \longrightarrow \R$ with $a = a^{\pi}$ $\pi$--a.s.

    \medskip $(iii)$ Let $a: \O^i \times \Yc \longrightarrow \R \cup \{ - \infty \}$ and let $k : \Yc \longrightarrow \R^+$ be a function such that $a$ is measurable with respect to the $\pi$-completion of the Borel $\sigma$-algebra on $\O^i \times \Yc$ for every $\pi \in \cplc(\X^i,\ast)$ which finitely integrates $k$.
    We claim that if $\Yc$ is $\sigma$-compact and $k$ is bounded on compacts, $a$ is already $\cplc(\X^i,\ast)$--universally measurable.
    Indeed, let $\Yc^K,$ $K\in\N,$ be compacts such that $\Yc^K \nearrow \Yc$ and $\sup_{y \in \Yc^K} |k(y)| \leq K.$ Indeed, these exist since for any sequence of compacts $\tilde \Yc^i \nearrow \Yc$ with $\tilde \Yc^1=\{y_0\}$ for some $y_0 \in \Yc$, we have that $s_i\coloneqq\sup_{y \in \tilde \Yc^i} |k(y)|<\infty$ and we can without loss of generality assume $s_1=1.$ It then suffices to set $i_1=1$, and inductively define
    \[
        i_{K+1}\coloneqq 
    \begin{cases} i_K &\text{if } s_{i_K+1}>K, \\
    i_K+1 &\text{else,}
    \end{cases}
    \quad \text{and} \quad
    \Yc^K\coloneqq \tilde \Yc^{i_K}.
    \]
    Let now $\pi \in \cplc(\X^i,\ast)$ be arbitrary.
   It is immediate that
    \[
        \pi^K \coloneqq \big((\o^i,y)\mapsto(\o^i,\xi(y)\big)_\# \pi \in \cplc(\X^i,\ast),
    \] where $\xi$ is as in the proof of {\rm\Cref{thm:attain_bary}}.
    It is also immediate that $\int k^i(y) \mathrm{d}\pi^K \le K$.
    By construction we have that $\pi \ll \tilde \pi \coloneqq \sum_{n \in \N} 2^{-n} \pi^{n} \in \cplc(\X^i,\ast)$ and $\int k^i \mathrm{d}\tilde \pi \le \sum_{n \in \N} n 2^{-n} < \infty$.
    Hence, there is a Borel function $\tilde a : \O^i \times \Yc \longrightarrow \R \cup \{ -\infty \}$ such that $a = \tilde a$ $\tilde \pi$--a.s., and therefore also $\pi$--a.s.
    We conclude that $a$ is $\cplc(\X^i,\ast)$--universally measurable. 
\end{remark}

\begin{theorem} \label{thm:attain_bary.CH}
    In the setting of {\rm\Cref{thm:attain_bary}}, assume {\rm ZFC}\footnote{That is, assume the Zermelo–Fraenkel set theory axioms and the axiom of choice.} and the continuum hypothesis.
    Then, we have
    \[
        \inf_{\Y \in {\rm FP}(\Yc)} \sum_{i = 1}^N \Cc\Wc_{c^i}(\X^i,\Y) = \sup_{f_1^{1:N} \in \Phi^{\rm u}(c^{1:N})} \int_{\O^i_1} f_1^i(\o^i_1) \P^i(\mathrm d\o^i_1).
    \]
    Moreover, there is $f^{1:N}_1 \in \Phi^{\rm u}(c^{1:N})$ such that
    \[
        \inf_{\Y \in {\rm FP}(\Yc)} \sum_{i = 1}^N \Cc\Wc_{c^i}(\X^i,\Y) = \sum_{i = 1}^N \int_{\O^i_1} f_1^i(\o^i_1) \P^i(\mathrm{d}\o^i_1).
    \]
\end{theorem}

\begin{proof}
    This proof works analogously to the proof of \Cref{thm:attain_bary} with obvious modifications.
    In particular, we replace \Cref{lem:step3.attain_bary} with \Cref{lem:step3.attain_bary.CH}.
\end{proof}

\begin{lemma} \label{lem:step3.attain_bary.CH}
    In the setting of {\rm\Cref{thm:attain_bary.CH}}, for every $(i,n) \in \{1,\dots,N\} \times \N$, let $c^{i,n} : \O^i \times \Yc \longrightarrow [0,\infty)$ be measurable such that $c^{i,n}(\o^i,y) \le \ell^i(\o^i) + k(y)$, $(\o^i,y) \in \O^i \times \Yc$, and $c^{1:N,n} \longrightarrow c^{1:N}$ holds pointwise.
    Further, for $n \in \N$, let $f_1^{1:N,n} \in \Phi^u(c^{1:N,n}).$ Then, there exists $f_1^{1:N} \in \Phi^u(c^{1:N})$ such that
    \begin{equation}
        \label{eq:lem.step3.attain_bary.CH.2}
        \sum_{i = 1}^N \int_{\O^i_1} f_1^i \mathrm d\P^i \ge \limsup_{n \to \infty} \sum_{i = 1}^N \int_{\O^i_1} f_1^{i,n} \mathrm d\P^i.
    \end{equation}
\end{lemma}

\begin{proof}
    Without loss of generality, we can assume that $S\coloneqq \limsup_{n \to \infty} \sum_{i = 1}^N \int_{\O^i_1} f_1^{i,n} \mathrm d\P^i > -\infty$ and that the limit superior in \eqref{eq:lem.step3.attain_bary.CH.2} is simply a limit, and also that for all $n \in \N$
    \begin{equation}
        \label{eq:lem.step3.attain_bary.CH.3}
        \sum_{i = 1}^N \int_{\O^i_1} f_1^{i,n}(\o^i_1) \P^i(\mathrm d\o^i_1) \ge S - 1.
    \end{equation}
    \emph{Step 1} (Construction of $f_1^{1:N}$): This step can be done analogously to Step 1 of \Cref{lem:step3.attain_bary} and, for every $i \in \{1,\dots,N\}$, we can assume without loss of generality that $f_1^{i,n} \longrightarrow f_1^i$ $\P^i$--almost surely with
    \[  
        \sum_{i = 1}^N \int_{\O^i_1} f_1^i(\o^i_1)\P^i(\mathrm{d}\o_1^i) \geq S.
    \]
    It remains to show that $f_1^{1:N} \in \Phi^{\rm u}(c^{1:N})$.

    \medskip \emph{Step 2} (Construction of $g^{1:N}$ and $f_{2:T}^{1:N}$): 
    Thanks to the continuum hypothesis and the axiom of choice, there exists a bijection of the set of countable ordinal numbers $\mathscr O$ and the set $\{ \pi^i \in \cplc(\P^i,\ast) \,\vert\, k \in \L^1(\pi^i)\}$, denoted by $\mathscr O \ni \alpha \longmapsto \pi^{\alpha,i}$. 
    We proceed by a transfinite induction over $\mathscr O$.

    \medskip \emph{Step 2.1} (Successor case): Let $\alpha \in \mathscr O$ be such that there is a sequence $(f_{1:N}^{\alpha,1:N,n},g^{\alpha,1:N,n})_{n \in \N}$ with $f_{1:T}^{\alpha,1:N,n} \in \Ac_{1:T}^{\Yc, u}$ and $g^{\alpha,1:N,n} \in \Gc^{\rm u}_\Yc$ such that for all ordinals $\beta \le \alpha$, $t \in \{2,\ldots,T\}$ and $i \in \{1,\ldots,N\}$, 
    \begin{gather}
        \label{eq:transfinite_induction.1}
        (f_t^{\alpha,i,n})_{n \in \N} \text{ and } (g^{\alpha,i,n})_{n \in \N} \text{ converge} 
        \; \pi^{\beta,i}\text{--a.s.,}
    \end{gather}
     as well as, for every $n \in \N$ and $\gamma \in \mathscr O$,
    \begin{gather}
        \label{eq:transfinite_induction.2}
        f^{\alpha,i,n}_1 + \sum_{t = 2}^T f_t^{\alpha,i,n} + g^{\alpha,i,n} \le c^{\alpha,i,n}\enspace\text{$\pi^{\gamma,i}$--a.s.}, \quad f_1^{\alpha,i,n} \longrightarrow f_1^i \enspace \text{$\P^i$--a.s.,}
        \\
        \label{eq:transfinite_induction.2.5}
         c^{\alpha,i,n} \le \ell^i+ k, \quad c^{\alpha,n,i} \longrightarrow c^i \quad\text{pointwise,}\quad 
        \text{and}\quad\sum_{i = 1}^N g^{\alpha,i,n} = 0.
    \end{gather}
    As in \Cref{rem:baryc_L1} and the proof of \Cref{thm:caus_dual}, we find that \[\sup_{n \in \N} \int \Big[ \sum_{t = 2}^T |f_t^{\alpha,i,n}| + |g^{\alpha,i,n}| \Big] \mathrm{d}\pi^{\alpha + 1,i} < \infty. \]
    By Koml\'os' lemma, there is a subsequence of $(f_{2:T}^{\alpha,1:N,n},g^{\alpha,1:N,n})_{n \in \N}$ such that its Ces\'aro means, denoted by $(f_{2:T}^{\alpha + 1,1:N,n},g^{\alpha + 1,1:N,n})_{n \in \N}$,
    satisfy that $(f_{2,T}^{\alpha + 1, i, n},g^{\alpha + 1, i ,n})_{n \in \N}$ converges $\pi^{\alpha + 1,i}$--a.s., for every $i \in \{1,\dots,N\}.$
    We denote the Ces\'aro means of the same subsequence of $(f_1^{\alpha,1:N,n}, c^{\alpha,1:N,n})_{n \in \N}$ by $(f_1^{\alpha + 1,1:N,n},c^{\alpha + 1, 1:N, n})_{n \in \N}$.
    Then we find that the hereby constructed sequence still satisfies \eqref{eq:transfinite_induction.2} and $\eqref{eq:transfinite_induction.2.5}$ as well as \eqref{eq:transfinite_induction.1} for all $\beta \le \alpha + 1$.
    In particular, their limits (where they exist) are consistent in the sense that
    \begin{multline}
        \label{eq:transfinite_induction.3}
        \big\{ (\o^i,y) \in \O^i \times \Yc \,\big\vert\, \forall (i,t)\; \lim_{n\rightarrow \infty} f_t^{\alpha,i,n} \text{ and } \lim_{n\rightarrow \infty} g^{\alpha,i,n} \text{ exist} \big\} \subseteq \\
        \big\{ (\o^i,y) \in \O^i \times \Yc \,\big\vert\, \forall (i,t)\; \lim_{n\rightarrow \infty} f_t^{\alpha + 1,i,n} \text{ and } \lim_{n\rightarrow \infty} g^{\alpha + 1,i,n} \text{ exist} \big\},
    \end{multline}
    and the respective limits coincide (where they exist).

    \medskip \emph{Step 2.2} (Limit case): Next, assume that \eqref{eq:transfinite_induction.1}-\eqref{eq:transfinite_induction.3} hold for all $\beta < \alpha$ where $\alpha \in \mathscr O$ is a limit ordinal.
    Since $\alpha$ is a countable ordinal number, there exists a bijection from $\mathbb N$ to $\{ \beta \in \mathscr O : \beta < \alpha \}$.
    Thus, there is an increasing sequence $(\beta_n)_{n \in \N}$ with $\sup_n \beta_n = \alpha$.
    For $n \in \N$, $i \in \{1,\ldots,N\}$ and $t \in \{1,\ldots,T\}$ we choose the diagonal sequence by setting
    \[
        \tilde f_t^{i,n} \coloneqq f_t^{\beta_n,i,n}, \quad
        \tilde g^{i,n} \coloneqq g^{\beta_n,i,n}, \text{ and } \tilde c^{n,i} \coloneqq c^{\beta_n,i,n}.
    \]
    As in the previous step, the sequence $(\tilde f_{2:T}^{1:N,n},\tilde g^{1:N,n})_{n \in \N}$ admits a subsequence whose Ces\'aro means, denoted by $(f_{2:T}^{\alpha,1:N,n},g^{\alpha, 1:N,n})_{n \in \N}$, satisfy that $(f_{2:T}^{\alpha,i,n}, g^{\alpha,i,n})_{n \in \N}$ converges $\pi^{\alpha,i}$--a.s., for every $i \in \{1,\dots,N\}$.
    Again, we denote the corresponding Ces\'aro means of the same subsequence of $(\tilde f_1^{1:N,n}, \tilde c^{1:N,n})_{n \in \N}$ by $(f^{\alpha,1:N,n},c^{\alpha,1:N,n})_{n \in \N}$. 
    As we chose $(\beta_n)_n$ increasing with limit $\alpha$, this sequence of functions also satisfies \eqref{eq:transfinite_induction.1} and \eqref{eq:transfinite_induction.3} with respect to all $\beta < \alpha$ as well as \eqref{eq:transfinite_induction.2} and \eqref{eq:transfinite_induction.2.5}.

    \medskip \emph{Step 2.3}: As the constructed limits are consistent, the following functions are well-defined
    \begin{align*}
        \tilde f_t^i &\coloneqq 
        \begin{cases}
           \displaystyle \lim_{n \rightarrow \infty} f_t^{\alpha,i,n} & \exists \alpha \in \mathscr O \text{ where the limit exists in }\R, \\
            -\infty & \text{otherwise},
        \end{cases} \\
        g^i & \coloneqq
        \begin{cases}
            \displaystyle\lim_{n \rightarrow \infty} g^{\alpha,i,n} & \exists \alpha \in \mathscr O \text{ where the limit exists in }\R, \\
            -\infty & \text{otherwise},
        \end{cases}            \end{align*}
    and satisfy for every $\alpha \in \mathscr O$, $\pi^{\alpha,i}$--a.s.
    \begin{equation}
        \label{eq:transfinite_induction.tilde.ineq}
        f_1 + \sum_{t = 2}^T \tilde f_t^i + g^i \le c^i \text{ and } \sum_{i = 1}^N g^i = 0.
    \end{equation}
    Observe that $g^i$ has to be everywhere finitely valued, since for every $y \in \mathcal Y$ there exists $\alpha \in \mathscr O$ with $\pi^{\alpha,i} = \mathbb P^i \otimes \delta_y \in \cplc(\X^i,\ast)$ and $\int k \mathrm{d}\pi^{\alpha,i} = k(y) < \infty$.
    In particular, $g^i$ is a universally measurable function on $\Yc$ and $\sum_{i = 1}^N g^i = 0$ everywhere.
    Furthermore, we have that $\tilde f_t^i$ is $\cplc(\X^i,\ast)$--universally measurable by \Cref{rem:universal_measurability}.
    
    \medskip 
    Because we have for fixed $\alpha \in \mathscr O$ that $f_t^{\alpha,i,n} \longrightarrow \tilde f_t^i$ $\pi^{\alpha,i}$--a.s. By similar arguments as in \eqref{wbound} and by passing to the limit, we obtain $g^i \in \L^1(\nu)$ for every $\nu$ such that $k \in \L^1(\nu).$ We conclude as in \eqref{eqn:bound_t} and \eqref{mart2} that $\tilde f_t^i \in \L^1(\pi^{\alpha,i}).$ Indeed, one can inductively verify that the right-hand side in \eqref{eqn:bound_t} is $\pi^{\alpha,i}$-integrable and this inequality is preserved by passing to the limit. It follows that  $\int \tilde f_t^i \mathrm d \pi^{\alpha,i}$ is well defined and \eqref{mart2} gives $\int \tilde f_t^i \mathrm d \pi^{\alpha,i} \geq 0$ by Fubini's theorem. Moreover, as in \eqref{mart2} and due to Fatou's lemma the function
    \[
        f_t^i(\o^i_{1:t},y_{1:t-1}) \coloneqq
        \begin{cases}
            \displaystyle \tilde f_t^i(\o^i_{1:t},y_{1:t-1}) - \int_{\O^i_t} \tilde f_t^i(\o^i_{1:t-1},\tilde \o^i_t,y_{1:t-1}) K_t^i(\o^i_{1:t-1};\mathrm{d}\tilde \o^i_t) & \text{if well-defined},\\
            -\infty &\text{otherwise,}
        \end{cases}
    \]
    is $\pi^{\alpha,i}$--a.e.\ finite and satisfies $f_t^i \le \tilde f_t^i$ $\pi^{\alpha,i}$--a.s.\
    Consequently, we have $\pi^{\alpha,i}$--a.s.
    \begin{equation}
        \label{eq:transfinite_induction.end}
        f_1 + \sum_{t = 2}^T f_t^i + g^i \le f_1 + \sum_{t = 2}^T \tilde f_t^i + g^i \le c^i.
    \end{equation}
    We deduce from \Cref{rem:universal_measurability} that $f_t^i$ is $\cplc(\X^i,\ast)$--universally measurable and \eqref{eq:transfinite_induction.end} holds $\cplc(\X^i,\ast)$--quasi-surely.
    Hence, $f_t^i \in \Ac_{i,t}^{\Yc, u}$ for all $(t,i) \in \{1,\dots,T\} \times \{1,\dots, N\}$, from where we conclude that $f_1^{1:N} \in \Phi^{\rm u}(c^{1:N})$.
\end{proof}

\subsection{Bicausal barycenters} \label{sec:bcbary}
At last, we consider the dual to the adapted barycenter problem. 
To this end, we fix for technical reasons a growth function $k : \Yc \longrightarrow \R$ and consider the minimization problem
\begin{equation}
    \label{eq:bcbary}
    \inf_{\Y \in {\rm FP}(\Yc,k)} \sum_{i = 1}^N \Ac\Wc_{c^i}(\X^i,\Y),
\end{equation}
where ${\rm FP}(\Yc,k)$ denotes the set of all filtered processes $\Y$ with $k(Y) \in \L^1(\P^\Y)$ and, as before, $c^i : \O^i \times \Yc \longrightarrow \R$ is Borel measurable.
Contrary to the causal barycenter problem, it is generally not possible to restrict the infimum in \eqref{eq:bcbary} to minimization over $\{ \Y^\nu \,\vert\, \nu \in \Pc(\Yc),\, k \in \L^1(\nu) \}$. The reason for this is that for the bicausal optimal transport problem, the map $\Y \longmapsto \Y^\nu,$ where $\nu = {\rm Law}_{\P^\Y}(Y)$, does not necessarily decrease the value of $\Ac\Wc_{c^i}(\X^i,\,\cdot\,)$.
Consequently, we have to work with ${\rm FP}(\Yc,k)$, or a suitable representation of it.

\medskip In what follows, we work with the `canonical filtered space' $(\Zc,\Gc_T, (\Gc_t)_{t = 1}^T, \hat Y)$ that was introduced in \citep{BaBePa21}.
The idea behind this space is that every filtered process $\Y \in {\rm FP}(\Yc)$ has a representative on it via the mapping
\[
    \Y \longmapsto (\Zc,\Gc_T, (\Gc_t)_{t = 1}^T, \nu^\Y, \hat Y),
\]
where $\nu^\Y$ denotes the law of ${\rm ip}(\Y)$, the information process of $\Y$. That is to say, the process ${\rm ip}(\Y) = ({\rm ip}_t(\Y))_{t = 1}^T \in \Zc = \prod_{t = 1}^T \Zc_t$ where we set inductively backward in time
\[
    {\rm ip}_T(\Y) \coloneqq \hat Y_T \in \Yc_T \eqqcolon \Zc_T, \quad {\rm ip}_t(\Y) \coloneqq \big(\hat Y_t, {\rm Law}_{\P^\Y}({\rm ip}_{t + 1}(\Y) | \Fc^\Y_t)\big) \in \Yc_t \times \Pc(\Zc_{t + 1}) \eqqcolon \Zc_t,
\]
for $t \in \{1,\dots,T-1\}$.
The remaining terms that make up this canonical filtered space are $(\Gc_t)_{t = 1}^T$, the canonical filtration on $\Zc_{1:T}$ generated by the coordinate projections, and the process $\hat Y = \hat Y_{1:T} : \Zc \longrightarrow \Yc$ where $\hat Y_t$ is the projection onto the $\Yc_t$-coordinate.
It follows from the recursive construction that $\Zc$ is a Polish space.
We remark that the canonical filtered space provides us, in a certain sense, with a minimal representation of filtered processes.
Here, the intuition is that for an element $z_t=(y_t,p_t)  \in \Zc_t$ the first coordinate is the state of the process at time $t$ while the measure $p_t$ describes its `expected behavior after time $t$'.

\medskip Therefore, for $t \in \{2,\dots,T\}$ we denote by $K_t : \Zc_{1:t-1} \longrightarrow \Pc(\Zc_{t:T})$ the measurable kernel defined inductively backward in time by
\[
    K_T(z_{1:T-1};\mathrm d z_T) \coloneqq p_{T-1}(\mathrm d z_T), \quad K_t(z_{1:t-1};\mathrm d z_{t:T}) \coloneqq p_{t-1}(\mathrm dz_t) K_{t+1}(z_{1:t}; \mathrm dz_{t+1:T}),
\]
where $z_{1:t-1} \in \Zc_{1:t-1}$ and $z_{t-1} = (y_{t-1},p_{t-1})$.
Before introducing the set of admissible dual potentials, we want to point out that we shall require $c^i : \O^i \times \Yc \longrightarrow \R$ to be dominated from above by $\ell^i(\o^i) + k(y)$ for suitable $\ell^i$.
In order to have a handle on integrability, which we require in the proof of \Cref{lem:step3.adapt_bary} below, we only take into considerations those $z \in \Zc$ such that $k(y_{1:t},\,\cdot\,)$ is integrable with respect to $K_t(z_{1:t}; \,\cdot\,)$, thus, we set
\[
    \Zc(k) \coloneqq \boxtimes_{t=1}^T \Zc^t_\bullet (k),\enspace \text{where}\enspace \Zc^t_{z_{1:t-1}}(k) \coloneqq \big\{ z_t \in \Zc_t \,\big\vert\, k(y_{1:t},\,\cdot\,) \in \L^1\big(K_{t+1}(z_{1:t};\,\cdot\,)\big) \big\}\enspace {\rm and}\enspace \Zc^T_{z_{1:T-1}}(k)\coloneqq \Zc_T.
\]
For this reason we also consider only those filtered processes $\Y$ such that $\E_{\P^\Y}\big[|k(Y)|\big]<\infty$ in the primal problem.

\medskip Relative to the family of kernels $(K_t)_{t = 2}^T$ we define the set of martingale compensators for $t \in \{2,\ldots,T\}$:
 \begin{align*}   
     \Ac_{i,t}^{\Zc}&\coloneqq \bigg\{ f^i_t(\o^i_{1:t},z_{1:t-1})=a^i_t(\o^i_{1:t},z_{1:t-1})-\int_{\O^i_t} a^i_t(\o^i_{1:t-1},\tilde{\o}^i_t,z_{1:t-1}) K^i_t(\o^i_{1:t-1};\mathrm{d}\tilde{\o}^i_t) \,\bigg\vert\, \\
    &\hspace{4cm}a^i_t \text{ is Borel measurable and } a^i_t(\o^i_{1:t-1},\,\cdot\,,z_{1:t-1}) \in  \L^1\big(\Bc(\O^i_t),K^i_t\big(\o^i_{1:t-1};\,\cdot\,)\big) \bigg\}, \\
    \Bc_{i,t}^\Zc &\coloneqq
    \bigg\{ 
        g_t^i(\o^i_{1:t-1},z_{1:t}) =
        a_t^i(\o^i_{1:t-1},z_{1:t}) - \int_{\Zc_t} a_t^i(\o^i_{1:t-1},z_{1:t-1},\tilde z_t) K_t(z_{1:t-1};\mathrm d\tilde z_t) \,\bigg\vert\, \\
        &\hspace{4cm}a_t^i \text{ is Borel measurable and }a_t^i(\o^i_{1:t-1},z_{1:t-1},\,\cdot\,) \in \L^1\big(\Bc(\Zc_t),K_t(z_{1:t-1};\,\cdot\,)\big)
    \bigg\}.
\end{align*}
and further, for $t = 1$, we let $\Bc_{i,1}^\Zc \coloneqq \Bc_1^\Zc \coloneqq \{ g_1 : \Zc_1 \longrightarrow \R \cup \{ - \infty \} \,\vert\, g_1 \text{ is Borel measurable} \}$.
The set of admissible dual potentials is then given by
\begin{multline} \label{eq:def.bcbary.potentials}
    \Phi^{\Zc}(c^{1:N},k) \coloneqq
    \bigg\{
        f^{1:N}_1 \in \prod_{i = 1}^N \L^1(\Fc_1^i,\P^i) \,\bigg\vert\, \exists f_{2:T}^{1:N} \in \Ac_{1:N,2:T}^\Zc, \exists  g_{1:T}^{1:N} \in \Bc_{1:N,1:T}^\Zc, \forall (\o^{1:N},z) \in \O^{1:N} \times \Zc(k): \\
        f_1^i(\o^i_1) + g_1^i(z_1) + \sum_{t = 2}^T \big[f_t^i(\o^i_{1:t},z_{1:t-1}) + g_t^i(\o^i_{1:t-1},z_{1:t})\big] \le c^i(\o^i,y) \text{ and }\sum_{i = 1}^N g_1^i(z_1) = 0
    \bigg\}.
\end{multline}

For $\nu \in \Pc(\Zc_1)$, we write $\Y^\nu \coloneqq (\Zc, \Gc_T, (\Gc_t)_{t = 1}^T, \nu \otimes K_2 , \hat Y)$ and we further define the set of all canonical filtered processes on $\Yc$ whose law finitely integrates $k$ by 
\[ 
    {\rm CFP}(\Yc,k)\coloneqq \bigg\{ \Y^\nu \coloneqq (\Zc, \Gc_T, (\Gc_t)_{t = 1}^T, \nu \otimes K_2 , \hat Y) \,\bigg\vert\, \nu \in \Pc(\Zc_1), \int_{\Zc} |k(y)| K_2(z_1;\mathrm{d}z_{2:T})\nu(\mathrm{d}z_1) <\infty \bigg\}.
\]

\begin{remark} 
\begin{enumerate}[label = (\roman*)]
\item We remark that when $\Y \in {\rm FP}(\Yc,k)$, that is, $\Y \in {\rm FP}(\Yc)$ with $\E_{\P^\Y}\big[|k(Y)|\big]<\infty,$ the law of the corresponding information process ${\rm ip}(\Y)$ is concentrated on $\Zc(k)$ and the canonical representative of $\Y$ lies in ${\rm CFP}(\Yc,k).$ 

\item We shall write $\Phi^\Zc(c^{1:N}) \coloneqq \Phi^{\Zc}(c^{1:N},1)$ for the set of functions $f^{1:N}_1$ such that \eqref{eq:def.bcbary.potentials} is satisfied everywhere for some $f_{2:T}^{1:N}$ and $g_{1:T}^i.$ This is consistent with the notation above by choosing $k \equiv 1.$ Similarly, the set of all canonical filtered processes is ${\rm CFP}(\Yc) \coloneqq {\rm CFP}(\Yc,1).$

\item  For instance, when $\Xc^i=\Yc=\R^{d \cdot T}$ and $c^i(\o^i,y)=\|\o^i-y\|_2^2,$ then the above is satisfied with $\ell^i=k=2\| \,\cdot\,\|_2^2$ provided that $\ell^i \in \L^1(\P^i).$ In this case, we can restrict the minimization in the primal problem to measures with finite second moments without changing the value of the infimum. 
\end{enumerate}
\end{remark}

\begin{remark}[Pull-back of dual potentials] \label{rem:pull_back}
    When $f_1^{1:N} \in \Phi^\Zc(c^{1:N},k)$ are admissible dual potentials with martingale compensators $(f_{2:T}^{1:N},g_{1:T}^{1:N}) \in \Ac^\Zc_{1:N,2:T} \times \Bc^\Zc_{1:N,1:T}$, then we can pull them back onto the filtered probability space of any process $\Y \in {\rm FP}(\Yc,k)$ in the following way:

    \medskip Since ${\rm ip}(\Y)$ takes $\P^\Y$--almost surely values in $\Zc(k)$, the functions $\bar f_t^i, \bar g_t^i : \O^i \times \O^\Y \longrightarrow \R$ given by
    \[
        \bar f_t^i(\o^i_{1:t},\o^\Y) \coloneqq f_t^i(\o^i_{1:t},{\rm ip}_{1:t-1}(\Y)(\o^\Y)),
        \quad
        \bar g_1^i(\o^\Y) \coloneqq g_1^i({\rm ip}_1(\Y)(\o^\Y)) \text{ and }
        \bar g_t^i(\o^i_{1:t-1},\o^\Y) \coloneqq g_t^i(\o^i_{1:t-1},{\rm ip}_{1:t}(\Y)(\o^\Y)),
    \]
    are well-defined, for $(t,i) \in \{2,\dots,T\} \times \{1,\dots,N\}$. 
    Furthermore, as ${\rm ip}_t(\Y)$ is $\Fc^\Y_t$-measurable, we have that $\bar f_t^i$ is $\Fc_t^i \otimes \Fc_{t-1}^\Y$-measurable while $\bar g_t^i$ is $\Fc^i_{t-1} \otimes \Fc_t^\Y$-measurable.
    The pulled-back potentials are admissible in the sense that
    \[
        f_1^i(\o^i_1) + \bar g_1^i(\o^\Y) + \sum_{t = 2}^T \bar f_t^i(\o^i_{1:t},\o^\Y) + \bar g_t^i(\o^i_{1:t-1},\o^\Y) \le c^i(\o^i,Y(\o^\Y)) \text{ and } \sum_{i = 1}^N \bar g_1^i(\o^\Y) = 0 \quad \cplc(\X^i,\Y)\text{--quasi-surely}.
    \]
\end{remark}

First, we give a duality result for compact spaces and lower-semicontinuous cost functions, which was proved in \cite{AcKrPa23a}.

\begin{proposition}
    Assume that $\Yc$ is compact and that $c^i$, $i \in \{1,\ldots,N\},$ are lower-semicontinuous and lower-bounded.
    Then we have
    \begin{align*}
        \inf_{\Y \in {\rm FP}(\Yc)} \sum_{i = 1}^N \Ac\Wc_{c^i}(\X^i,\Y) 
        = \sup_{f_1^{1:N} \in \Phi^\Zc(c^{1:N})} \sum_{i = 1}^N \int_{\O^i} f_1^i(\o^i_1) \P^i(\mathrm{d}\o^i_1).
    \end{align*}
\end{proposition}

\begin{proof}
The result was proved in \cite{AcKrPa23a} with continuous potentials, while duality with measurable potentials can be shown similarly as in \Cref{prop:lsc_bary}.
\end{proof}

\begin{theorem}[Duality for adapted barycenters] \label{thm:adapt_bary}
    Let $\Yc$ be $\sigma$-compact.
    For every $i \in \{1,\ldots,N\},$ let $c^i : \O^i \times \Yc \longrightarrow \R$ be measurable and lower-bounded, $k : \Yc \longrightarrow \R$ be bounded on compacts, and $\ell^i \in \L^1(\P^i)$ be such that $c^i(\o^i,y) \leq  \ell^i(\o^i)+k^i(y),$ $(\o^i,y) \in \O^i \times \Yc$.
    Then, we have
    \[ 
        \inf_{\Y \in {\rm FP}(\Yc,k)} \sum_{i=1}^N \Ac\Wc_{c^i}(\X^i,\Y)
        =\sup_{f_1^{1:N} \in \Phi^\Zc(c^{1:N},k)} \sum_{i=1}^N \int_{\O^i_1} f^i_1(\o^i_1) \P^i(\mathrm{d}\o^i_1). 
    \]
    Moreover, there is $f^{1:N}_1 \in \Phi^\Zc(c^{1:N},k)$ such that 
    \[
        \inf_{\Y \in {\rm FP}(\Yc,k)} \sum_{i=1}^N \Ac\Wc_{c^i}(\X^i,\Y)=\sum_{i=1}^N \int_{\O^i_1} f^i_1(\o^i_1) \P^i(\mathrm{d}\o^i_1). 
    \]
\end{theorem}

\begin{proof} 
 We use the very same notation as in the proof of \Cref{thm:caus_dual} with the obvious modifications.
 
\medskip \underline{Part A} (compact $\Yc$):
    We first assume that $\Yc$ is compact. Note that Step 1 and Step 2 follow from the same arguments as in the proof of \Cref{thm:caus_dual}.  Moreover, we have that $\Zc(k) = \Zc$ and ${\rm CFP}(\Yc,k)={\rm CFP}(\Yc)$ since $k$ is bounded on compacts.

    \medskip \emph{Step 3} (Continuity from above of D and attainment): In this step we show that $D$ is jointly continuous from above on $\Mc_{1:N}^-$ and that $D(c)$ is attained.
    To that end, for every $i \in \{1,\ldots,N\}$ let $(c^{i,n})_{n \in \N}$ be a sequence in $\Mc_i^-$ such that $c^{i,n} \searrow c^i$ for some $c^i \in \Mc_i^-.$ Note that due to the assumptions, we have $\inf_{n \in \N} D(c^{1,n},\ldots,c^{N,n}) \in \R$.
    Again, we denote
    \[
        S^n \coloneqq D(c^{1,n},\dots,c^{N,n}) = \sup_{f_1^{1:N} \in \Phi^\Zc(c^{1:N,n},k)} \sum_{i = 1}^N \int_{\O^i_1} f_1^i(\o^i_1) \P^i(\mathrm{d}\o^i_1), 
        \quad \text{and} \quad
        S \coloneqq \inf_{n \in \N} S^n.
    \]
    For $n \in \N$, let $f_1^{1:N,n} \in \Phi^\Zc(c^{1:N,n},k)$ be a $1/n$-minimizer for $D(c^{1:N,n}).$ 
    By \Cref{lem:step3.adapt_bary} there is $f_1^{1:N} \in \Phi^\Zc(c^{1:N},k)$ with
    \[
        D(c^{1:N}) \ge \sum_{i = 1}^N \int_{\O^i} f_1^i(\o^i_1) \P^i(\mathrm d\o_1^i) \ge S \ge D(c^{1:N}),
    \]
    which shows that $f_1^{1:N}$ is a dual optimizer as well as continuity from above on $\Mc_{1:N}^-$.

    \medskip \emph{Step 4} (Choquet): Finally, the last step follows again from the same arguments as in the proof of Step 3.4 of \Cref{thm:attain_bary}.

    \medskip \underline{Part B} ($\sigma$-compact $\Yc$):   
    It remains to verify the result for a general $\sigma$-compact path space $\Yc.$ 
    Let us without loss of generality assume that $c^i \geq 0$ and, for $n \in \N$ and $t \in \{1,\ldots,T\}$, let $\Yc^n_t \subseteq \Yc_t$ be compact with $\Yc^n_t \nearrow \Yc_t.$ Define $\xi_t^n : \Yc_{t} \longrightarrow \Yc^n_{t}$ by 
    \[
        \xi^n_t (y_{t})= 
        \begin{cases} 
            y_{t} &y_{t}  \in \Yc^n_{t},\\y^0_{t} &\text{otherwise,} 
        \end{cases}
    \] 
    where $y_0 \in \Yc^1$ is fixed.
    We set $\xi^n\coloneqq\xi^n_{1:T}:\Yc \longrightarrow \Yc^n.$
    We further note that the space $\Yc^n\coloneqq \Yc^n_{1:T}$ is compact, and, by \cite[Theorem 5.1]{BaBePa21}, so is the space $\Zc^n \subseteq \Zc,$ which corresponds to the canonical filtered path space for processes with paths in $\Yc^n.$ 
    We also have that for every $\xi^n_t,$ there exists a measurable map $\zeta^n_t: \Zc_{t} \longrightarrow \Zc^n_{t}$ such that the map $\zeta^n\coloneqq \zeta^n_{1:T}$ is uniquely determined by the following property: if $\Y^\nu= (\Zc, \Gc_T, (\Gc_t)_{t = 1}^T, \nu \otimes K_2 , \hat Y) \in {\rm FP}(\Yc),$ then 
    \[
        (\Zc, \Gc_T, (\Gc_t)_{t = 1}^T, \nu \otimes K_2 , \xi^n(\hat Y)) \sim (\Zc^n, \Gc_T^n, (\Gc_t^n)_{t = 1}^T, \zeta^n_{\#}(\nu\otimes K_2), \hat Y^n) \in {\rm CFP}(\Yc^n,1),
    \]
    where $(\Gc^n_t)_{t=1}^T$ is the canonical filtration on $\Zc^n.$
    Finally, we define $c^{i,n}: \O^i \times \Yc \longrightarrow \R$ by $c^{i,n}(\o^i,y)\coloneqq c^{i}(\o^i,\xi^n(y))$.
    It is clear that we have the pointwise convergence $ \lim_{n} c^{i,n} = c^i.$

\medskip    Clearly, we have that
    \begin{align*}
        \inf_{\Y \in {\rm CFP}(\Yc^n,1)} \sum_{i=1}^N \Ac\Wc_{c^{i}\vert_{\Yc^n}}(\X^i,\Y) =
        \inf_{\Y \in {\rm CFP}(\Yc,k)} \sum_{i=1}^N \Ac\Wc_{c^{i,n}}(\X^i,\Y)
        \ge \inf_{\Y \in {\rm CFP}(\Yc,k)} \sum_{i=1}^N \Ac\Wc_{c^{i}}(\X^i,\Y), 
    \end{align*} 
    where we denote by $c^{i}\vert_{\Yc^n}$ the restriction of $c^i$ to $\Xc^i \times \Yc^n.$  Similarly as in \Cref{prop:lsc_bary}, we can verify the weak duality
    \[ 
        \sup_{f_1^{1:N} \in \Phi^\Zc(c^{1:N},k)} \sum_{i=1}^N \int_{\O^i_1} f^i_1(\o^i_1) \P^i(\mathrm{d}\o^i_1) 
        \leq \inf_{\Y \in {\rm CFP}(\Yc,k)} \sum_{i=1}^N \Ac\Wc_{c^i}(\X^i,\Y).
    \]

    It remains to show equality and attainment.    
    Let $f_1^{1:N,n} \in \Phi^{\Zc^n}(c_{|\Yc^n}^{1:N},k_{|\Yc^n})=\Phi^{\Zc^n}(c_{|\Yc^n}^{1:N},1)$ be optimal for the problem 
    \[
        \sup_{f_1^{1:N} \in \Phi^{\Zc^n}(c^{1:N}_{\vert {\Yc^n}},k_{|\Yc^n})} \sum_{i=1}^N \int_{\O^i_1} f^i_1(\o^i_1) \P^i(\mathrm{d}\o^i_1).
    \]
    We claim that $f_1^{1:N,n} \in \Phi^\Zc(c^{1:N,n},k)$.
    To this end, consider $(\tilde f_{2:T}^{1:N,n},\tilde g_{1:T}^{1:N,n}) \in \Ac_{1:N,2:T}^{\Zc^n}\times  \Bc_{1:N,1:T}^{\Zc^n}$ be such that
    \[ 
        f_1^{i,n}(\o^i_1) + \tilde g_1^{i,n}(z_1^n) + \sum_{t = 2}^T \Big[ \tilde f_t^{i,n}(\o^i_{1:t},z^n_{1:t-1}) + \tilde g_t^{i,n}(\o^i_{1:t-1}, z^n_{1:t}) \Big] \le c^{i}_{\vert{\Yc^n}}(\o^i,y^n) \text{ and }
            \sum_{i = 1}^N g^{i,n}_1(z_1^n) = 0,
    \]
    for all $(\o^{1:N},z^n) \in \O^{1:N} \times \Zc^n$.
    For $(\o^{1:N},z) \in \O^{1:N} \times \Zc$, we set
    \[ 
        f_t^{i,n}(\o^i_{1:t},z_{1:t-1}) \coloneqq \tilde f_t^{i,n}(\o^i_{1:t},\zeta^n (z_{1:t-1})) \text{ and } g_t^{i,n}(\o^i_{1:t},z_{1:t-1}) \coloneqq \tilde g_t^{i,n}(\o^i_{1:t},\zeta^n (z_{1:t-1})),
    \]
    which implies that
    \[ 
        f_1^{i,n}(\o^i_1) + g_1^{i,n}(z_1) + \sum_{t = 2}^T \Big[ f_t^{i,n}(\o^i_{1:t},z_{1:t-1}) +  g_t^{i,n}(\o^i_{1:t-1}, z_{1:t}) \Big] \le c^{i,n}(\o^i,y) 
        \text{ and }
        \sum_{i = 1}^N g^{i,n}_1(z_1) = 0,
    \]
    which shows our claim that $f_1^{1:N,n} \in \Phi^\Zc(c^{1:N,n},k)$.

    \medskip Consequently, we can apply \Cref{lem:step3.adapt_bary} to the sequence $(f_1^{1:N,n})_{n \in \N}$, which provides us with $f_1^{1:N} \in \Phi^\Zc(c^{1:N},k)$ such that
    \begin{multline*}
        \sum_{i = 1}^N \int_{\O^i_1} f_1^i(\o_1^i) \P^i(\mathrm d \o_1^i) \ge \lim_{n \to \infty} \sum_{i = 1}^N \int_{\O^i_1} f_1^{i,n}(\o_1^i) \P^i(\mathrm d\o_1^i) \\
        = \lim_{n \rightarrow \infty}\inf_{\Y \in {\rm CFP}(\Yc,k)} \sum_{i=1}^N \Ac\Wc_{c^{i,n}}(\X^i,\Y)
        \ge \inf_{\Y \in {\rm CFP}(\Yc,k)} \sum_{i=1}^N \Ac\Wc_{c^{i}}(\X^i,\Y).
    \end{multline*}
    The converse inequality follows from the weak duality and the proof is concluded.
\end{proof}

\begin{lemma} \label{lem:step3.adapt_bary}
   In the setting of {\rm \Cref{thm:adapt_bary}}, for every $(i,n) \in \{1, \dots, N\} \times \N$, let $c^{i,n} : \O^i \times \Yc \longrightarrow [0,\infty)$ be measurable such that $c^{i,n}(\o^i,y) \le \ell^i(\o^i) + k(y)$, $(\o^i,y) \in \O^i \times \Yc$, and $c^{1:N,n} \longrightarrow c^{1:N}$ holds pointwise.
    Further, for $n \in \N$, let $f_1^{1:N,n} \in \Phi^\Zc(c^{1:N,n},k)$.
    Then, there exists $f_1^{1:N} \in \Phi^\Zc(c^{1:N},k)$ such that
    \begin{equation}
        \label{eq:step3.adapted_bary.2}
        \sum_{i = 1}^N \int_{\O^i_1}f_1^{i}(\o^i_1) \P^i(\mathrm d\o^i) \ge \limsup_{n \to \infty} \sum_{i = 1}^N \int_{\O^i_1} f_1^i(\o^i_1) \P^i(\mathrm d\o^i).
    \end{equation}
\end{lemma}

\begin{proof}
    Without loss of generality, we can assume that
    \[
        \limsup_{n \to \infty} \sum_{i = 1}^N \int_{\O^i_1}f_1^{i,n}(\o^i_1) \P^i(\mathrm d\o^i) =
        \lim_{n \to \infty} \sum_{i = 1}^N \int_{\O^i_1}f_1^{i,n}(\o^i_1) \P^i(\mathrm d\o^i) \eqqcolon S.
    \]
    By definition of $\Phi^\Zc(c^{1:N,n},k)$, there are functions  $f_{2:T}^{1:N} \in \Ac_{1:N,2:T}^\Zc$ and   $g_{1:T}^{1:N} \in \Bc_{1:N,1:T}^\Zc$ such that for $(\o^i,z) \in \O^i \times \Zc(k)$
    \begin{equation}
        \label{eq:AWbary.1}
        f_1^{i,n}(\o^i_1) + g_1^{i,n}(z_1) + \sum_{t = 2}^T \Big[ f_t^{i,n}(\o^i_{1:t},z_{1:t-1}) + g_t^{i,n}(\o^i_{1:t-1}, z_{1:t}) \Big] \le c^{i,n}(\o^i,y) \text{ and }
        \sum_{i = 1}^N g^{i,n}_1(z_1) = 0.
    \end{equation}
    By shifting $f^{i,n}_1$ and $g^{i,n}_1$ by constants, we can assume that $\int_{\O^i_1} f^{i,n}_1 \mathrm{d}\P^i = \int_{\O^j_1} f^{j,n}_1 \mathrm{d}\P^j$ for all $(i,j,n) \in \{1,\dots,N\}^2 \times \N$.

    \medskip Analogous to the same step in the proof of \Cref{lem:step3.attain_bary},
    we have that the sequence $(f_1^{i,n})_{n \in \N}$ is bounded in $\L^1(\Fc_1^i,\P^i)$. Thus, we can assume without loss of generality, by passing to the Ces\`aro means of a suitable subsequence, that for every $i \in \{1,\dots,N\}$
    $f_1^{i,n} \longrightarrow f_1^i$ $\P^i$--a.s.\ and $f_1^i \in \L^1(\Fc_1^i, \P^i)$ satisfy \eqref{eq:step3.adapted_bary.2}.
    It remains to show admissibility, that is, to show $f_1^{1:N} \in \Phi^\Zc (c^{1:N},k)$.

    \medskip \emph{Step 1} (Construction of $g_1^{1:N}$):
    To find suitable $g_1^{1:N} \in \Bc^\Zc_{1:N,1}$ with $\sum_{i = 1}^N g_1^i = 0$, we integrate the inequality on the left-hand side in \eqref{eq:AWbary.1} with respect to $\P^i$ and $K_2$, and derive the bound
    \[
        g_1^{1,n}(z_1) \leq \int_{\O^i} \ell^i(\o^i) \P^i(\mathrm d\o^i) + \int_{\Zc} k(y) K_2(z_1;\mathrm dz_{2:T}) -\inf_{k \in \N} \int_{\O^i_1} f_1^{1,k}(\o^i_1) \P^i(\mathrm d\o^i_1).
    \] 

    Note that the right-hand side is well defined and finite for every $z_1 \in \Zc_1(k)$.
    By \Cref{komlos}, the sequence $(g^{1:N,n}_1)_{n \in \N}$ admits a subsequence of forward convex combinations that depend measurably on $z_1$ and converge pointwise for every $z_1 \in \Zc_1(k).$ Taking convex combinations of all the other dual potentials doesn't interfere with their respective properties.

    \medskip \emph{Step 2} (Construction of martingale compensators):
    Observe that the only constraint that is coupling, for $i \in \{1,\ldots,N\}$, the inequalities in \eqref{eq:AWbary.1}, is the congruency condition $\sum_{i=1}^N g_1^{i,n}(z_1)=0.$
    In the previous step we have constructed suitable sequences such that $g_1^{1:N,n} \longrightarrow g_1^{1:N}$ pointwise on $\Zc_1(k)$ and the limits satisfy the congruency condition.
    Hence, for the rest of the proof these inequalities completely decouple, which allows us to use the same construction as in in Step 3.2 and Step 3.3 of \Cref{existDualBC} and thereby find suitable martingale compensators $f^{1:N}_{2:T}$ and $g^{1:N}_{2:T}$.

    \medskip We conclude that $f_1^{1:N} \in \Phi^\Zc(c^{1:N},k)$, which completes the proof 
\end{proof}

\appendix
\section{Appendix} \label{sec:appendix}

\begin{lemma} \label{komlos} 
Let $(\Xc,\Fc^\Yc)$, $(\Yc,\Fc^\Yc)$ be standard Borel spaces.
Let $Y^n :\Xc \times \Yc \longrightarrow \R$, $n \in \N$, be measurable. Let further $\P : \Xc \longrightarrow \Pc(\Yc)$ and $C : \Xc \times \Yc \longrightarrow \R$ be measurable and such that $\inf_{n \in \N} Y^n(x,\,\cdot\,) \geq C(x,\,\cdot\,)$ $\P(x)$--a.s.\ for every $x \in \Xc.$
Then, there are measurable maps $\tY^n : \Xc \times \Yc \longrightarrow \R$, $n \in \N$, and $\tY : \Xc \times \Yc \longrightarrow \R \cup \{+ \infty \}$ satisfying: 
\begin{enumerate}[label = (\roman*)]
\item \label{f}
    there are measurable functions $\lambda_k^n : \Xc \longrightarrow [0,1]$, $(k,n) \in \N^2$, with $\sum_{k \in \N} \lambda_k^n = 1$ such that 
    \[
        \tY^n(x,\,\cdot\,) = \sum_{k \in \N} \lambda_k^n (x) Y^k(x,\,\cdot\,) \in {\rm conv} \left( Y^n(x,\,\cdot\,), Y^{n+1}(x,\,\cdot\,),\ldots \right), \quad x \in \Xc,
    \] 
    where, for random variables $Z^1,Z^2,\ldots$, we denote \[{\rm conv} ( Z^1,Z^2,\ldots )\coloneqq \Big\{ \sum_{k=1}^N \lambda_k Z^k \,\Big\vert\, N \in \N,\; \lambda_k \in [0,1],\;\sum_{k=1}^N \lambda_k = 1 \Big\};\]
\item \label{s}
    for every $x \in \Xc$, we have $\displaystyle \lim_{n \to \infty} \tY^n(x,y) = \tY(x,y)\;\text{ for }\P(x)$--almost every $y \in \Yc.$
\end{enumerate}
\end{lemma}
\begin{proof} The proof follows similar steps as \citeauthor*{DeSc06} \cite[Lemma 9.8.1]{DeSc06}. We have to verify that each step preserves measurability of the involved functions. For the most part of the proof, we will suppress dependencies on $y \in \Yc$. Since $(\Xc,\Fc^\Yc)$, $(\Yc,\Fc^\Yc)$ are standard Borel spaces, there exists a Polish topologies generating $\Fc^\Yc$ and $\Fc^\Yc,$ respectively. We can hence assume without loss of generality that $\Xc$ and $\Yc$ are Polish.
We can further without loss of generality assume that $C=0$, otherwise we work with $Y^{\prime,n}\coloneqq Y^n-C$ instead of $Y^n$. 
For $(x,n) \in \Xc \times \N$, we define
\[ I^n(x)\coloneqq\inf \left\lbrace \E_{\P(x)} \left[ \exp(-Y) \right] \,\vert\, Y \in \text{conv} \left( Y^n(x), Y^{n+1}(x),\ldots \right)\right\rbrace.  \]
Note that $0\leq I^n(x)\leq 1$ since $Y^n(x,\,\cdot\,) \geq 0$ $\P(x)$--a.s.\ for every $x \in \Xc.$ We claim that for every $n \in \N$, there are measurable weights $\lambda_k^n : \Xc \longrightarrow [0,1]$, $k \in \N$, such that for $x\in \Xc$, $\lambda_n^k(x) \neq 0$ for finitely many $k$, $\sum_{k \in \N} \lambda^n_k = 1$ and $\Yb^n(x) \coloneqq \sum_{k \in \N} \lambda_k^n(x) Y^k(x)$ satisfies
\[
    \E_{\P(x)} \left[ \exp\left(-\Yb^n(x)\right) \right] \leq I^n(x) + \frac{1}{n}.
\] 
To that end, we endow the space of null sequences $c_0\coloneqq\lbrace \lambda=(\lambda_0,\lambda_1,\ldots) \in \R^\N : \lim_{i \rightarrow \infty} \lambda_i=0 \rbrace$ with the supremum norm $\lVert\lambda \rVert_\infty=\sup_{i \in \N \cup \{0\}} \lvert \lambda_i \rvert$.
It is well-known that $(c_0,\|\cdot \|_\infty)$ is complete and separable. 
Further, we consider the subspace of sequences that are eventually constant zero
\[ 
    c_{00}^1\coloneqq\bigcup_{\ell=0}^\infty \Sc_\ell,
\] 
where 
\[ 
    \Sc_\ell=\bigg\lbrace \lambda=(\lambda_0,\lambda_1,\ldots) \in c_0 \,\bigg\vert\, \lambda_i \in [0,1],\; i \in \N \cup \{0 \},\; \lambda_i=0 \text{ for } i>\ell,\;\sum_{i=0}^\ell \lambda_i=1 \bigg\rbrace.  
\] 
It is clear that, for every $\ell \in \N \cup \{0\}$, $\Sc_\ell$ is a compact subspace of $c_0$, hence $c_{00}^1$ is measurable and $\sigma$-compact.
For every $n \in \N$, we let $f^n: \Xc \times c_{00}^1 \longrightarrow [0,1]$ be given by
\[ 
    f^n(x,\lambda)\coloneqq\E_{\P(x)} \bigg[ \exp\Big(-\sum_{i=0}^\infty \lambda_i Y^{n+i}(x)\Big) \bigg]. 
\] 
Since $\lambda_i \neq 0$ for only finitely many indices, the sum $\sum_{i=0}^\infty \lambda_i Y^{n+i}(x)$ is always well-defined and finite.
Observe that $f^n$ is measurable in $x$ and continuous in $\lambda$. Clearly, for $x \in \Xc$
\[ 
    I^n(x)=\inf \left\lbrace f^n(x,\lambda) \,\big\vert\,\lambda \in c_{00}^1  \right\rbrace.
\]
Hence, we find a measurable selection by invoking the Arsenin--Kunugui theorem, see \emph{e.g.}\ \cite[Theorem 18.18]{Ke95},
to conclude that, for every $n \in \N$, there exists a measurable $(1/n)$-minimizer $\tilde{\lambda}^n : \Xc \longrightarrow c_{00}^1$ such that 
\[ 
    \E_{\P(x)} \bigg[ \exp\Big(-\sum_{i=0}^\infty \tilde{\lambda}_i^n(x) Y^{n+i}(x)\Big) \bigg]\leq I^n(x)+\frac{1}{n}.
\]
For fixed $x \in \Xc$, we can show in the very same way as in \cite{Komlos}, that the sequence \[\exp\bigg(-\sum_{i=0}^\infty \tilde{\lambda}_i^n(x) Y^{n+i}(x)\bigg),\; n \in \N, \] is Cauchy in $\L^1(\Bc(\Yc),\P(x))$ and hence admits a limit there. We now prove $\ref{f}$ and $\ref{s}$ to show that the limit is indeed measurable. 
Successively, we define
\begin{align*}
n_1(x)&\coloneqq\text{argmin}
\bigg\lbrace n \in \N \, \bigg\vert\, \forall \ell \geq n : 
\E_{\P(x)}\bigg[ \Big\lvert \exp\Big(-\sum_{i=0}^\infty \tilde{\lambda}_i^{\ell}(x) Y^{\ell+i}(x)\Big)-\exp\Big(-\sum_{i=0}^\infty \tilde{\lambda}_i^{n}(x) Y^{n+i}(x)\Big) \Big\rvert \bigg] \leq 2^{-2} \bigg\rbrace,
\end{align*}
and, for $k \in \N$, $k \ge 2$,
\begin{align*}
n_{k}(x)&\coloneqq\text{argmin} \bigg\lbrace n \geq (n_{k-1}(x) \wedge k) \, \Big\vert \, \forall \ell \geq n : \\ & \qquad \qquad \E_{\P(x)}\bigg[ \Big\lvert \exp\Big(-\sum_{i=0}^\infty \tilde{\lambda}_i^{\ell}(x) Y^{\ell+i}(x)\Big)-\exp\Big(-\sum_{i=0}^\infty \tilde{\lambda}_i^{n}(x) Y^{n+i}(x)\Big) \Big\rvert\bigg] \leq 2^{-2k} \bigg\rbrace.
\end{align*}
It is evident from the construction that, for every $k \in \N$, $x \longmapsto n_k(x)$ is measurable, and consequently, 
\[ 
    (x,y) \longmapsto \exp\Big(-\sum_{i=0}^\infty \tilde{\lambda}_i^{n_k(x)}(x) Y^{n_k(x)+i}(x,y)\Big) 
\] 
is also measurable. 
Using the Borel-Cantelli lemma we conclude that for every $x \in \Xc$:
\[ \P(x)\bigg[ \limsup \bigg\lbrace \Big\lvert \exp\Big(-\sum_{i=0}^\infty \tilde{\lambda}_i^{n_k(x)}(x) Y^{n_k(x)+i}(x)\Big)-\exp\Big(-\sum_{i=0}^\infty \tilde{\lambda}_i^{n_{k+1}(x)}(x) Y^{n_{k+1}(x)+i}(x)\Big) \Big\rvert > 2^{-k}  \bigg\rbrace \bigg]=0. \] 
As a consequence we have
\[ \P(x)\bigg[ \bigg\lbrace \exists N \in \N \,\forall k \geq N :  \Big\lvert \exp\Big(-\sum_{i=0}^\infty \tilde{\lambda}_i^{n_k(x)}(x) Y^{n_k(x)+i}(x)\Big)-\exp\Big(-\sum_{i=0}^\infty \tilde{\lambda}_i^{n_{k+1}(x)}(x) Y^{n_{k+1}(x)+i}(x)\Big) \Big\rvert < 2^{-k}  \bigg\rbrace \bigg]=1. \] 
This means that, for every $x \in \Xc$, the sequence
\[ 
    \exp\Big(-\sum_{i=0}^\infty \tilde{\lambda}_i^{n_k(x)}(x) Y^{n_k(x)+i}(x)\Big),\quad k \in \N, 
\]
admits a $\P(x)$--almost-sure limit in $[0,1]$, whence,
\[
    \tY^k(x)\coloneqq\sum_{i=0}^\infty \tilde{\lambda}_i^{n_k(x)}(x) Y^{n_k(x)+i}(x),\quad x \in \Xc, 
\] 
has a limit in $[0,\infty]$.
Recall that by assumption $(Y^n)_{n \in \N}$ is a sequence of jointly measurable functions on $\Xc \times \Yc$, from where we deduce that $\tY^k : \Xc \times \Yc \longrightarrow [0,\infty)$ is also jointly measurable.
Thus, $\tY \coloneqq\limsup_{k\rightarrow \infty} \tY^k$ is jointly measurable and we have, for each $x \in \Xc$,
\[
    \tY^k(x)=\sum_{i=0}^\infty \lambda^k_i(x) Y^{k+i}(x),
\] 
where the weights $\lambda^k$ are measurable and satisfy
\begin{align*}
\lambda^k_i(x)=\begin{cases}
\tilde{\lambda}^{n_k(x)}_{i+k-n_k(x)}(x),\; \text{if } i +k \geq n_k(x), \\
0,\; \text{otherwise.}
\end{cases}
\end{align*}
This concludes the proof.
\end{proof}

\begin{lemma} \label{qs} 
    Let $\Xc$, $\Yc$ be Polish spaces and let $\Q \in \Pc(\Xc \times \Yc)$ disintegrate as $\Q(\mathrm{d}x,\mathrm{d}y) = \P_1(\mathrm{d}x) \otimes \P(x)(\mathrm{d}y)$ for some $\P_1 \in \Pc(\Xc)$ and measurable $\P: \Xc \longrightarrow \Pc(\Yc)$.
    Let $Y^k, Y : \Xc \times \Yc \longrightarrow \R$, $k \in \N$, be measurable functions such that for $\P_1$--{\rm a.e.}\ $x$ holds
    \[
        \lim_{k \to \infty} Y^k(x,\,\cdot\,) = Y(x,\,\cdot\,) \quad \P(x)\text{{\rm--a.s.}}
    \]
    Then $\{ (x,y) \in \Xc \times \Yc :  \limsup_{k \to \infty} Y^k(x,y) \neq Y(x,y) \text{ or }  \liminf_{k \to \infty} Y^k(x,y) \neq Y(x,y) \}$ is a $\Q$-null set.
\end{lemma}
    
\begin{proof}
    Since $\{ (x,y) \in \Xc \times \Yc :  \limsup_{k \to \infty} Y^k(x,y) \neq Y(x,y) \text{ or }  \liminf_{k \to \infty} Y^k(x,y) \neq Y(x,y) \}$ is measurable, the claim follows directly from Fubini's theorem.
\end{proof} 

We prove a multidimensional version of the Choquet capacitability theorem \cite[Proposition 2.1]{BaChKu19} and refer for further details to \cite[Section 2]{BaChKu19}. 
To this end, let us consider the following setting. For every $i \in \{1,\ldots,N\},$ let $H^i \subseteq G^i$ be two sets of functions from $\Xc^i$ to $[-\infty,\infty].$ Assume that $H^i$ is a lattice and $G^i$ contains all suprema of increasing sequences in $G^i$ as well as infima of all sequences in $G^i.$ Denote the set of all infima of sequences in $H^i$ by $H^i_\delta.$ Let $\Phi : \prod_{i=1}^N G^i \longrightarrow [-\infty, \infty]$ be a mapping which is increasing in all entries and let us extend $\Phi$ for arbitrary functions $Y^i: \Xc^i\longrightarrow [-\infty,\infty],$ $i \in \{1,\ldots,N\},$ by \[ \hat{\Phi}(Y^1,\ldots,Y^N)\coloneqq\inf \{ \Phi(X^1,\ldots,X^N) \,\vert\, X^i \leq Y^i,\; Y^i \in  G^i \}.\]
We have the following result.
\begin{lemma} \label{lem:choquet} Let $\Phi: \prod_{i=1}^N G^i \longrightarrow \R$ be increasing in all entries and assume that
\begin{enumerate}[label = (\roman*)]
\item $\lim_{n \rightarrow \infty} \Phi(X^1_n,\ldots,X^N_n)=\Phi(\lim_{n \rightarrow \infty}X^1_n,\ldots,\lim_{n \rightarrow \infty}X^N_n)$ for any $N$-tuple of decreasing sequences $(X_n^i)_{n \in \N}$ in $H^i,$ $i \in \{1,\ldots,N\};$
\item $\lim_{n \rightarrow \infty} \Phi(X^1_n,\ldots,X^N_n)=\Phi(\lim_{n \rightarrow \infty}X^1_n,\ldots,\lim_{n \rightarrow \infty}X^N_n)$ for any $N$-tuple of increasing sequences $(X_n^i)_{n \in \N}$ in $G^i,$ $i \in \{1,\ldots,N\}.$
\end{enumerate}
Then $\hat{\Phi}(X^1,\ldots,X^N)=\sup \{ \Phi(Y^1,\ldots,Y^N) \,\vert\, Y^i \leq X^i, \; Y^i \in H^i_{\delta} \}$ for any $H^i$-Suslin functions $X^i,$ $i \in \{1,\ldots,N\}.$
\end{lemma}
\begin{proof} The statement is a direct corollary of \cite[Proposition 2.1]{BaChKu19}. Let us define sets of functions from $\{1,\ldots,N\} \times (\prod_{i=1}^N\Xc^i)$ to $[-\infty,\infty]$ by
\begin{equation*}
    H \coloneqq \{ X(i,x^1,\ldots,x^N) = X^i(x^i) \,\vert\, X^i \in H^i \}, \quad {\rm and}\quad  G \coloneqq \{ Y(i,x^1,\ldots,x^N) = Y^i(x^i) \,\vert\, Y^i \in G^i \}.
\end{equation*}      
It is then easy to verify that $H \subseteq G,$ $H$ is a lattice and $G$ contains all suprema of increasing sequences in $G$ as well as infima of all sequences in $G.$ Moreover, let us define $\Psi : G \longrightarrow [-\infty, \infty]$ and $\hat{\Psi}: H \longrightarrow [-\infty, \infty]$ by \[ \Psi(X)\coloneqq \Phi(X(1,\cdot),\ldots,X(N,\cdot)),\quad{\rm resp.}\quad \hat{\Psi}(Y)\coloneqq\inf \{ \Psi(X) : X \leq Y,\; Y \in  G \}.\] It is then clear that
\begin{multline*}\hat{\Psi}(Y)=\inf \{ \Phi(X(1,\cdot),\ldots,X(N,\cdot)) \,\vert\, X(i,\cdot) \leq Y(i,\cdot),\;i \in \{1,\ldots,N\},\; X \in  G \}\\ =\inf \{ \Phi(X^1,\ldots,X^N) \,\vert\, X^i \leq Y(i,\cdot),\;i \in \{1,\ldots,N\},\; X^i \in  G^i \}=\hat{\Phi}(Y(1,\cdot),\ldots,Y(N,\cdot)).
\end{multline*} Applying the Choquet capacitability theorem \cite[Proposition 2.1]{BaChKu19} to the functional $\Psi,$ we conclude
\[ \hat{\Psi}(X)=\sup \{ \Phi(Y) \,\vert\, Y \leq X, \; Y \in H_{\delta} \},\quad \text{for any $H$-Suslin function $X$},\] where $H_\delta$ denotes the set of all infima of sequences in $H.$
For every $i \in \{1,\ldots,N\},$ let $X^i$ be $H^i$-Suslin functions and set $X(i,x^1,\ldots,x^N) \coloneqq X^i(x^i)$. Then $X$ is $H$-Suslin and we have
\begin{align*}
 \hat{\Psi}(X)&=\hat{\Phi}(X(1,\cdot),\ldots,X(N,\cdot))=\hat{\Phi}(X^1,\ldots,X^N),\\
 \sup \{ \Phi(Y) \,\vert\, Y \leq X, \; Y \in H_{\delta} \}&=\sup \{ \Phi(Y(1,\cdot),\ldots,Y(N,\cdot)) \,\vert\, Y(i,\cdot) \leq X(i,\cdot), \; Y \in H_{\delta} \}\\
 &=\sup \{ \Phi(Y^1,\ldots,Y^N) \,\vert\, Y^i \leq X^i, \; Y^i \in H^i_{\delta} \},
\end{align*} where the last equality follows from the fact that $Y \in H_\delta$ if and only if $Y(i,\,\cdot\,) \in H^i_\delta,$ $i \in \{1,\ldots,N\}.$ This concludes the proof.
\end{proof}

\bibliography{refs}

\section*{Statements and Declarations}
 Daniel Kr\v{s}ek gratefully acknowledges the support of the SNF project MINT 205121-219818. Other than that, the authors declare that no funds, grants, or other support were received during the preparation of this manuscript and the authors have no relevant financial or non-financial interests to disclose.
\end{document}